%% file: ms.tex
\documentclass{article}
\usepackage{amsmath,amssymb,amsthm}
\usepackage[usenames]{color}
\usepackage{bm}
\usepackage{graphicx}

\newtheorem{theorem}{Theorem}[section]
\newtheorem{remark}[theorem]{Remark}
\newtheorem{lemma}[theorem]{Lemma}
\newtheorem{definition}[theorem]{Definition}
\numberwithin{equation}{section}

\newcommand{\Z}{\mathbb{Z}}
\newcommand{\C}{\mathbb{C}}
\newcommand{\R}{\mathbb{R}}
\newcommand{\F}{\mathbb{F}}
\newcommand{\bydef}{\stackrel{\mbox{\tiny\textnormal{\raisebox{0ex}[0ex][0ex]{def}}}}{=}}
\newcommand{\conv}[1]{\langle #1 \rangle}

\newcommand{\bb}{\bar{b}}
\newcommand{\bu}{\bar{u}}
\newcommand{\bkappa}{\bar{\kappa}}
\newcommand{\bx}{\bar{x}}
\newcommand{\hx}{\hat{x}}
\newcommand{\hu}{\hat{u}}
\newcommand{\hkappa}{\hat{\kappa}}
\newcommand{\hb}{\hat{b}}
\newcommand{\hr}{\hat{r}}
\newcommand{\rstar}{r^*}

\newcommand{\FF}{\mathcal{F}}
\newcommand{\BB}{\mathcal{B}}
\renewcommand{\AA}{\mathcal{A}}

\newcommand{\norm}[2]{|#1|_{#2}}
\newcommand{\knorm}[1][k]{\norm{#1}{\bkappa}}

\newcommand{\kk}[1][k]{\norm{#1}{\bkappa}}
\newcommand{\roundup}{\,\,\uparrow}
\newcommand{\rounddown}{\,\, \downarrow}
\newcommand{\CC}{\mathcal{C}_P}
\newcommand{\EE}{\mathcal{D}_P}
\newcommand{\RJ}{\R^J}
\newcommand{\CJ}{\C^J}
\newcommand{\FJ}{\CJ}
\newcommand{\FFJ}{\F^J}

\newcommand{\PPP}{\mathcal{P}}

\newcommand{\RRR}{\mathcal{R}}
\newcommand{\tRRR}{\widetilde{\RRR}}

\newcommand{\UUU}{\mathcal{U}}
\newcommand{\VVV}{\mathcal{V}}

\newcommand{\hsp}{\hspace{2mm}}
\newcommand{\yP}{y_P}
\newcommand{\J}{\mathbb{I}}
\newcommand{\conj}{\mathbb{I}_*}
\newcommand{\fullconj}{\mathbb{J}_*}
\newcommand{\orbitcount}{\Gamma_{\!\G}}
\newcommand{\EEE}{\mathcal{E}}
\newcommand{\E}{\mathrm{E}}
\newcommand{\W}{W}
\newcommand{\tp}{\widetilde{p}}
\newcommand{\M}{\mathcal{M}}

\newcommand{\re}{\operatorname{Re}}
\newcommand{\tr}{\operatorname{tr}}
\newcommand{\Ssym}{\mathcal{S}}
\newcommand{\BBsym}{\BB^{\Ssym}}
\newcommand{\YY}[1]{Y^{[#1]}}
\newcommand{\WY}[1]{W^{[#1]}}
\newcommand{\ZY}[1]{Z^{[#1]}}

\newcommand{\tc}{\tilde{c}}
\newcommand{\tkappa}{\tilde{\kappa}}
\newcommand{\PHI}[1]{\Psi^{[#1]}}
\newcommand{\II}{\mathcal{I}}
\newcommand{\morseindex}{\iota}
\newcommand{\negnum}{\mathcal{N}}
\newcommand{\THETA}{\Theta_{\hx}}

\renewcommand{\L}{\mathcal{L}}
\newcommand{\LL}{\widetilde{\L}}
\newcommand{\hLLl}[1][l]{\widehat{\L}_{#1}}
\newcommand{\ls}{l} 
\newcommand{\DLk}{\boldsymbol{\Delta}^{\L}_k}
\newcommand{\Dk}[1][k]{\Delta_{#1}}
\newcommand{\Dkj}[2][k]{\Delta_{#1}^{#2}}
\newcommand{\Dkk}{\Dk \kappa}
\newcommand{\PP}{P}

\newcommand{\G}{\mathcal{G}}
\newcommand{\ZZ}{\mathcal{Z}}
\newcommand{\talpha}{\widetilde{\alpha}}
\newcommand{\Zdom}{\ZZ_{\text{\textup{dom}}}}
\newcommand{\Zsym}{\ZZ_{\text{\textup{sym}}}}
\newcommand{\Xsym}{X^{\sym}}
\newcommand{\X}{\mathrm{\mathbf{X}}}
\newcommand{\XK}{\X^K_0}
\newcommand{\Xinfty}{\X^\infty_0}
\newcommand{\tX}{\mathrm{\mathbf{\tilde{X}}}}
\newcommand{\e}{\mathrm{e}}
\newcommand{\sym}{\text{\textup{sym}}}
\newcommand{\Gacts}[1][k]{\G.#1}
\newcommand{\gacts}[1][k]{\ggacts{g}{#1}}
\newcommand{\ggacts}[2]{#1.#2}
\newcommand{\x}{\mathrm{x}}
\newcommand{\tf}{\widetilde{f}}
\newcommand{\tE}{\widetilde{E}}

\newcounter{thmlistcntr}
\newenvironment{thmlist}{\begin{list}{\textup{(\alph{thmlistcntr})}}{\usecounter{thmlistcntr}\setlength{\leftmargin}{5ex}\setlength{\labelwidth}{4ex}\setlength{\parskip}{0ex}\setlength{\itemsep}{0ex}\setlength{\topsep}{0.2ex}\setlength{\parsep}{0.0ex}}}{\end{list}}

\title{Optimal periodic structures with general space group symmetries in the Ohta-Kawasaki problem}

\author{Jan Bouwe van den Berg\thanks{VU Amsterdam, Department of Mathematics,
De Boelelaan 1081, 1081 HV Amsterdam, The Netherlands,  {\tt janbouwe@few.vu.nl};
partially supported by NWO-VICI grant 639033109.}
\and {J.F. Williams\thanks{Simon Fraser University, Department of Mathematics,
8888 University Drive Burnaby, BC, V5A 1S6, Canada,  {\tt jfwillia@sfu.ca}.}}}

\begin{document}
\maketitle

\begin{abstract}
We consider the problem of rigorously computing periodic minimizers to the Ohta-Kawasaki energy. We develop a method to prove existence of solutions and determine rigorous bounds on the distance between our numerical approximations and the true infinite dimensional solution and also on the energy. We use a method with prescribed symmetries to explore the phase space, computing candidate minimizers both with and without experimentally observed symmetries.
We find qualitative differences between the phase diagram of the Ohta-Kawasaki energy and self consistent field theory when well away form the weak segregation limit.
\end{abstract}

\section{Introduction}
\label{s:intro}
\input{intro}

\section{The symmetries}
\label{s:symmetries}
\input{symmetries}

\section{The zero finding problem}
\label{s:zerofinding}
\input{zerofinding}

\section{Group actions in Fourier space}
\label{s:groupactions}
\input{groupactions}

\section{The fixed point problem}
\label{s:fixedpoint}
\input{fixedpoint}

\section{Computation of the energy}
\label{s:energy}
\input{energy}

\section{The Morse index}
\label{s:stability}
\input{stability}

\section{Numerical results}
\label{s:numerics}
\input{numerics}

{\small
\bibliographystyle{plain}
\bibliography{OK-references}
}

\appendix
\section{The estimates}
\label{sa:estimates}
\input{boundP}

\input{boundY}

\input{derivatives}
\input{boundZ}
\input{boundW}

\end{document}

%% file: intro.tex
%!TEX root = ms.tex

The Ohta-Kawasaki energy is an important model in applied mathematics developed initially
as a model for diblock copolymers, and more recently, for self-assembly driven by electrically charged phase separation.
More generally it is a canonical model for energy
driven pattern formation in systems with
Coulomb-like interactions of short-range repulsion and long-range attraction. In this paper, we use symmetry preserving rigorously verified numerical
techniques to investigate the landscape of the Ohta-Kawasaki energy in
three dimensions.\\

\noindent  {\bf The diblock copolymer problem} 
Diblock copolymers are soft materials formed with linear chain molecules consisting of two immiscible covalently bonded subchains, type A and type B, which
can self-assemble into various configurations depending on physical parameters.
The immiscibility of the subchains leads to
short scale repulsion, which competes with long-range attraction driven by the bonding. This interaction creates a rich class of complex structures \cite{BF}. These different geometries can provide materials with distinct mechanical, optical, and magnetic properties.

In material science, the energy landscape of diblock copolymers is primarily
described using two parameters: $\chi N$, the product of the Flory-Huggins interaction parameter
    and   $N$, the index of polymerization; and $f$, the molecular weight which measures the relative length of the
    A-monomer chain compared with the length of the whole macromolecule.
The Flory-Huggins interaction parameter $\chi$ measures the incompatibility of the
two monomers. 

The primary problem is: given a pair $(f, \chi N)$ what configuration has the lowest free energy? Two foundational results in understanding the energy landscape field are due to Khandpur et al \cite{Khand} who computed the first experimental energy landscape and Matsen and
Schick \cite{MATSENSCHICK} who pioneered the Self Consistent Field Theory
(SCFT) computational technique for diblock copolymers. Both results are described in the review article \cite{BatesRev}.

The only discovered structures believed to be global minimizers in large domains are: one-dimensional lamellae, cylinders packed on hexagonal lattices, double gyroids, body-centred cubically packed spheres and close-packed spheres.
All these solutions are periodic with period specified as part of
energy minimization, {\em not externally}. 
Experimentally these states are identified by performing small angle X-ray scattering. Each solution profile has a distinct set of crystallographic symmetries which can be identified from the resulting power spectra.\\

\noindent {\bf The Ohta-Kawasaki energy} 
Ohta and Kawasaki~\cite{OK1} derived a density functional theory approximating the diblock copolymer free energy in terms of the averaged macroscopic monomer density. Defining an indicator function
$u(\x) \in [-1,1]$
with $u \equiv -1$ being all type A and $u \equiv +1$ all type B the energy can be rescaled as
\begin{equation}
\label{eq:OKEnergy}
    \EEE(u) = \frac{1}{|\Omega|} \int_\Omega \frac{1}{2 \gamma^2} |\nabla u |^2 + \frac{1}{4}(1-u^2)^2+ \frac{1}{2} | \nabla v |^2 \, d\x.
\end{equation}
Here $u$ is a periodic function on a  domain
    $\Omega $,
    $m = \frac{1}{|\Omega|} \int_\Omega u(\x) d\x$ is its average and $\gamma$ is
	a parameter related to $\chi N$ \cite{Choksi2003}. 
	Physically the mass fraction $m$ 
satisfies $-1 \le m \le 1$ and $\gamma > 0$.
	The function $v$ is the unique solution of
    the linear elliptic problem $-\Delta v =  u - m$
    with periodic boundary conditions satisfying  $\int v(\x) d\x =0$.

    Critical points of the energy (\ref{eq:OKEnergy}) are most easily found by taking the
    gradient in $H^{-1}$ \cite{CPW}:
    \begin{equation}
    \label{eq:OK}
      - \Delta \left(\frac{1}{\gamma^2} \Delta u + u - u^3\right) - (u-m) =0.
    \end{equation}

\begin{figure}
      \centering
            \includegraphics[width=\textwidth]{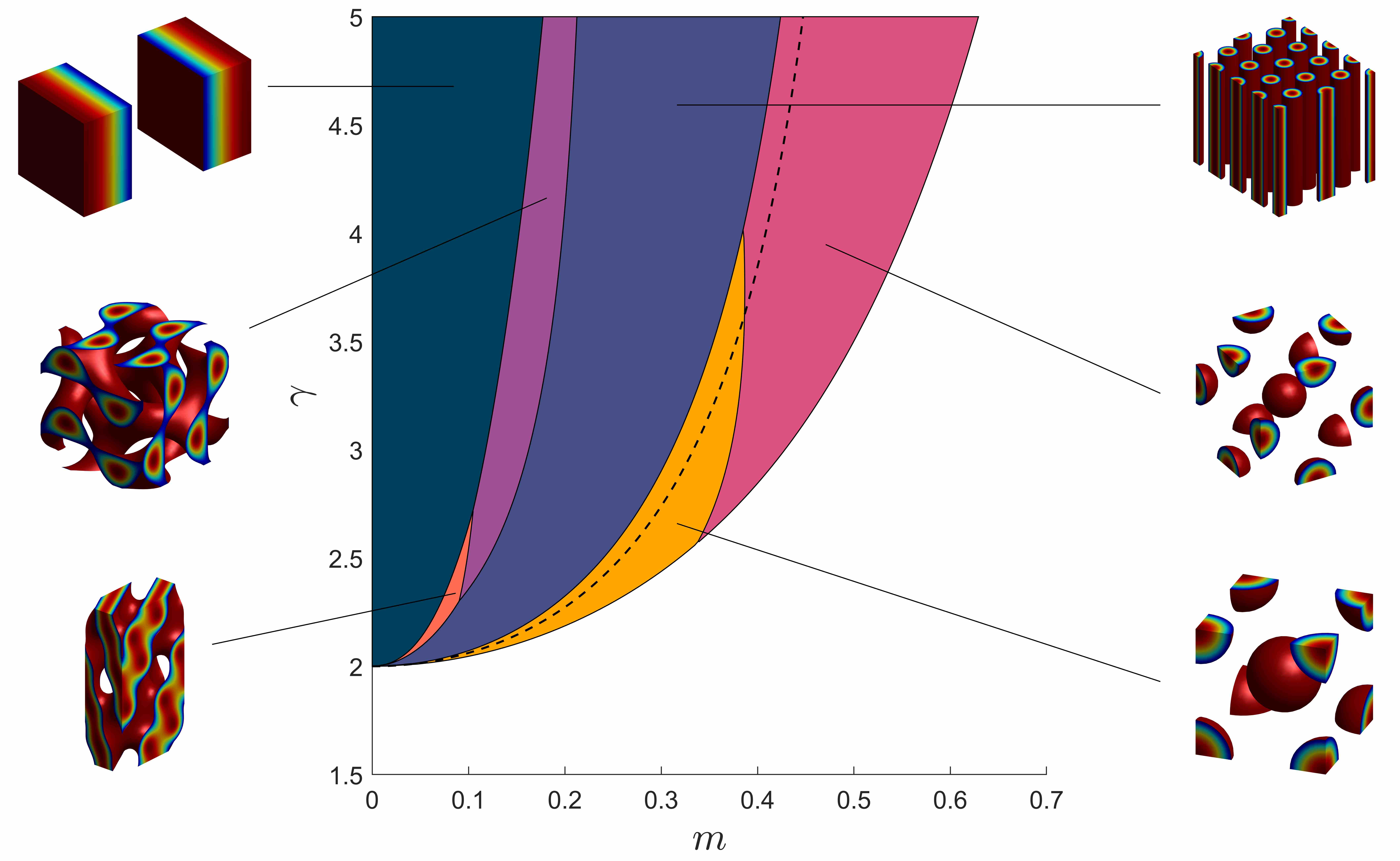}

            \caption{Phase diagram of global minimizers for the Ohta-Kawasaki energy (see the main text for an explanation of the terminology ``global minimizer''). A typical profile from each region is indicated. The mixed state, $u \equiv m$, is linearly stable below the dashed curve and the unique global minimizer below the coloured regions. The depicted profiles, and all in this paper, are level sets at $u \equiv 0$. That is, the positive and negative spaces represent those filled with one of the two different monomers.}
			\label{fig:mainresult}
\end{figure}

\noindent {\bf Crystallographic symmetries}  Experimentally, energy minimizers are seen to be periodic. In samples containing hundreds, or more, of repeated cells in all directions, the periodicity is independent of domain size. Crystallographic, or space group, symmetries combine the translational
symmetries of a lattice together with other elements such as
directional flips, rotation and screw axes.
Physically, determining the space group a given material belongs to
is an essential step in structure analysis as it minimizes the
information required for a complete description.
We use this reduction to construct efficient numerical methods
which guarantee our solutions have the desired symmetry.
There are 219, 230 including mirror images, three-dimensional
crystallographic space groups. 

Physically, global minimizers have been identified which are lamellar, and from space groups: 70 (O$^{70}$), 183 (HPC), 216 (CPS), 229 (BCC spheres) and 230 (double gyroid). We use this fact to
construct solutions possessing the correct symmetries. Examples of these phases and the energy landscape appear in Figure \ref{fig:mainresult}, which summarizes the main result of this paper.
It was constructed by finding minima of \eqref{eq:OKEnergy} within 
certain symmetry classes for many different values of $m$ and $\gamma$, 
where we vary both the profiles \emph{and} the length scales.  While complete details of how this was done are presented in Section~\ref{s:numerics}, let us briefly explain the main ideas here.  
The technique developed in Sections~\ref{s:symmetries}--\ref{s:fixedpoint} allows us to validate numerical approximations of critical points in the sense that a unique (true) critical point lies within an explicitly given distance of the numerical approximation, which is described by finitely many Fourier modes. In particular, we control all errors induced by the finite truncation. Moreover, we are able to establish the Morse index of the critical point and hence decide on (in)stability (see Section~\ref{s:stability}). Finally, we can compute the energy with excellent rigorous error control (see Section~\ref{s:energy}), hence we can, for fixed parameter values $m$ and $\gamma$, determine the minimizer of the energy.

A few caveats should be mentioned.
We do \emph{not} claim to have a mathematical proof that the solutions found are \emph{global} minimizers, not even among all periodic profiles (there is no guarantee that the global minimizing state is periodic, although it is widely believed to be the case). Such a result seems outside the realm of possibility of any method at the current status of science. However, we did an extensive search in many different symmetry groups, including all symmetry groups where stable states have been observed for any parameter values either experimentally or numerically. This leads us to candidates which we \emph{prove} to be local minimizers within the associated symmetry class, i.e., we prove that within an explicit distance (in the norm detailed in Theorem~\ref{thm:contraction}) of the numerical approximation there is a unique point that is stationary with respect to variations in  both the profile and  the length scales. Moreover, we prove that this stationary point is a local minimizer with respect to all such variations (Morse index 0). We then compute the energy of these local minimizers with a rigorous explicit error bound (confidence interval), which allows us to conclude which of the local minimizers in the various space groups has the lowest energy, i.e., we prove that is has lower energy than any of the other local minima that we have found. To avoid overly complicated formulations, throughout the remainder of this paper we use the terminology ``global minimizers'' for these stationary points. \\

\noindent {\bf Connection to previous results.}
Most relevant to the
current work are results discussing the structure of solutions near the point $m=0, \gamma = 2$ such as \cite{CPW,CMW,vdBW2}.
The existence of various solutions has been demonstrated, and it is known that
global minimizers are periodic as $\gamma \to \infty$ for small volume fraction \cite{CP1,CP2,Goldman2014}. We are only
interested in $\gamma \ge 2$ as no non-trivial minimizers exist for $\gamma < 2$ \cite{CPW}. 

There have also been considerable numerical computations by
solving the stationary problem, \cite{SCN,Wanner, vdBW2, vdBW3}, integrating the PDE \cite{CPW,Teramoto}, minimizing the functional,  and, in the material science community, using Self-Consistent Field Theory (SCFT) for this problem \cite{BatesRev}. 
SCFT is the current state of the art for realistic physical simulations of block copolymers. In \cite{Choksi2003} the authors compare the SCFT approach to the Ohta-Kawasaki energy 
and derive the correspondence between the SCFT parameters ($\chi N, f$) and those used here ($\gamma, m$) in the limit $(\gamma \to 2^+, m \to 0).$

In papers \cite{vdBW2} and \cite{vdBW3}, the present authors developed rigorous numerical
methods to explore the parameter space efficiently, first in two dimensions
by following equal energy curves and subsequently by finding symmetric solutions in three
dimensions. This work is the extension of those two works to fully three-dimensional structures where we also rigorously optimize the length scales and provide rigorous bounds on the energy.
This method provides a rigorous bound on radius of a ball containing both the numerical approximation and the exact solution. The foundations of this functional analytic technique are outlined in \cite{JBAMS}; see also Section~\ref{s:fixedpoint} for further references.\\

\noindent{\bf Main contributions}
The literature of experimental results, both computational and physical, contains many examples showing that what was once considered true no longer is. The field of diblock copolymers is no exception. Early physical experiments suggested the existence of a stable perforated lamella phase which was later shown to be (physically) metastable\footnote{Critical points which are local minimizers but not global minimizers are typically called metastable in the experimental literature. Random fluctuations in the system can, eventually, lead to lower energies.}. And, the first computational phase diagrams omitted the O$^{70}$ phase.

These refinements are due to improved techniques and resources. In this paper, we present an additional check when examining such problems. The use of rigorous numerics provides complete confidence in one's results and a way to determine if solutions are locally stable or not, thus ruling out mathematical metastability.

To this end, the main contributions in the theoretical development of this work are the extension of \cite{vdBW3} to all 230 space groups and to include domain scale optimization, rigorous bounds on the energy of solutions and a rigorous method for computing the Morse index of solutions. We then used these methods to determine the landscape of energy minimizers, prove the existence of additional profiles such as O$^{70}$, the perforated lamella and various unseen exotic profiles and find the first qualitative difference between the SCFT model and the OK energy. In particular, the BCC packed spheres are not minimizers indefinitely but rather the close-packed spheres are energetically preferable in a much larger of the phase space for the OK energy than the SCFT model. This is due to the difference in the potentials used in the two models \cite{Choksi2003}.\\

\noindent{\bf Outline of the paper}
The remainder of the paper is laid out as follows. In Section
\ref{s:symmetries} we explain how we set up the variation of length scales in the context of an arbitrary symmetry group. Optimization over both the symmetric profiles and length scales leads to a zero finding problem, which is introduced in Section~\ref{s:zerofinding}. The relevant group actions in Fourier space are analyzed in Section~\ref{s:groupactions}, where we also prove several lemmas necessary to extend the specialized setting of~\cite{vdBW3} to the general approach in the current paper. In Section~\ref{s:fixedpoint}, the fixed point theorem is formulated and proven. Sections \ref{s:energy} and \ref{s:stability} present the rigorous computations of the energy and Morse Index respectively. Section 
\ref{s:numerics} presents a description of our computational strategy and our numerical results. Many of the detailed estimates required for the proof reside in Appendix A.

Matlab code for the proofs and figures presented in this paper is available at~\cite{codeoptimalSG}.

%% file: symmetries.tex
%!TEX root = ms.tex

Since are interested in periodic critical points of the energy $\EEE$, we expand $u$ as  
\begin{equation}\label{e:Fourier}
  u (\x)= \sum_{k \in \Z^3} c_k \exp(\textup{i} \L k \cdot \x),
\end{equation}
for some invertible $3 \times 3$ matrix $\L$.
The Laplacian corresponds to a diagonal operator on the Fourier coefficients
\[ 
  -\Delta u (\x) = 
  \sum_{k \in \Z^3} \DLk c_k \exp(\textup{i} \L k \cdot \x)
\]
where
\[
  \DLk \bydef \sum_{j=1}^{3} ( \L k )_j^2 .
\]

\begin{remark}
To relate $\L$ to periodicity in $\x$ we observe that $u(\x+2\pi (\L^{-1})^T \vec{e}_j) = u(\x)$ for $j=1,\dots, 3$, where $\{\vec{e}_j\}_{j=1}^3$ form the standard basis of $\R^3$. In particular, the (fundamental) domain of periodicity $\Omega$ is the parallelepiped given by $ \{\x \in \R^3 : \x =  2 \pi (\L^{-1})^T \tilde{\x} \text{ with } \tilde{\x} \in [0,1]^3 \}$.
\end{remark}

\begin{remark}
To optimize over the length scale(s) we will need to be able to scale the vectors in the unit cell. We incorporate this into $\L$ by decomposing it as
\begin{equation}\label{e:Ldecomposition}
  \L = \LL \left( \begin{array}{ccc} 
  \ell_1 & 0 & 0 \\
  0 &  \ell_2 & 0 \\
  0 & 0 & \ell_3 
  \end{array} \right),
\end{equation}
where $\LL$ is independent of the length scale(s) and depends on the symmetry class under consideration only, see Remark~\ref{r:tildeL} below.
This thus corresponds to the fundamental domain
\[
  \{\x \in \R^3 : \x =  2 \pi (\LL^{-1})^T \hat{\x} \text{ with } \hat{\x} \in [0,1/\ell_1] \times [0,1/\ell_2] \times [0,1/\ell_3] \} .
\]
\end{remark}

In addition to periodicity we assume
 invariance under some space group $\G$:
\[
  u(C_g \x + 2\pi  (\L^{-1})^T  D_g)=u(\x) \qquad\text{for all } g \in \G,
\]
where $C_g$ is an orthogonal matrix ($C_g^{-1} = C_g^T$), and $D_g \in [0,1]^3$.
In this paper we will only consider space groups that have a priori fixed angles in the periodicty cell and that do not allow continuous shifts. While we believe our methods can certainly be adopted to such settings, this will involve some additional technicalities (mainly due to the presence  of continuous symmetries), which we want to avoid here, since those groups are not required to describe the patterns that are observed in the Ohta-Kawasaki problem. Nevertheless, this could be useful in other applications, and we plan to return to these issues in future work.
This means that we focus on the following space groups~\cite{Bilbao,ITA}:
\begin{alignat*}{1}
  \text{orthorhombic groups:}~ & \text{16--24} \text{ and } \text{47--74}, \\
  \text{tetragonal groups:}~ &  \text{81-98} \text{ and } \text{111--142}, \\
  \text{trigonal groups:}~ & \text{147--155} \text{ and } \text{162--167}, \\
  \text{hexagonal groups:}~ & \text{174--182} \text{ and } \text{187--194}, \\
  \text{cubic groups:}~ &  \text{195--230},
\end{alignat*}
where in each type we have, in view of the discussion on continuous symmetries above, excluded the so-called pyramidal groups.
Furthermore, throughout we do not describe (essentially) one- and two-dimensional cases (e.g.\ lamellae, cylinders), but the techniques carry over immediately, and the code also covers  these lower dimensional cases.

Several remarks are in order to describe the relation between periodicity and the symmetries encoded in the space group.
\begin{remark}\label{r:match}
Depending on the type of group $\G$,
the symmetries may imply that some of the length scales $\ell_j$ in the decomposition~\eqref{e:Ldecomposition} are not independent. 
In particular,  there are one, two or three independent scale variables $\ls_j$. 
We denote the number of independent length scales by $J$.
For the orthorhombic groups we have $J=3$ and we simply set  $\ell_j = \ls_j$ for $j=1,2,3$.
For the tetragonal, trigonal and hexagonal groups we have $J=2$: $\ell_1=\ell_2=\ls_1$
and $\ell_3 = \ls_2$. For the cubic groups we have $J=1$,
hence $\ell_1=\ell_2=\ell_3 =\ls_1$. 
In particular, we introduce the dilation matrices $\hLLl$ as follows:
\[
  J=3:~ \hLLl  \bydef 
  \left( \begin{smallmatrix}
    l_1 & 0 & 0 \\
    0 &  l_2 & 0 \\
    0 & 0 & l_3 
    \end{smallmatrix} \right) ,
	\qquad
  J=2:~ \hLLl \bydef   
  \left( \begin{smallmatrix}
    l_1 & 0 & 0 \\
    0 &  l_1 & 0 \\
    0 & 0 & l_2 
    \end{smallmatrix} \right),
	\qquad
  J=1:~ \hLLl \bydef   
  \left( \begin{smallmatrix}
    l_1 & 0 & 0 \\
    0 &  l_1 & 0 \\
    0 & 0 & l_1 
    \end{smallmatrix} \right) .
\]

The coordinates are thus gathered according to which independent length scale $l_j$ they are linked to. 
The matrices $C_g$ act block-wise 
on these grouped coordinates in the sense that
$C_g$ is diagonal when $J=1$, while when $J=2$ then $C_g$ is blockdiagonal with blocks of size $2$ and $1$: $(C_g)_{ii'}=0$ for $i=1,2$, $i'=3$ and for 
$i=3$, $i'=1,2$. We say that $C_g$ has a periodicity compatible block structure.
A different way to characterize this compatibility is through the commutativity property
\begin{equation}\label{e:CLcommute}
    C_g \hLLl = \hLLl C_g \qquad \text{for any } l \in \RJ_+ \text{ and any } g\in \G.
\end{equation}
\end{remark}

\begin{remark}\label{r:tildeL}
For the orthorhombic, tetragonal and cubic space groups we choose the factor $\LL$ in the decomposition~\eqref{e:Ldecomposition} to simply be the identity. For the tetragonal and hexagonal space groups we choose
\[
  \LL = \LL_{\triangle} \bydef \left( \begin{array}{ccc}
  1 & 1/2 & 0  \\
  0 & \sqrt{3}/2 & 0 \\
  0 & 0 & 1
   \end{array} \right) .
\]
We note that $\LL_{\triangle}$ commutes with $\hLLl$ for $J=2$ and any $l\in \R^2_+$. 
With this choice of $\LL$ the matrices $C_g$ map the lattices $(\LL^{-1})^T\Z^3$ and $\LL\Z^3$ to themselves for any $g\in \G$.
\end{remark}

The compatibility between the independent length scales and the symmetry group
described in Remark~\ref{r:match} implies that
we may simplify the description of $\DLk$ as follows.
No cross product terms $\ls_j \ls_{j'}$ with $j \neq j'$ appear in the expression for $\DLk$. This motivates us to
introduce $\kappa_j=\ls_j^2$.
We distinguish four cases and in each write 
$\DLk=\Dkk$, where,  
depending on the type of space group, 
\begin{alignat*}{3}
\text{orthorhombic}:&\quad&
\Dkk &\bydef \kappa_1 k_1^2  + \kappa_2 k_2^2 + \kappa_3 k_3^2 && J=3,\\
\text{tetragonal}:&& 
\Dkk &\bydef \kappa_1 (k_1^2 + k_2^2) + \kappa_2 k_3^2  && J=2,\\
\text{trigonal and hexagonal}:
&& \Dkk &\bydef \kappa_1 \bigl(k_1^2 + k_1k_2 + k_2^2 \bigr) + \kappa_2 k_3^2 && J=2, \\
\text{cubic}:&&
\Dkk &\bydef \kappa_1 (k_1^2 + k_2^2 + k_3^2) &\qquad& J=1. 
\end{alignat*}
We summarize these as 
\[
  \Dkk = \sum_{j=1}^J \Dkj{j} \kappa_j,
 \qquad\text{where } \Dkj{j} \in [0,\infty) 
 \text{ for all }k\in \Z^3 \text{ and } 1\leq j \leq J.
\]
Clearly $\Dkk$ is linear in $\kappa$ for fixed $k$, while it is a positive
definite quadratic form in $k$ for fixed $\kappa \in \RJ_+$. The
linearity in $\kappa$ will be most important, hence it is stressed by choosing
the notation~$\Dkk$. On occasion we will also interpret, for fixed $k$, \
the vector
$\Dk = \{\Dkj{j}\}_{j=1}^J$ as an element of the dual of $\RJ$.

A crucial property of $\Dkj{j}$ is that 
\begin{equation}\label{e:DkjCginvariance}
  \Dkj[\LL^{-1} C_g \LL k]{j} =\Dkj[k]{j} \qquad\text{for any }
  k\in \Z^3, 1\leq j \leq J, \text{ and } g \in \G.
\end{equation} 
For the orthorhombic, tetragonal and cubic space groups this follows from the orthogonality and the periodity compatible block structure of $C_g$, whereas for the tetragonal and trigonal space groups~\eqref{e:DkjCginvariance} puts an additional constraint on the admissible matrices $C_g$.

Finally, to recover $\L$ from $\kappa$ we use $\L = \L_\kappa \bydef \LL \hLLl[\sqrt{\kappa}]$.
The corresponding domain of periodicity is denoted by
\begin{equation}\label{e:omegakappa}
  \Omega_\kappa \bydef 
  \{\x \in \R^3 : \x =  2 \pi (\L_\kappa^{-1})^T \tilde{\x} 
                       \text{ with } \tilde{\x} \in [0,1]^3 \} .
\end{equation}
It follows from~\eqref{e:CLcommute} that
\begin{equation}\label{e:LCL}
  \L_\kappa^{-1} C_g \L_\kappa= \LL^{-1} C_g \LL,
\end{equation}
and since $C_g$ maps the lattice $\LL \Z^3$ to itself, see Remark~\ref{r:tildeL}, it follows from~\eqref{e:LCL} that
$C_g$ also maps the lattice 
$(\L_\kappa^{-1})^T \Z^3$ to itself.

%% file: zerofinding.tex
%!TEX root = ms.tex

We aim to find stationary points of the energy $\EEE$ with respect to variations in the Fourier coefficients~$c_k$ and length scales $\kappa \in \RJ_+$, under the constraint $c_0=m$. We express $\EEE$ in terms of $\kappa$ and~$c$:
\begin{equation}\label{e:defE}
  E(\kappa,c) \bydef 
      \frac{1}{2} \sum_{k \in \Z^3 \setminus \{0\}} \PP(\Dkk) \, c_k c_{-k}
  \, +  \frac{1}{4} \conv{c^4}_0  + \frac{1-2m^2}{4} ,
\end{equation}
where $\PP:(0,\infty) \to \R$ is given by
\[
  \PP(y) \bydef \frac{1}{\gamma^2} \,\! y  - 1 + \frac{1}{y} .
\]
In what follows 
we will use that for large Fourier modes $\PP(\Dkk)$ has fixed sign (positive) and its size increases monotonically. Indeed 
\begin{alignat}{2}
  \PP(y) & > 0  &\qquad&\text{for all } y > \yP, \label{e:PyP1} \\
  \PP'(y) & > 0 &\qquad&\text{for all } y > \yP, \label{e:PyP2}
\end{alignat}
where 
\begin{equation}\label{e:defy0}
  \yP \bydef \begin{cases}
   \gamma  & \quad\text{for } \gamma \leq 2,\\
   \frac{\gamma^2+\sqrt{\gamma^4-4\gamma^2}}{2}
    & \quad\text{for } \gamma > 2.\\
   \end{cases}
\end{equation}
The parameter range $\gamma \leq 2$ is included for completeness only. No pattern formation is observed in that regime due to global convexity of the energy functional~\cite{CPW}.

Throughout this paper we use the notation $\conv{c \, c'}$ for the discrete convolution product
\[
\conv{c \, c'}_k \bydef \sum_{k_1+k_2=k} c_{k_1} c'_{k_2},
\]
which naturally extends to powers. In particular
 \[
   \conv{c^4}_0 = \sum_{k_1+k_2+k_3+k_4=0} c_{k_1} c_{k_2} c_{k_3} c_{k_4} .
 \] 
The normalization of $E$ is such that the trivial uniform state has zero energy.

In Fourier space the PDE~\eqref{eq:OK} is transformed to
\begin{equation}\label{e:fk}
  f_k(\kappa,c) \bydef  \PP(\Dkk) c_k + \conv{c^3}_k  =0
\end{equation}
for $k \in \Z \setminus \{0\}$.
\begin{remark}\label{r:dEdc}
Equation~\eqref{e:fk} corresponds to stationarity of $E$ with respect to variations in the Fourier coefficients. Indeed, by using that $\Dk[-k] = \Dk$ we see that $\frac{\partial E}{\partial c_{k}}=f_{-k}$ for all $k \in \Z \setminus \{0\}$.
\end{remark}
Stationarity of the energy with respect to variations in $\kappa$
is equivalent to 
\begin{equation}\label{e:hj}
  h_j(\kappa,c) \bydef \frac{1}{2} \sum_{k \in \Z^3 \setminus \{0\} } \PP'(\Dkk)  \, \Dk^j \, c_k c_{-k} = 0,
\end{equation}
for $1 \leq j \leq J$.
Namely,
\begin{equation}\label{e:dEdkappa}
    \frac{\partial E}{\partial \kappa_j}(\kappa,c)= h_j(\kappa,c)  =0
	\qquad \text{ for all } j =1,\dots,J,
\end{equation}
corresponds to stationarity of the energy $\EEE$ with respect to variations in the proportions of the fundamental domain of the periodicity cell $\Omega=\Omega_\kappa$ which respect the symmetry type (in the sense of Remark~\ref{r:match}).

For $u$ a real-valued function, we have the symmetry $c_{-k}=c_k^*$, and in particular $c_k c_{-k} = |c_k|^2$. However, we choose to work over $\C$ throughout without requiring $c_{-k}=c_k^*$, 
as this unifies some of the arguments. 
In Definition~\ref{def:conjugation} we will introduce a conjugation operator which is used to recover real-valuedness of the solution a posteriori, see Theorem~\ref{thm:contraction}.

We choose as the norm on the Fourier indices
\begin{equation}\label{e:knorm}
   \knorm \bydef \sqrt{\Dk \bkappa},
\end{equation}
for some $\bkappa \in \R^J_+$ to be fixed later.
Indeed, $\bkappa$ will be chosen to be the numerical approximation of the optimal domain size parameters. 
On the set of Fourier coefficients we introduce the norm on 
\begin{equation}\label{e:nu-norm}
  \| c \|_\nu   \bydef  \sum_{k \in \Z^3} |c_k|  \nu^{\knorm}.
\end{equation}
The corresponding Banach space is a Banach algebra under convolution multiplication:
\begin{equation}\label{e:BA}
  \| \conv{c \, c'} \|_\nu \leq \| c \|_\nu  \| c' \|_\nu .
\end{equation}

On $\RJ$ and $\CJ$ we  choose the norm
\begin{equation}\label{e:normFJ}
   | \kappa |_{\FFJ} \bydef \max_{1\leq j \leq J} \frac{|\kappa_j|}{\bkappa_j} 
  \qquad \text{for } \F = \R \text{ and } \F=\C.
\end{equation}
The norm on the dual is then given by 
\[ 
  | q |^*_{\FFJ}  \bydef \sum_{j=1}^J |q_j|  \bkappa_j 
  \qquad\text{for } q \in \FFJ.
\]
In particular, this choice of norm has the convenient property 
(writing $\Dk = (\Dkj{j})_{1\leq j\leq J}$)
\begin{equation}\label{e:choice}
  | \Dk |^*_{\FFJ} = \Dk \bkappa, 
   \qquad \text{for } \F = \R \text{ and } \F=\C.
\end{equation}

\begin{definition}\label{def:conjugation}
On $\CJ$, $\C^{\Z^3}$ and then on $\CJ\times \C^{\Z^3}$ we introduce the conjugation operation
\begin{alignat*}{1}
  (\fullconj \kappa)_j &\bydef \overline{\kappa_j} ,\\
  (\fullconj c)_k & \bydef \overline{c_{-k}} ,\\
  \fullconj (\kappa,c) & \bydef (\fullconj \kappa, \fullconj c),
\end{alignat*}
where the overline denotes complex conjugation.
\end{definition}
The convolution product is $\fullconj$-equivariant:
\[
  \conv{\fullconj c \, \fullconj c'} = \fullconj \conv{c \, c'}.
\]
It then follows from the formulas~\eqref{e:fk} and~\eqref{e:hj} for $f$ and $h$,
combined with the fact that $P$ is real analytic, that
\begin{equation}\label{e:hfconjequivariant}
  \fullconj \bigl(h(\kappa,c),f(\kappa,c) \bigr) =
  \bigl(h(\fullconj(\kappa,c)),f(\fullconj(\kappa,c)) \bigr),
\end{equation}
and, intimately related to this (since $f$ and $h$ correspond to the partial derivatives of $E$),
\[
  E(\fullconj(k,c)) = \overline{E(k,c)}.
\]

%% file: groupactions.tex
%!TEX root = ms.tex

We briefly summarize the setup from~\cite{vdBW3}, where more details can be found. 
The symmetry operations in $\G$ lead to a right group action $\gamma_g$ on the Fourier coefficients. Namely, when $c$ denote the Fourier coefficients of $u(\x)$, see~\eqref{e:Fourier}, then
\[
u(C_g \x + D_g) = 
\sum_{k \in \Z^3} (\gamma_g c)_k \exp(\textup{i} \L k \cdot \x),
\]
where we write
\begin{equation}\label{e:defgamma}
  (\gamma_g c)_k = \alpha_g(k) \, c_{\beta_g(k)}.
\end{equation}
Here (in view of~\eqref{e:LCL})
\begin{equation}\label{e:defbeta} 
  \beta_g (k) = \LL^{-1} C_g \LL k
\end{equation}
is well defined as a map from $\Z^3$ to itself, see Remark~\ref{r:tildeL},
and
\[ 
  \alpha_g(k) =  \exp \bigl( 2\pi \textup{i} \LL^{-1} C_g \LL k \cdot D_g \bigr)
  = \exp \bigl(2\pi \textup{i} \beta_g (k)  \cdot D_g \bigr) .
\]
We see that $\alpha_g(k) \in \{ z \in \C : |z|=1 \}$ for all $k \in \Z^3$ and all $g \in \G$. It satisfies 
\[
  \alpha_{g_1 g_2} (k) = \alpha_{g_1} (\beta_{g_2} (k)) \alpha_{g_2}(k),
\]
as well as 
\begin{equation}\label{e:alphaadd}
	 \alpha_g(k+k')=\alpha_g(k)\alpha_g(k') .
\end{equation}
The left group action $\beta_g$ acts linearly on $\Z^3$ and throughout the remainder of this paper it is denoted by 
\[
\gacts \bydef \beta_g(k).
\]
The orbit of $k \in \Z^3$ is given by
\[
  \Gacts \bydef \{ \gacts : g \in \G \},
\]
and we will denote the number of elements by $|\Gacts|$.
The stabilizer of $k \in \Z^3$ is denoted by
\[
  \G_k \bydef \{ g \in \G : \gacts[k]=k \}.
\]

\begin{remark}\label{r:invariantnorm}
In view of~\eqref{e:DkjCginvariance} and~\eqref{e:defbeta} we have
\begin{equation}\label{e:Deltajgk}
	\Delta^j_{\gacts} = \Delta^j_k
	\qquad \text{for all } 1\leq j \leq J \text{ and } k \in \Z^3.
\end{equation}
This implies that $\Dk$ is invariant under the $\beta_g$ action of $\G$.
Hence $\knorm$ is invariant under $\beta_g$. Consequently, the norm 
$\| c \|_\nu$ is invariant under $\gamma_g$.
\end{remark}

We define the Banach space of symmetric sequences by
\[ 
  \Xsym \bydef \{ c = (c_k)_{k\in{\Z^3}} : \| c \|_\nu < \infty \text{ and }
  \gamma_g c = c \text{ for all } g \in \G \}.
\]
It follows from~\cite[Lemma~3.11]{vdBW3} that symmetric solutions have only nonzero coefficients for indices in
\[
  \Zsym \bydef \{ k \in \Z^3 : \sum_{g \in \G_k} \alpha_g(k) = |\G_k| \}.
\]
In other words
\begin{equation}\label{e:xsymzsym}
  c_k = 0 \qquad \text{for all } c \in \Xsym  \text{ and } k \notin \Zsym.
\end{equation}
We choose a fundamental domain $\Zdom \subset \Z^3$, which contains precisely one element of each orbit.
The symmetry reduced indices are given by 
\[ 
   \ZZ \bydef \Zsym \cap \Zdom,
\]
and since $c_0 =m$ is not a variable, we will often restrict attention to $\ZZ_0 \bydef \ZZ \setminus \{0\}$.

The Banach space of symmetry reduced variables is given by 
\[
  \X \bydef \{ b = (b_k)_{k\in\ZZ} : \| b \|_{\X} < \infty \} ,
\]
with norm
\[
  \| b \|_{\X} \bydef \sum_{k \in \ZZ} |b_k| \omega_k ,
\] 
where the weights are given by
\[
  \omega_k \bydef |\Gacts| \nu^{\knorm} .
\]
We define the subspace $\X_0 = \{ b \in \X : b_0=0 \}$
and we identify $b \in \X_0$ with $b=(b_k)_{k \in \ZZ_0}$.
The induced norm on $\X_0$ is denoted by
\begin{equation}\label{e:X0norm}
  \| b \|_{\X_0} \bydef \sum_{k \in \ZZ_0} |b_k| \omega_k .
\end{equation}

For $k' \in \Gacts$ there is (at least one) $\widetilde{g}(k,k')\in \G$ such that $\gacts = k'$
and we define $\talpha(k,k')=\alpha_{\widetilde{g}(k,k')}^{-1} (k)$.
This definition is independent of the choice of $\widetilde{g}$ by~\cite[Lemma~3.13]{vdBW3}. 
We utilize $\talpha$ to make the relation between  $\X_0$ and $\Xsym$ explicit.
For any $k \in \Z^3$, let the basis element $\e_k$ be given by
$(\e_k)_{k'} = \delta_{k,k'}$.
Then we define the map $\sigma$ on $\X$ by
\begin{equation}\label{e:defsigma}
  \sigma(b) \bydef \sum_{k \in \ZZ} b_k \sum_{k' \in \Gacts} \talpha(k,k') \e_{k'}.
\end{equation}
It follows from \cite[Lemma~3.15]{vdBW3} that $\sigma(b) \in \Xsym$.
Moreover, $\sigma(b)_k=b_k$ for $k \in \ZZ$, 
and (\cite[Lemma~3.17]{vdBW3})
\begin{equation}\label{e:compatiblenorms}
  \|\sigma(b)\|_\nu = \| b \|_{\X}.
\end{equation}
Furthermore, since $\Xsym$ and $\X$ may be identified~\cite[Lemma~3.16]{vdBW3}, we have
\begin{equation}\label{e:pointwise}
   |c_k| \leq \|c\|_\nu \, \omega_k^{-1} \qquad \text{for all } k \in \Z^3 \text{ and }  c \in \Xsym. 
\end{equation}

\begin{remark}\label{r:extendconj}
We lift the conjugation operation, see Definition~\ref{def:conjugation},
to $\X$ as
\[
  (\conj b)_k = (\fullconj \sigma(b))_k \qquad \text{ for } k \in \ZZ. 
\]
We extend it to the product space $\FJ \times \X$, on which it acts component-wise as $\conj (\kappa,b)=(\conj \kappa,\conj b)$.
The set of conjugate symmetric elements is denoted by
\begin{equation}\label{e:defSsym}
  \Ssym \bydef \{ (\kappa,b) \in \FJ \times \X : \conj (\kappa,b) = (\kappa,b) \}  = \RJ \times \{ b \in X : \conj b = b \}.
\end{equation}
\end{remark}

\begin{remark}\label{r:explicitGamma}
In order to make the action of $\conj$ on $b$ explicit,
it is useful to introduce 
\[
  \Phi(k) \bydef \Gacts \cap \Zdom \qquad \text{for } k \in \Z^3,
\]
i.e., $\Phi(k)$ is the (unique) element in the group order of $k$ which lies in the fundamental domain $\Zdom$.
Since $\Gacts[-k] = - \Gacts$, we have $\Phi(-\Phi(-k))=k$ for $k \in \Zdom$.
We define the permutation $\tau$ on $\ZZ_0$ by 
\[
  \tau(k) \bydef \Phi(-k) \qquad\text{for } k\in \ZZ_0,
\]
which thus has the property $\tau=\tau^{-1}$.
We also set
\[
  \phi_k \bydef \talpha(\tau(k),-k) \qquad \text{for } k \in \ZZ_0 .
\]
With this notation we obtain
$\sigma(b)_{-k} =  \phi_k b_{\tau(k)}$ and 
\[
  (\conj b)_k = \overline{\sigma(b)_{-k}} =  \overline{\phi_k} \, \overline{b_{\tau(k)}} .
\]
By applying $\conj$ twice it allows that 
$\phi_k \overline{\phi_{\tau(k)}}$, hence $\phi_k = \phi_{\tau(k)}$.
For later use we introduce the linear operator (on $\X$ and then its extension to $\C^J \times \X$)
\begin{subequations}
	\label{e:defJ}
\begin{alignat}{1}
  (\J \;\! b)_k & \bydef \phi_k b_{\tau(k)},  \\
  \J\;\!(\kappa,b) & = (\kappa,\J\;\! b ),  
\end{alignat}
\end{subequations}
so that $\conj b = \overline{\J\;\!  b}$. 
By applying $\conj$ twice we obtain $b_k = (\conj^2 b)_k =  \overline{\phi_k} \phi_\tau(k) b_{\tau^2(k)}$, hence $\phi_k = \phi_{\tau(k)}$,
and 
\begin{equation}\label{e:usefulJ}
  \overline{\J} = \J^{-1} \qquad\text{and}\qquad  \J^T = \J,
\end{equation}
where we slightly abuse matrix notation for $\J$, and the overline denotes elementwise complex conjugation.
\end{remark}

Furthermore, it follows from the property~\eqref{e:alphaadd} of $\alpha_g$, linearity of $\beta_g$,
and the fact that $\beta_g$ permutes the lattice $\Z^3$, that the convolution product respects the symmetry:
\[
  \conv{ \gamma_g a \, \gamma_g b }=   \gamma_g \conv{  a b }.
\]
By combining this invariance with that of $\Dk$, 
see Remark~\ref{r:invariantnorm}, 
we conclude that $f$ is an unbounded map from $\Xsym$ to $\Xsym$,
see Section~\ref{s:fixedpoint} for a more precise statement.

\begin{theorem}\label{thm:solvePDE}
Let $\kappa \in \R^J_+$. Assume $b \in \X_0$ is such that $\conj b=b$ and
\[
  f_k(\sigma(m\e_0+b)) =0 \qquad\text{for all } k \in \ZZ_0.
\]
Then $u(\x)$ given by~\eqref{e:Fourier} with $c=\sigma(m\e_0+b)$
is a real-valued $\G$-invariant periodic solution of~\eqref{eq:OK}.
\end{theorem}
\begin{proof}
This is the content of Lemma~3.19 in \cite{vdBW3}.
\end{proof}

It will turn out to be useful to derive some additional properties of $\Xsym$.
We start with two lemmas that relate the symmetries to expressions such as the one for $h_j$.
\begin{lemma}\label{l:symquad}
Let $q : \Z^3 \to \C$ be such that $q(\gacts)=q(k)$ for all $g \in \G$,
and  $\sum_{k \in \Z^3} |q(k)|\nu^{-2\knorm} < \infty$.
Let $c, c' \in \Xsym$, then
\[
  \sum_{k \in \Z^3} q(k) c_k c'_{-k} =  \sum_{k \in \ZZ} q(k) |\Gacts| c_k  c'_{-k} .
\] 
\end{lemma}
\begin{proof}
The growth condition guarantees that the series on the left converges.
By using~\eqref{e:xsymzsym} and the orbit-stabilizer theorem we write
\begin{equation}\label{e:qcc}
	\sum_{k \in \Z^3} q(k) c_k c'_{-k}
 =
	\sum_{k \in \Zsym} q(k) c_k c'_{-k}
=  
	\sum_{k \in \ZZ} \frac{1}{|\G_k|} \sum_{g \in \G} q(\gacts) c_{\gacts} c'_{-\gacts} .
\end{equation}
By using~\eqref{e:defgamma}, the assumption that $c$ and $c'$ are $\gamma_g$-invariant, and property~\eqref{e:alphaadd} of~$\alpha_g$, we infer that
\[
c_{\gacts} c'_{-\gacts} = \alpha^{-1}_g(k) \alpha^{-1}_g(-k) (\gamma_g c)_k (\gamma_g c')_{-k}  = c_k c'_{-k}.
\]
Since it is assumed that  $q(\gacts)=q(k)$ for all $g\in \G$, the argument of the inner sum in the righthand side of~\eqref{e:qcc} is thus independent of $g$.
Hence 
\[
	\sum_{k \in \Z^3} q(k) c_k c'_{-k} =
	\sum_{k \in \ZZ} \frac{|\G|}{|\G_k|} q(k) c_{k} c'_{-k}
	= \sum_{k \in \ZZ} |\Gacts| q(k) c_{k} c'_{-k},
\] 
where in the final step we have applied the orbit-stabilizer formula.
\end{proof}

\begin{lemma}\label{l:symdiff}
Let $q : \Z^3 \to \C$ be such that $q(\gacts)=q(k)$ for all $g \in \G$,
and  $\sum_{k \in \Z^3} |q(k)|\nu^{-2\knorm} < \infty$.
Assume that  $q(-k)=q(k)$ for all $k\in \Z^3$. 
Let $\mathcal{H}$ be the functional on $\X$ given by
\[
  \mathcal{H}(b)= \frac{1}{2} \sum_{k \in \Z^3} q(k) \sigma(b)_k \sigma(b)_{-k}.
\]
Then for any $k' \in \ZZ$
\[
  D_{b_{k'}} \mathcal{H}(b) =  q(k') |\Gacts[k']|  \sigma(b)_{-k'}.
\]
\end{lemma}
\begin{proof}
By the product rule 
\[
 D_{b_{k'}} \mathcal{H}(b) =
	 \frac{1}{2} \sum_{k \in \Z^3} q(k)  [ D_{b_{k'}} \sigma(b)_k] \sigma(b)_{-k}
	 +  \frac{1}{2} \sum_{k \in \Z^3} q(k)  \sigma(b)_k  [  D_{b_{k'}} \sigma(b)_{-k} ]
\]
Replacing $k$ by $-k$ in the second sum, and using that $q(-k)=q(k)$, this reduces to
\[
 D_{b_{k'}} \mathcal{H}(b) =
 \sum_{k \in \Z^3} q(k)  [D_{b_{k'}} \sigma(b)_k ] \sigma(b)_{-k}
\]
Since $\sigma$ is a linear map, $D_{b_{k'}} \sigma(b) = \sigma(\e_{k'})$
for any $k' \in \ZZ$.
Hence
\[
 D_{b_{k'}} \mathcal{H}(b) =
 \sum_{k \in \Z^3} q(k)  \sigma(\e_{k'})_k  \sigma(b)_{-k}  .
\]
We then apply Lemma~\ref{l:symquad} to this expression to obtain
\[
  D_{b_{k'}} \mathcal{H}(b) = 
  \sum_{k \in \ZZ} q(k) |\Gacts| (\e_{k'})_k  \sigma(b)_{-k} 
  = 
  q(k') |\Gacts[k']|  \sigma(b)_{-k'}. \qedhere
\]
\end{proof}

\begin{remark}
By putting $q(0)=0$ the results in Lemmas~\ref{l:symquad} and~\ref{l:symdiff} also hold when we replace $\Z^3$ and $\ZZ$ 
by $\Z^3 \setminus \{0\}$ and $\ZZ_0$, respectively.
\end{remark}

In a similar way we can deal with derivatives of powers.
\begin{lemma}\label{l:symdiffpower}
Let $p \in \mathbb{N}$, then 
\[
  D_{b_{k'}} \langle \sigma(b)^p \rangle_0 =   |\Gacts[k']| \, p \langle\sigma(b)^{p-1} \rangle_{-k'}.
\]
\end{lemma}
\begin{proof}
By the chain rule and arguments similar to those in the proof of Lemma~\ref{l:symdiff} we infer
\begin{equation*}
	D_{b_{k'}} \langle \sigma(b)^N \rangle_0 =
	p \sum_{k\in \Z^3} \langle \sigma(b)^{p-1} \rangle_{-k} \sigma(\e_{k'})_k 
	= p \sum_{k\in \Z} |\Gacts| \langle \sigma(b)^{p-1} \rangle_{-k} (\e_{k'})_k =  p \, |\Gacts[k']| \langle \sigma(b)^{p-1} \rangle_{-k'}  ,	
\end{equation*}	
where we have used Lemma~\ref{l:symquad} to obtain the second identity.
\end{proof}

The next lemma will be used to control the derivative of a symmetric convolution product.
\begin{lemma}\label{l:shiftestimate}
Let $c \in \Xsym$. Then
\[
   \sum_{k' \in \ZZ} \frac{\omega_{k'}}{\omega_{k}}  \sum_{k'' \in \Gacts}  |c_{k'-k''} |   \leq \|c\|_\nu 
   \qquad \text{for all } k \in \ZZ.
\]
\end{lemma}
\begin{proof}
Let $k \in \ZZ$. 
We recall that 
\[
  \omega_k = |\Gacts| \nu^{\knorm} = \frac{|\G|}{|\G_k|} \nu^{\knorm}.
\]
Using that $\knorm$ is $\beta_g$-invariant (see~\eqref{e:Deltajgk}), as well as the property  $|c_{\gacts}|=|c_k|$
for $c\in\Xsym$ (since $|\alpha_g(k)|=1$), we obtain 
\begin{alignat*}{1}
	 \sum_{k' \in \ZZ} \frac{\omega_{k'}}{\omega_{k}}  \sum_{k'' \in \Gacts}  |c_{k'-k''} | 
	 & =
	 \sum_{k' \in \ZZ}  \frac{|\G_{k}|}{|\G_{k'}|}  \sum_{k'' \in \Gacts}  |c_{k'-k''}|  \nu^{\knorm[k']-\knorm} \\
	 &=  \sum_{k' \in \ZZ}  \frac{1}{|\G_{k'}|}  \sum_{g \in \G}  \left|c_{k'-\gacts}\right|  \nu^{\knorm[k']-\knorm} \\
	 &= \sum_{k' \in \ZZ} \frac{1}{|\G_{k'}|}  \sum_{g \in \G}  \left|  c_{\ggacts{g^{-1}}{k'}-k} \right|  \nu^{\knorm[k']-\knorm} \\
	 &= \sum_{k' \in \ZZ} \frac{1}{|\G_{k'}|}  \sum_{g^{-1} \in \G}  \left|  c_{\ggacts{g^{-1}}{k'}-k} \right|  \nu^{\knorm[\ggacts{g^{-1}}{k'}]-\knorm} \\
	 &= \sum_{\tilde{k}' \in \Zsym} \left|  c_{\tilde{k}'-k} \right| 
	 \nu^{\knorm[\tilde{k}']-\knorm} \\
	 & \leq \sum_{\tilde{k}' \in \Zsym} \left|  c_{\tilde{k}'-k} \right| 
	 \nu^{\knorm[\tilde{k}'-k]}\\
	 &\leq  \sum_{\tilde{k} \in \Z^3} |  c_{\tilde{k}} | 
	 \, \nu^{\knorm[\tilde{k}]}\\
	 & = \| c\|_\nu,  
\end{alignat*}
where we have applied the orbit-stabilizer theorem twice,
and the first inequality is due to the reverse triangle inequality.
\end{proof}

%% file: fixedpoint.tex
%!TEX root = ms.tex

We write $\FJ_+ = \{ \kappa \in \FJ : \re \kappa_j >0 \}$
and define $\FF : \FJ_+ \times \X_0 \to \FJ \times \tX_0$ by
\begin{equation}\label{e:defFF}
\FF (\kappa,b) \bydef \bigl[H(\kappa,b),F(\kappa,b) \bigr],
 \bydef \bigl[ 
(h_j(\kappa,\sigma(m\e_0+b)))_{1 \leq j \leq J} , 
(\tf_k(\kappa,\sigma(m\e_0+b)))_{k \in \ZZ_0} \bigr],
\end{equation}
where $\tf_k$ is a rescaled version of $f_k$:
\begin{equation}\label{e:tf}
  \tf_k(\kappa,c) \bydef |\Gacts| f_k(\kappa,c).
\end{equation}
As we will see later in Remark~\ref{r:symmetricHessian},
this premultiplication by the number of elements in the group orbit is quite natural, while it does not complicate the setup significantly.
Here the Fourier part of the codomain can be chosen to be
\[
  \tX_0 \bydef \Bigl\{ (b_k)_{k \in \ZZ_0} : \textstyle\sum_{k \in \ZZ_0} \min\{1,|P(\Dk 1)|\} \, |b_k| \, \omega_k < \infty \Bigr\}.
\]
Projection onto $\FJ$ is denoted by $\pi_\kappa$,
whereas the projections on $\X_0$ and $\tX_0$ are both denoted by~$\pi_b$. 

At times it is convenient to rewrite $H$, by using Lemma~\ref{l:symquad}, as
\begin{subequations}
\begin{alignat}{1}
	H_j(\kappa,b) & = 
	\frac{1}{2} \sum_{k \in \Z^3 \setminus \{0\} } \PP'(\Dkk)  \, \Dk^j \, (\sigma(m\e_0+b))_k (\sigma(m\e_0+b))_{-k}  \label{e:Hwrite} \\
    & = 
   	\frac{1}{2} \sum_{k \in \ZZ_0 }  \PP'(\Dkk)  \, \Dk^j \, |\Gacts| \, b_k 
	(\sigma(m\e_0+b))_{-k}  \label{e:Hrewrite}.
\end{alignat}
\end{subequations}
Furthermore, $F$ naturally splits into a linear (in $b$) and nonlinear part:
\[
  F_k(\kappa,b) = P(\Dk \kappa ) |\Gacts| b_k + |\Gacts| \Phi_k(b), \qquad \text{for any } k \in \ZZ_0,
\]
where
\begin{equation}\label{e:defPhi}
  \Phi(b) \bydef  \conv{(\sigma(m\e_0+b))^3}.
\end{equation}

To perform numerical computations, and to set up a Newton-like fixed point scheme, we choose a finite dimensional projection.
For $K >0$ we define the finite cut-off index sets 
\begin{alignat*}{1}
%  \Z^K & \bydef \{ k \in \Z^3 : \kk \leq K \},  \\
  \ZZ^K_0 & \bydef \{ k \in \ZZ_0^3 : \kk \leq K \}.
\end{alignat*}
We denote by $N=N(K)$ the number of elements in $\ZZ^K_0$.
We set
\begin{alignat*}{1}
  \XK & \bydef \{ b \in \X_0 : b_k =0 \text{ for all } k \notin \ZZ^K_0 \},\\
  \Xinfty &\bydef \{ b \in \X_0 : b_k =0 \text{ for all } k \in \ZZ^K_0 \}.
\end{alignat*}
The total space is now $ X= \FJ \times \XK \times \Xinfty$
with projections $\pi_\kappa$, $\pi_K$ and $\pi_\infty$.
To reduce clutter, we do not introduce notation for the natural inclusions of $\FJ$, $\XK$, $\Xinfty$ and $\X_0$ into the product space $\X$.
It should always be clear from the context whether a constituents or a subspaces is meant.
\begin{remark}\label{r:matrixmatrix}
For linear operators $M$ on $\FJ \times \XK$ we will use the notation
\begin{alignat*}{1}
  \pi_\kappa M &= M_{11} \pi_\kappa + M_{12} \pi_K ,\\
  \pi_K M &= M_{21} \pi_\kappa + M_{22} \pi_K,
\end{alignat*}
hence in matrix notation
\[
   M = \left( \begin{array}{cc} M_{11} & M_{12} \\ M_{21} & M_{22} \end{array} \right),
\]
with $M_{11}$ a $J\times J$ matrix, $M_{12}$ a $J \times N$ matrix, $M_{21}$ an $N \times J$ matrix, and $M_{22}$ an $N \times N$ matrix.
\end{remark}

We compute numerically an approximate zero $\bx=(\bkappa,\bb) \in \RJ_+ \times \X_0^K$
of 
\[
  (\pi_\kappa+\pi_K) \FF(x) =0.
\]
\begin{remark}
We will restrict attention to $(\bkappa,\bb)$ in the conjugate symmetric set $\Ssym$, replacing $\bb$ by $\frac{1}{2}(\bb+\conj \bb)$ if needed. Since $\bkappa$  naturally lies in $\R^J_+$, in fact we restrict attention to 
\[
  \Ssym_+ \bydef  \Ssym \cap (\FJ_+ \times \X_0)  = 
  \{ x=(\kappa,b) \in \Ssym : \kappa \in \RJ_+ \}.
\]
\end{remark}
We then compute numerically the Jacobian 
\[
  M = D \bigl[(\pi_\kappa+\pi_K) \, \FF \circ (\pi_\kappa+\pi_K)\bigr] (\bx)
\]
of the finitely truncated problem, evaluated at the numerical zero.
Finally, we invert the $(J+N)\times(J+N)$ matrix~$M$ numerically to obtain a matrix~$A$,
which may be decomposed into blocks as introduced in Remark~\ref{r:matrixmatrix}. 
We use this matrix~$A$ to set up the fixed point problem.
The tail part of the full Jacobian of $\FF$ will be approximated by
\begin{equation}\label{e:defLambda}
  (\Lambda b)_k \bydef \PP(\Dk \bkappa) |\Gacts| b_k 
  \qquad\text{for all } k \in \ZZ_0.
\end{equation}
In order to invert $\Lambda$ on $\X^\infty_0$, we assume that $K$ is sufficiently large, so that $\PP(\Dk \bkappa)$ has fixed sign for $\knorm > K$. In particular we choose 
\begin{equation}\label{e:KyP}
  K > \sqrt{\yP},
\end{equation}
so that, in view of the choice~\eqref{e:knorm} for the norm $\knorm$,
we have $\PP(\Dk \bkappa)>0$ for $\knorm > K$.
We then set
\begin{equation*}
  (\Lambda^{-1} b)_k \bydef
  \begin{cases} 
	     0 & \quad\text{for  } \knorm \leq K , \\
		 \frac{1}{\PP(\Dk \bkappa) |\Gacts|} b_k
   & \quad\text{for  }  \knorm > K .
   \end{cases}
\end{equation*}
The diagonal operators $\Lambda$ and $\Lambda^{-1}$ are inverses on $\X^\infty_0$.

To streamline the notation we define the diagonal operator 
\begin{equation}\label{e:orbitcount}
  (\orbitcount b)_k \bydef |\Gacts| b_k ,
\end{equation}
representing multiplication by the number of elements in the group orbit.
With this notation we may write
\begin{equation}\label{e:Fcondensed}
  F(\kappa,b) = \Lambda b + \orbitcount \Phi(b).
\end{equation}

With these preliminaries in place, and writing $x=(\kappa,b)$, we study the   (fixed point) operator 
\begin{equation}\label{e:defT}
  T(x) \bydef  x - \AA \FF(x) ,
\end{equation}
 with the linear operator $\AA$ defined through
\begin{subequations}
\label{e:defAA}
\begin{alignat}{1}
  (\pi_\kappa + \pi_K)  \AA x &=  A (\pi_\kappa + \pi_K)  x , \\
  \pi_\infty \AA x &= \Lambda^{-1} \pi_\infty x .
\end{alignat}
\end{subequations}
We note that $\AA$ is injective on $\CJ \times \tX_0$ whenever the matrix $A$ is invertible.

We work on a product neighbourhood (ball) 
\[
  \BB_{(r_1,r_2)}(x) = \{ (\kappa,b) : 
  |\kappa-\pi_\kappa x|_{\FJ} \leq r_1 , \|b - \pi_b x\|_{\sym} \leq  r_2 \} ,
\]
and we write $r=(r_1,r_2)$. 
\begin{remark}\label{r:kappapositive}
When $|\kappa-\pi_\kappa \bx|_{\FJ} <1$ then $|\kappa_j - \bkappa_j| < \bkappa_j$ for $1 \leq j \leq J$, which implies $\kappa \in \C^J_+$.
\end{remark}
We will mainly use the conjugate symmetric subset of the ball
\[
  \BBsym_{r}(x) \bydef \BB_r(x) \cap \Ssym = \{ \tilde{x} \in \BB_r(x)  : \conj \tilde{x} = \tilde{x} \}.
\]
The natural center of the ball will be either $0$ or a the numerically obtained approximate solution~$\bx$, which in practice we choose in $\Ssym$,
so that $\BBsym_{r}(\bx)$ is nonempty.
In Appendix~\ref{sa:estimates} we will establish estimates of the form
\begin{subequations}\label{e:tenbounds}
\begin{alignat}{2}
  | \pi_\kappa [T (\bx) -\bx] |_{\FJ} &\leq \YY{1} \label{e:boundY1}\\  
  \| \pi_b [T (\bx) -\bx] \|_{\X_0} &\leq \YY{2} \label{e:boundY2}\\
  \|  \pi_\kappa D_\kappa T (\bx)  \|_{B(\FJ,\FJ)}  &\leq \ZY{1}_1  \label{e:boundZ11}\\
  \|  \pi_\kappa D_b T (\bx)  \|_{B(\X_0,\FJ)}  &\leq \ZY{1}_2  \\
  \|  \pi_b D_\kappa T (\bx)  \|_{B(\FJ,\X_0)}  &\leq \ZY{2}_1  \\
  \|  \pi_b D_b T (\bx)  \|_{B(\X_0,\X_0)}  &\leq \ZY{2}_2   \\
  \| \pi_\kappa [D_\kappa T (x)  - D_\kappa T (\bx)] \|_{B(\FJ,\FJ)}  &\leq \WY{1}_{11}r_1 + \WY{1}_{12}r_2  &\qquad& 
  \text{for all } x \in \BBsym_r(\bx)   \label{e:boundW11} \\
  \|  \pi_\kappa [D_b T (x) -D_b T(\bx)]  \|_{B(\X_0,\FJ)}  &\leq \WY{1}_{21}r_1 + \WY{1}_{22}r_2   &\qquad&
  \text{for all } x \in \BBsym_r(\bx) \\
  \| \pi_b [D_\kappa T (x)  - D_\kappa T (\bx)] \|_{B(\FJ,\X_0)}  &\leq \WY{2}_{11}r_1 + \WY{2}_{12}r_2  &\qquad&
  \text{for all } x \in \BBsym_r(\bx) \\
  \| \pi_b [D_b T (x) - D_b T(\bx)]  \|_{B(\X_0,\X_0)}  &\leq 
  \WY{2}_{21}r_1 + \WY{2}_{22}r_2   &\qquad&
  \text{for all } x \in \BBsym_r(\bx)  \label{e:boundW22}. 
\end{alignat}
\end{subequations}
These estimates will be valid for $0 \leq r_1 \leq \rstar_1$
and $0 \leq r_1 \leq \rstar_2$, for some choice of $\rstar = (\rstar_1,\rstar_2) \in \R^2_+$.
Although this restriction on the $r$-values allows us to work with $W^{[i]}_{i',i''} \in \R_+$, the analysis goes through without change when $W^{[i]}_{i',i''}=W^{[i]}_{i',i''}(r)$ provided this dependence on $r$ is nondecreasing in $r_1$ and $r_2$ (which is not very restrictive).

In the following we will often write 
\[ 
 \WY{i}_{i'}(r) \bydef \WY{i}_{i'1} \, r_1+\WY{i}_{i'2} \, r_2 
           \qquad\text{for } i,i'=1,2.
\]
We define the radii polynomials 
\begin{subequations}\label{e:radpols12}
\begin{alignat}{1}
   p_1(r) &\bydef \YY{1} + r_1 [\ZY{1}_1 + \tfrac{1}{2} \WY{1}_1(r) - 1] + r_2 [\ZY{1}_2 + \tfrac{1}{2}\WY{1}_2(r) ] ,\\
   p_2(r) &\bydef \YY{2} + r_1 [\ZY{2}_1 + \tfrac{1}{2}\WY{2}_1(r)] + r_2 [\ZY{2}_2 + \tfrac{1}{2}\WY{2}_2(r) - 1] . 
\end{alignat}
\end{subequations}
Since the dependence of $\W^{[i]}_{i'}$ is polynomial, we call these \emph{radii polynomials}.
As we will see in the proof of Theorem~\ref{thm:contraction}, the operator $T$ maps $\BB_r$ into itself when $p_1(r)$ and $p_2(r)$ are both negative. 
Contractivity of $T$ on $\BB_r$ is controled by the $2\times 2$ matrix
\begin{equation}\label{e:defM}
  \M(r) \bydef \left[ 
   \begin{array}{cc}
	   \ZY{1}_1 + \WY{1}_1(r)  & \ZY{1}_2 + \WY{1}_2(r)  \\[1mm]
	   \ZY{2}_1 + \WY{2}_1(r)  & \ZY{2}_2 + \WY{2}_2(r)  
   \end{array}
   \right] .
\end{equation}
It will become apparent in the proof of Theorem~\ref{thm:contraction} that we need the dominant (Perron-Frobenius) eigenvalue 
of this positive matrix to be less than $1$.
We express this eigenvalue in terms of the trace and determinant of $\M(r)$:
\begin{equation}\label{e:tracedet}
  \sigma(r) \bydef \frac{\tr \M(r) + \sqrt{(\tr \M(r))^2 - 4 \det \M(r) }}{2}.
\end{equation}
The demand $\sigma(r)<1$ is equivalent to
requiring that the polynomials
\begin{subequations}\label{e:radpols34}
\begin{alignat}{1}
   \tp_3(r) &\bydef \tr \M(r) -2 ,\\
   \tp_4(r) &\bydef \tr \M(r) - \det \M(r)  - 1 ,
\end{alignat}
\end{subequations}
are negative.
Hence we aim to find $0<\hr \leq \rstar$, with inequalities interpreted component-wise, such that 
\begin{subequations}
\label{e:pleq0}
\begin{alignat}{2}
  p_1(\hr_1, \hr_2) &<0, \qquad& 
  p_2(\hr_1, \hr_2) &<0, \label{e:p12ll0}\\
  \tp_3(\hr_1, \hr_2) &<0, \qquad&
  \tp_4(\hr_1, \hr_2) &<0. \label{e:tp34ll0}
\end{alignat}
\end{subequations}

\begin{remark}\label{r:minimax}
For positive $2 \times 2$ matrices $\M$ one may express the demand that the dominant eigenvalue is less than $1$ in terms of the trace and the determinant. 
More generally, consider a positive $m \times m$ matrix $\M$.
The Collatz-Wielandt formula provides the minimax characterization
\begin{equation}\label{e:minimax}
  \sigma = \min_{y \in \R^m_+} \max_{1\leq i \leq m} \frac{(\M y)_i}{y_i}
\end{equation}
for the dominant eigenvalue.
In particular, when the dimension of the matrix is larger than $2$ and an explicit expression like~\eqref{e:tracedet} is not available, it is not difficult to obtain rigorous upper bounds on $\sigma(r)$ by choosing appropriate test vectors $y\in \R^m_+$.
\end{remark}

\begin{theorem}\label{thm:contraction}
Assume $\bx \in \Ssym_+$. 
Assume the matrix $A$ is invertible and $\conj A = A \conj$. 
Let $\rstar \in \R^2_+$ with $\rstar_1 < 1$.
Let $\YY{i} , \ZY{i}_{i'}, \WY{i}_{i'i''} \in \R_+$ satisfy the bounds~\eqref{e:tenbounds} for $i,i',i''=1,2$ and $0 \leq r \leq \rstar$. 
Let $p_1$, $p_2$ and $\tp_3$, $\tp_4$ be as defined in~\eqref{e:radpols12}
and ~\eqref{e:radpols34}, respectively.
If $\hr=(\hr_1,\hr_2)>0$ is  
such that $\hr \leq \rstar$
and inequalities~\eqref{e:pleq0} are satisfied, 
then
\begin{thmlist}
\item
$\FF$ has a unique zero $\hx=(\hkappa,\hb)$ in $\BBsym_{\hr}(\bx)$; 
\item 
$\hb$ corresponds to a $\G$-invariant real-valued periodic solution 
\[
  \hu(\x) = \sum_{k \in \Z^3} (\sigma(m\e_0+\hb))_k \exp(\textup{i} \L_{\hkappa} k \cdot \x)
\]
of the PDE~\eqref{eq:OK}; 
\item 
the domain of periodicity $\Omega_{\hkappa}$, given by~\eqref{e:omegakappa}, is
such that the energy is stationary with respect to variations in its
($\G$-symmetry respecting) proportions. 
\end{thmlist}
\end{theorem}
\begin{proof}
We observe that $T$ preserves conjugate symmetry, i.e., it maps $\Ssym$ to itself. Indeed, it follows from Remark~\ref{r:extendconj} combined with Equations~\eqref{e:hfconjequivariant} and~\eqref{e:defFF} that  $\FF(\conj x) = \conj \FF(x)$ for any $x \in \FJ_+ \times \X_0$.
Since $A \conj = \conj A$ by assumption, we conclude from~\eqref{e:defAA} that $\AA\conj=\conj\AA$.
This implies that $T(\conj x) = \conj T(x)$, hence $T$ restricts to a map from $\Ssym_+$ to $\Ssym$.

Next we show that $T$ maps $\BBsym_{\hr}(\bx)$ to itself.
Fix any $x = (\kappa,b) \in \Ssym$ such that $|\kappa- \bkappa |_{\FJ} \leq \hr_1$ and $\|b- \bb \|_{\X_0} \leq \hr_1$.
Then 
\[
  |\pi_\kappa T(x) - \bkappa |_{\FJ}
  \leq
    |\pi_\kappa [T(x) - T(\bx)]  |_{\FJ}  + | \pi_\kappa[T(\bx) - \bx] |_{\FJ}
	\leq |\pi_\kappa [T(x) - T(\bx)]  |_{\FJ} +    \YY{1}.
\]
We estimate
\begin{alignat*}{1}
 |\pi_\kappa [T(x) - T(\bx)] |_{\FJ} &= \left|
 \int_0^1 \left\{ \pi_\kappa D_\kappa T((1-s)x+s\bx) (\kappa-\bkappa)  + 
 \pi_\kappa D_b T((1-s)x+s\bx) (b-\bb) \right\} ds 
  \right|_{\FJ} \\ 
 & \leq |\kappa-\bkappa|_{\FJ}  \int_0^1 \| \pi_\kappa D_\kappa T((1-s)x+s\bx) \|_{B(\FJ,\FJ)} ds   \\
 & \hspace*{1cm} +
 \| b-\bb\|_{\X_0}
 \int_0^1 \| \pi_\kappa D_b T((1-s)x+s\bx) \|_{B(\X_0,\FJ)} ds \, .
\end{alignat*}
Since $\Ssym$ is convex and $\WY{1}_1 (r)$ is linear in $r$,  we obtain
from~\eqref{e:boundZ11} and~\eqref{e:boundW11}
\begin{alignat*}{1}
\| \pi_\kappa D_\kappa T((1-s)x+s\bx) \|_{B(\FJ,\FJ)} 
& = \| \pi_\kappa D_\kappa T(\bx) + D_\kappa   T(\bx+s(x-\bx)) - D_\kappa T(\bx) \|_{B(\FJ,\FJ)} \\
& \leq \| \pi_\kappa D_\kappa T(\bx) \|_{B(\FJ,\FJ)}  +  \|D_\kappa   T(\bx+s(x-\bx)) - D_\kappa T(\bx) \|_{B(\FJ,\FJ)} \\
& \leq \ZY{1}_1 +  \WY{1}_1(s\hr_1,s\hr_2)
= \ZY{1}_1 +  s\WY{1}_1(\hr_1,\hr_2).
\end{alignat*}
Similarly, we obtain the bound
\[
\| \pi_\kappa D_b T((1-s)x+s\bx) \|_{B(\X_0,\FJ)} \leq 
 \ZY{1}_2 + s \WY{1}_2(\hr_1,\hr_2). 
\]
We combine the above estimates into
\[
 |\pi_\kappa T(x) - \bkappa |_{\FJ} \leq
  \YY{1} + \hr_1 \int_0^1 \ZY{1}_1 + s \WY{1}_1(\hr)  ds +  
  \hr_2  \int_0^1 \ZY{1}_2 + s \WY{1}_2(\hr)  ds 
  = p_1(\hr) + \hr_1  < \hr_1.
\]
An entirely analogous argument leads to the estimate
\[
 \|\pi_b T(x) - \bb \|_{\X_0} \leq 
   p_2(\hr) + \hr_2  < \hr_2.
\]
We conclude that $T$ maps $\BBsym_{\hr}(\bx)$ to itself.

Next we want to establish that $T$ is contractive on $\BBsym_{\hr}(\bx)$,
hence we need to choose a norm on $\FJ \times \X_0$.
Consider the dominant eigenvalue $\sigma(\hr)$ of $\M(\hr)$,
and let $\rho=[\rho_1,\rho_2]^T \in \R^2_+$ be an associated eigenvector.
We choose as the nor on the product space $\FJ \times \X_0$
\[
  \| (\kappa , b) \|_{\rho} =   \max \left\{ \frac{|\kappa|_{\CJ}}{\rho_1}, \frac{\| b \|_{\X_0}}{\rho_2}   \right\} .
\]
We denote $\FJ \times \X_0$ interpreted as a Banach space with norm $\|x\|_\rho$ by $X_\rho$. 
We now fix any $x\in \BBsym_{\hr}(\bx)$ and write
 the linear operator
$DT(x) \in B(X_\rho,X_\rho)$ in terms of a $2 \times 2$ block structure as in Remark~\ref{r:matrixmatrix}.
Each of the block operators is estimated using~\eqref{e:boundZ11}--\eqref{e:boundW22} to arrive at the matrix $\M(r)$ in~\eqref{e:defM}.
Indeed, the minimax characterization~\eqref{e:minimax} of the dominant eigenvalue of $\M(r)$ is the motivation for choosing the norm $\|x\|_\rho$ on the product space.
Hence we obtain 
\begin{alignat*}{1}
   \| DT(x) \|_{B(X_\rho,X_\rho)} & =    
   \max\left\{ 
   \frac{\rho_1 \| \pi_\kappa D_\kappa T(x)\|_{B(\FJ,\FJ)} +
   \rho_2 \| \pi_\kappa  D_\kappa T (x)\|_{B(\X_0,\FJ)} }{\rho_1} , \right.\\
   & \hspace*{4cm} \left.
   \frac{\rho_1 \| \pi_b D_\kappa T(x)\|_{B(\FJ,\X_0)} +
    \rho_2 \| \pi_b D_\kappa T(x) \|_{B(\X_0,\X_0)} }{\rho_2}
    \right\} ,
\end{alignat*}
which we estimate, by using the triangle inequality and~\eqref{e:boundZ11}--\eqref{e:boundW22}, 
by
\begin{alignat}{1}
   \| DT (x)\|_{B(X_\rho,X_\rho)} &    
   \leq 
   \max\left\{ 
   \frac{\rho_1 (\ZY{1}_1+\WY{1}_1 (\hr)) +
   \rho_2 (\ZY{1}_2+\WY{1}_2 (\hr))}{\rho_1} , 
   \right. \nonumber \\ & \hspace*{4cm} \left.
   \frac{\rho_1 (\ZY{2}_1+\WY{2}_1 (\hr)) +
    \rho_2 (\ZY{2}_2+\WY{2}_2 (\hr)) }{\rho_2}
    \right\} . \nonumber
   \\
   & = \sigma(\hr) <1.  \label{e:lessthan1}
\end{alignat}
Since the estimate is uniform for $x \in \BBsym_{\hr}(\bx)$
we conclude that $T$ is contractive on $\BBsym_{\hr}(\bx)$
for the norm $\|x\|_\rho$.

From the above and the Banach contraction mapping theorem 
we conclude that $T$ has a unqiue fixed point $\hx$ in $\BBsym_{\hr}(\bx)$.
Since the matrix $A$ is invertible by assumption, it follows from~\eqref{e:defAA} that $\AA$ is invertible, hence injective.
We conclude that 
$\hx$ is the unique zero of $\FF$ in $\BBsym_{\hr}(\bx)$. This finishes the proof of part (a).

Part (b) follows from Theorem~\ref{thm:solvePDE}. We note that it follows from part (a) that $\hkappa$ is real-valued, whereas the assumption that $\hr_1 \leq \rstar_1<1$ implies that $\hkappa_j>0$ for all $1\leq j \leq J$ by Remark
~\ref{r:kappapositive}.

Part (c) follows from $H_j(\hkappa,\hb)=h_j(\hkappa,\sigma(m\e_0+\hb))=0$ 
and Equation~\eqref{e:dEdkappa}.
\end{proof}

\begin{remark}\label{r:IADF}
It follows from~\eqref{e:lessthan1} that $\| I - \AA D \FF(\hx) \|_{B(X_\rho,X_\rho)} <1$, hence $\AA D \FF(\hx)$ is invertible. 
It is immediate that $D \FF(\hx)$ is injective and $\AA$ is surjective.  
Moreover, it is not difficult to see that both $\AA:\X_0 \to \tX_0$ and $D \FF (\hx): \X_0 \to \tX_0$ are Fredholm operators of index $0$. We conclude that both $D \FF(\hx)$ and $\AA$ are invertibility. In particular, it is not necessary to assume in Theorem~\ref{thm:contraction} that the matrix $A$ is invertible, as this is implied by the inequalities~\eqref{e:pleq0}.
\end{remark}

Theorem~\ref{thm:contraction} is a variation on a general strategy.
We do not review the relevant literature here, but refer instead to~\cite{ariolikochintegration,JBJPNotices,Nakao2001survey,Plum2001survey,Gomez-Serrano} and then references therein.
We merely note her that some of the notation used here was introduced in~\cite{Yamamoto}, while the idea to collect the necesary inequalities in radii polynomials stems from~\cite{DLM}.
There are some ingredients to the theorem that are less standard. First, we
restrict attention to the conjugate symmetric set, which simplifies the
estimates. Second, following~\cite{JBAMS} (see also \cite{JBJonathan}) we do
not choose a norm on the product space $\FJ \times \X_0$ a priori. Indeed, we
keep track of the different components of the derivatives when performing the
estimates. This refinement leads to the two polynomials $p_1(r_1,r_2)$ and
$p_2(r_1,r_2)$ depending on two variables rather than a single polynomial in a
single variable. Moreover, it allows us to separate the invariance of
$\BB_{\hr}(\bx)$ from contractivity. The norm chosen for contractivity has been
optimized (in the sense explained in Remark~\ref{e:minimax}), which improves on
the arguments in~\cite{JBAMS}. An alternative choice is to work with a weighted $1$-norm on the product space, see e.g.~\cite{BredenKuehn}. We refer to~\cite{vdBW2} for a similar
argument in the context of continuation. As in~\cite{NS} the splitting of the
Jacobian into a $Z$-part and a $W$-part leads to a factor $\frac{1}{2}$ in the
$W$-terms in $p_1$ and $p_2$, at the expense (see Remark~\ref{r:simpler1}) of needing to
check the additional inequalities~\ref{e:tp34ll0}. Finally, we note that
instead of the fixed point map $T$ defined in~\eqref{e:defT}, one can alternatively
derive a similar result by considering the map $x \to x-D\FF(\bx)^{-1}\FF(x)$,
cf.~\cite{NS}, although the conditions~\eqref{e:pleq0} would need to be adapted slightly (the
main difference compared to~\cite{NS} is that we are dealing with a product
system here).

There are (at least) two ways to simplify the conditions~\ref{e:pleq0}, as outlined in the next remarks.
\begin{remark}\label{r:simpler1}
One may simplify the conditions~\eqref{e:pleq0} to the \emph{stronger} requirements
\begin{alignat*}{1}
\YY{1} + \hr_1 [\ZY{1}_1 + \WY{1}_1(\hr) - 1] + r_2 [\ZY{12} + \WY{1}_2(\hr) & < 0 \\
\YY{2} + \hr_1 [\ZY{2}_1 + \WY{2}_1(\hr)] + r_2 [\ZY{22} +\WY{2}_2(\hr) - 1] & < 0. 
\end{alignat*}	
Indeed, when these inequalities are satisfied the estimate $\sigma(\M(\hr))<1$ on the dominant eigenvalue is obtained by using the test vector $y=(\hr_1,\hr_2)^T$
in~\eqref{e:minimax}. This corresponds to the approach taken in~\cite{JBAMS}.
\end{remark}

\begin{remark}
The sets $\varOmega_i \bydef \{ r \in \R^2_+ : p_i(r)<0 \}$ are convex for $i=1,2$. 
Furthermore, if $(r_1,r_2) \in \varOmega_1$, then $(r_1,\tilde{r}_2) \in \varOmega_1$ for all $0 < \tilde{r}_2 < r_2$, and similarly for $\varOmega_2$ (with the roles of $r_1$ and $r_2$ exchanged).
We conclude that if $r, r' \in \varOmega_1 \cap \varOmega_2$, which we assume to be nonempty, then so is
$(\min\{r_1,r'_1\},\min\{r_2,r'_2\})$. 
It follows that there is a \emph{minimal radius vector} $r_{\min} \in
\overline{\varOmega_1 \cap \varOmega_2}$, i.e., the point in the latter set which is closest to the origin. Clearly $p_1(r_{\min})=p_2(r_{\min})$. It follows that there exists a
direction $\hat{y} \in \R^2_+$ pointing into $\varOmega_1\cap\varOmega_2$ from $r_{\min}$, i.e., $\nabla p_i(r_{\min}) \cdot \hat{y} <0$
for $i=1,2$.

Now assume that $\WY{1}_{12}=\WY{1}_{21}$ and $\WY{2}_{12}=\WY{2}_{21}$, which
is natural when the corresponding estimates are derived by bounding the second
derivatives $\pi_\kappa D_b D_\kappa T(x)$ and $\pi_b D_b D_\kappa T(x)$,
respectively. Then the estimate $\sigma(\M(\hr))<1$ is obtained by using the
test vector $y=\hat{y}$ in the minimax charaterization~\eqref{e:minimax}.
When the symmetry condition
$\WY{i}_{12}=\WY{i}_{21}$ is satisfied, the inequalities~\eqref{e:tp34ll0} are thus implied
by~\eqref{e:p12ll0} for $r \in \varOmega_1\cap\varOmega_2$ sufficiently close to $r_{\min}$. Hence, in that case finding an $\hr$ such that the inequalities~\eqref{e:p12ll0} hold suffices to prove existence of a solution in $\BBsym_{\hr}(\bx)$.
\end{remark}

We confined evaluation of the derivatives in~\eqref{e:tenbounds} to points in the symmetric set $\BBsym_r(\bx)$, mainly because only real-valued $\kappa$ are physically relevant and this avoids some additional estimates. A drawback is that we require $A$ to be conjugate symmetry preserving.
Indeed, both the assumptions on~$A$ about invertibility and conjugate symmetry preservation in Theorem~\ref{thm:contraction} need to be verified. In particular, while in some other papers the invertibility check is implied by negativity of the radii polynomials, this is not the case for the particular (somewhat weaker) formulation of the bounds~\eqref{e:tenbounds} chosen here. Furthermore, since $A$ only approximates the inverse of the Jacobian, unless we take extra care it would satisfy the conjugate equivariance $\conj A = A \conj$ only approximately. 
\begin{remark}\label{r:tweakA}
In practice we verify the assumptions on $A$ in
Theorem~\ref{thm:contraction} as follows.
After computing a numerical inverse $A$ of the
approximate Jacobian $M$ we first replace it by $A \mapsto \frac{1}{2} (A+ \overline{A}^T)$ to make it selfadjoint (need in the Morse index computation, see Section~\ref{s:stability}), and then replace the resulting matrix by $A \mapsto \frac{1}{2} (A+ \conj A \conj)$ 
to ensure that it preserves conjugation symmetry (implemented using the explicit matrix representation~$\conj A \conj = \overline{\J}\;\! \overline{A}\;\! \J$, with $\J$ given in \eqref{e:defJ}). 
\end{remark}

\begin{remark}\label{r:errorcontrol}
We have explicit error control. Let
\[
  \bu(\x) \bydef \sum_{k \in \Z^3,\knorm\leq K} (\sigma(m\e_0+\bb))_k \exp(\textup{i} \L_{\bkappa} k \cdot \x)
\]
be the approximate periodic solution.
Then it is a corollary of Theorem~\ref{thm:contraction} that 
\[
	\bigl\|\hu (\x) - \bu\bigl( (\L_{\bkappa}^T)^{-1}  \L_{\hkappa}^T  \x\bigr) \bigr\|_\infty \leq \hr_2 ,
\]
where the rescaling of $\x$ is necessary to correct for dephasing. 
Moreover, $|\hkappa_j - \bkappa_j| \leq \bkappa_j \hr_1 $
for  $1\leq j \leq J$.
\end{remark}

The operator norms appearing in~\eqref{e:tenbounds}
are given explicitly by
\begin{alignat}{1}
  \|Q\|_{B(\FJ,\FJ)} &= 
  \max_{1\leq j \leq J} \bkappa_j^{-1} \sum_{j'=1}^J | Q_{j j'}| \bkappa_{j'} \,,
  \label{e:norm11}\\	
  \|Q\|_{B(\X_0,\FJ)} &= 
  \sup_{k\in\ZZ_0} \omega_{k}^{-1} \max_{1\leq j \leq J} \bkappa_j^{-1} |Q_{jk}| \,,
  \label{e:norm12} \\	
  \|Q\|_{B(\FJ,\X_0)} &=  
  \max_{\epsilon \in \{-1,1\}^J } \sum_{k\in\ZZ_0} \omega_k
  	 	\biggl| \sum_{j=1}^J  Q_{k j} \epsilon_j \bkappa_j \biggr| \,,
  \label{e:norm21}\\	
  \|Q\|_{B(\X_0,\X_0)} &= 
  \sup_{k\in\ZZ_0} \omega_{k}^{-1} \sum_{k'\in\ZZ_0} \omega_{k'} |Q_{k'k}| \,.
  \label{e:norm22}	
\end{alignat}
When the operators reduce to finite matrices (e.g.~$A_{ii'}$, $i,i'=1,2$)
then the suprema reduce to maxima and the infinite sums reduce to finite ones,  hence in that case these expressions can be evaluated explicitly.

For future use we note that, for any Banach space $X'$ (in particular  $X'=\FJ$ and $X'=\X_0$)
\begin{alignat}{1}
\| Q \|_{B(\X_0,X')} &= 
\sup_{k\in\ZZ_0} \omega_k^{-1} \| Q \e_k\|_{X'} \label{e:l1sup} \\ 
 &= \max\bigl\{ \| Q \pi_{K'} \|_{B(\X_0,X')}, \| Q(I-\pi_{K'})\|_{B(\X_0,X')} \bigr\}
\qquad \text{for any } K'>0.  \label{e:splitl1norm}
\end{alignat}
For $X'=\C$ the characterization~\eqref{e:l1sup} may also be expressed as
\begin{equation}
\label{e:dualest}
\biggl| \sum_{k \in \ZZ_0} q_k a_k \biggr| \leq \|q\|_{\X_0}^* \|a\|_{\X_0}
\qquad\text{with } \|q\|_{\X_0}^*=\sup_{k\in\ZZ_0} \omega_k^{-1} |q_k|.
\end{equation}

\begin{remark}
In addition to~\eqref{e:KyP}, in Appendix~\ref{sa:estimates}, where to derive explicit bounds, we will put additional restrictions on the computational parameter $K$, namely
\begin{equation}\label{e:Krestrictions}
  K^2 > \max\{ y_P , \gamma (1-r_1^*)^{-1} , (1-r_1^*)^{-1}, 2  \}.
\end{equation}
\end{remark}

%% file: energy.tex
%!TEX root = ms.tex

In this section we derive estimates that link the energy $E(\hkappa,\sigma(m\e_0+\hb))$ of the stationary point $\hx=(\hkappa,\hb)$, which is obtained in Theorem~\ref{thm:contraction} to the energy of the numerical approximation~$\bx=(\bkappa,\bb)$. Since $\bb$ represents finite many nontrivial Fourier modes, the energy $E(\bkappa,\sigma(m\e_0+\bb))$ can be computed using interval arithmetic.
We homotope between $\hx$ and $\bx$ by setting
\[
  \E(s) \bydef E((1-s)\hkappa+s\bkappa,\sigma(m\e_0+(1-s)\hb+s\bb))
  \qquad \text{for } s\in[0,1].
\]
Since $\hx$ is a critical point of $E$ (see Remark~\ref{r:dEdc} and Equation~\eqref{e:dEdkappa}), we have $\E'(0)$.
In the remainder of this section we will derive a bound
\[
  \E^{[2]} \geq \max_{s \in [0,1] } \bigl|\E''(s)\bigr|.
\]
Taylor's theorem then implies the energy error bound
\begin{equation}\label{e:enclosedenergy}
  \bigl| E(\hkappa,\sigma(m\e_0+\hb)) - E(\bkappa,\sigma(m\e_0+\bb)) \bigr| = \bigl| \E(0)-\E(1)\bigr| \leq \frac{1}{2} \E^{[2]}. 
\end{equation}

To express $\E''(s)$ conveniently we introduce 
\begin{alignat*}{2}
   \kappa_s & \bydef (1-s) \hkappa + s \bkappa  &\qquad&\text{for } s\in [0,1], \\ 
   c_s & = m\e_0 + (1-s) \sigma(\hb) + s \sigma(\bb)  &\qquad&\text{for } s\in [0,1] .
\end{alignat*}
and
\begin{alignat*}{1}
 \tkappa &\bydef \hkappa-\bkappa \\
  \tc &\bydef \sigma(\hb-\bb) = \sigma(\hb)-\sigma(\bb) ,
\end{alignat*}
where we have used linearity of $\sigma$.
Furthermore, $\tc \in \Xsym$ and $\tc_0=0$.
Since $\hx \in \BBsym_{\hr}(\bx)$ we have $|\tkappa|_{\CJ} \leq \hr_1$,
as well as  $\| \tc \|_{\nu} \leq  \hr_2$
and hence $|\tc_k| \leq \hr_2 \omega_k^{-1}$ by~\eqref{e:pointwise}.
moreover, $\hx, \bx  \in \Ssym_+$ implies that
$\kappa_s \in \RJ_+$ and $\tkappa \in \RJ$, as well as $\fullconj c_s = c_s$ and  $\fullconj \tc = \tc$.

With this notation in place we write
\begin{equation}\label{e:d2Eds2}
  \E''(s)  = D^2_\kappa E(\kappa_s,c_s) [\tkappa,\tkappa] +  2  D_\kappa D_c E(\kappa_s,c_s) [\tkappa,\tc]  + D^2_c E(\kappa_s,c_s) [\tc,\tc] .
\end{equation}
We estimate each of these terms separately.
The main ideas of this analysis are similar to Appendix~\ref{sa:estimates}.
In all these estimates we fix $s\in [0,1]$ arbitrary.
We begin with the first term in~\eqref{e:d2Eds2}:
\begin{alignat*}{1}
D^2_\kappa E(\kappa_s,c_s) [\tkappa,\tkappa] 
& = \frac{1}{2} \sum_{k\in \Z^3_0} P''(\Dk \kappa_s) \, 
(\Dk \tkappa)^2 \,  (c_s)_k (c_s)_{-k}. 
\end{alignat*}
We write $c_s = \sigma(\bb) + \hr_2 c$ for some $c \in \Xsym$ with $\|c\|_\nu \leq 1$ (and hence $|c_k| \leq \omega_k^{-1}$).
We then expand $(c_s)_k (c_s)_{-k}$ and write
\begin{alignat}{1}
\frac{1}{2}  \sum_{k\in \Z^3_0} P''(\Dk \kappa_s) \, 
(\Dk \tkappa)^2 \,  (c_s)_k (c_s)_{-k} & =
 \frac{1}{2} \sum_{k\in \Z^3_0} P''(\Dk \kappa_s) \, 
(\Dk \tkappa)^2 \,  \sigma(\bb)_k \sigma(\bb)_{-k} \nonumber \\
&\hspace*{1cm} + 
 \frac{\hr_2}{2} \sum_{k\in \Z^3_0} P''(\Dk \kappa_s) \, 
(\Dk \tkappa)^2 \,  [\sigma(\bb)_k c_{-k} + \sigma(\bb)_{-k} c_{k}] \nonumber \\
&\hspace*{2cm}  +
\frac{\hr_2^2}{2} \sum_{k\in \Z^3_0} P''(\Dk \kappa_s) \, 
(\Dk \tkappa)^2 \,  c_k c_{-k} .  \label{e:D2kappaE}
\end{alignat}
We note that 
\begin{equation}\label{e:Dktkappa}
  |\Dk \tkappa| \leq |\Dk|_{\RJ}^* |\tkappa|_{\RJ} \leq \hr_1 \Dk \bkappa.
\end{equation}
The first term in~\eqref{e:D2kappaE} we then estimate by computing the finite interval arithmetic sum (using Lemma~\ref{l:symquad})
\[
\frac{1}{2} \Bigl| \sum_{k\in \Z^3_0} P''(\Dk \kappa_s) \, 
(\Dk \tkappa)^2 \,  \sigma(\bb)_k \sigma(\bb)_{-k}  \Bigr|
\leq \PHI{1} \bydef
\frac{\hr_1^2}{2}  \sum_{k \in \ZZ^K_0}
 \bigl|\PPP^{[2]}_k(\hr_1)\bigr|  (\Dk \bkappa)^2  |\Gacts| \,
  |\bb_k| |\sigma(\bb)_{-k}|  ,
\] 
wher $\PPP^{[2]}_k(\hr_1)$ is defined in~\eqref{e:PPP2}, replacing $\rstar_1$ by $\hr_1$.
For the second term in~\eqref{e:D2kappaE} we rearrange terms (sending $k \to -k$ in one of them), and use Lemma~\ref{l:symquad} and~\eqref{e:dualest} to estimate 
\begin{alignat*}{1}
	 \frac{\hr_2}{2} \Bigl| \sum_{k\in \Z^3_0} P''(\Dk \kappa_s) \, 
	(\Dk \tkappa)^2 \,  [\sigma(\bb)_k c_{-k} + \sigma(\bb)_{-k} c_{k}]  \Bigr|
	& =
	  \hr_2 \Bigl|  \sum_{k\in \ZZ^K_0} P''(\Dk \kappa_s) \, 
	(\Dk \tkappa)^2 \, |\Gacts| \, \sigma(\bb)_k c_{-k} \Bigr| \\
&\leq  \PHI{2}  \bydef \hr_1^2 \hr_2
	 \max_{k \in \ZZ^K_0} |\PPP^{[2]}_k(\hr_1)\bigr|  (\Dk \bkappa)^2 \, 
	  |\bb_k | \, \nu^{-\knorm} ,
\end{alignat*}
since $|c_{-k}| \leq |\Gacts|^{-1} \nu^{-\knorm}$.
For the third term in~\eqref{e:D2kappaE} we estimate, again using~\eqref{e:Dktkappa},
\begin{alignat*}{1}
\frac{\hr_2^2}{2} \Bigl|  \sum_{k\in \Z^3_0} P''(\Dk \kappa_s) \, 
(\Dk \tkappa)^2 \,  c_k c_{-k} \Bigr|   &\leq 
\frac{\hr_1^2 \hr_2^2}{2}  \max_{k\in \Z^3_0} |P''(\Dk \kappa_s)| \, 
(\Dk \bkappa)^2 \, \omega_k^{-2}
  \\  & \leq   \PHI{3} \bydef \frac{\hr_1^2 \hr_2^2}{2} 
  \max\left\{ 
    \max_{k \in  \ZZ^K_0} | \PPP^{[2]}_k(\hr_1)| \, (\Dk \bkappa)^2 \, \omega_k^{-2} 
   \, , \,  \EE^{[2]} (K,\nu,\hr_1)
   \right\}
\end{alignat*}
where in the final inequality we have used the bounds $\omega_k \geq \nu^{\knorm}$ and Lemma~\ref{l:Eterm}.

Next, we estimate the second term in~\eqref{e:d2Eds2}:
\begin{alignat*}{1}
D_\kappa D_c E(\kappa_s,c_s) [\tkappa,\tc] 
& = \frac{1}{2} \sum_{k\in \Z^3_0} P'(\Dk \kappa_s) \, 
(\Dk \tkappa) \,
[(\tc)_k (c_s)_{-k} + (c_s)_{k} \tc_{-k} \\
& = \sum_{k\in \Z^3_0} P'(\Dk \kappa_s) \, 
(\Dk \tkappa) \,  (c_s)_{-k} \tc_{k}. 
\end{alignat*}
Using the notation as above, we split this into two terms:
\begin{alignat}{1}
\sum_{k\in \Z^3_0} P'(\Dk \kappa_s) \, 
(\Dk \tkappa) \,  (c_s)_{-k} \tc_{k}
& \nonumber\\
& \hspace*{-1cm}
= \sum_{k\in \Z^3_0} P'(\Dk \kappa_s) \, 
(\Dk \tkappa) \,  \sigma(\bb)_{-k} \tc_{k}
 +
\hr_2 \sum_{k\in \Z^3_0} P'(\Dk \kappa_s) \, 
(\Dk \tkappa) \,  c_{-k} \tc_{k} . \label{e:D2mixE}
\end{alignat}
We estimate the first term in the righthand side of~\eqref{e:D2mixE} by
\begin{alignat*}{1}
 \biggl|  \sum_{k\in \Z^3_0} P'(\Dk \kappa_s) \, 
                      (\Dk \tkappa) \,  \sigma(\bb)_{-k} \tc_{k} \biggr|
  & =
  \biggl| \sum_{k\in \ZZ_0} P'(\Dk \kappa_s) \, 
                 (\Dk \tkappa) \, |\G.k| \, \sigma(\bb)_{-k} \tc_{k} \biggr| 
  \\ & \leq 
  \max_{k \in \ZZ_0^K} 
  |P'(\Dk \kappa_s)| \, ( \Dk \tkappa) \, |\G.k| \, | \sigma(\bb)_{-k} | \, \hr_2 \omega_k^{-1}
  \\&\leq
  \PHI{4} \bydef \hr_1 \hr_2 \max_{k \in \ZZ^K_0}
  | \PPP^{[1]}_k(\hr_1)|  \, (\Dk \bkappa) \, |\sigma(\bb)_{-k}| \, \nu^{-\knorm}.
\end{alignat*}
where we have used Lemma~\ref{l:symquad} to justify the identity and we have applied~\eqref{e:dualest} to establish the first inequality,
and~\eqref{e:PPP1} for the second one.
The second term in the righthand side of~\eqref{e:D2mixE}
can be estimated similarly by 
\begin{alignat}{1}
 \hr_2 \biggl| \sum_{k\in \Z^3_0} P'(\Dk \kappa_s) \, 
(\Dk \tkappa) \,  c_{-k} \tc_{k} \biggr|
& = 
\hr_2 \biggl| \sum_{k\in \ZZ_0} P'(\Dk \kappa_s) \, 
(\Dk \tkappa) \, |\G.k| \, c_{-k} \tc_{k} \biggr| \nonumber \\
&  \leq 
\hr_1 \hr_2^2 \sup_{k \in \ZZ_0}
| \PPP^{[1]}_k(\hr_1)| \, (\Dk \bkappa) \, |\G.k| \omega_k^{-2},
\label{e:supinsecondterm}
\end{alignat}
where we have also used that $|c_{-k}| \leq \omega_k^{-1}$.
The righthand side in~\eqref{e:supinsecondterm} is then estimated, analogously 
to above, but now including a tail term, by
\begin{equation}\label{e:phi5}
\PHI{5} \bydef \hr_1 \hr_2^2 
  \max\left\{ 
    \max_{k \in  \ZZ^K_0} | \PPP^{[1]}_k(\hr_1)| \, (\Dk \bkappa) \, |\G.k| \, \omega_k^{-2} 
   \, , \,  \EE^{[1]} (K,\nu)
   \right\} ,
\end{equation}
where we have used $\omega_k^{-1} \leq |\G.k|^{-1}$ and  Lemma~\ref{l:Eterm} and to control the terms in~\eqref{e:supinsecondterm} with $\knorm > K$.

Finally, we estimate the third term in~\eqref{e:d2Eds2}:
\begin{equation}\label{e:D2cE}
D^2_c E(\kappa_s,c_s) [\tc,\tc] 
=  \sum_{k\in\Z^3_0} P(\Dk) \tc_k \tc_{-k} 
     + 3 \conv{ \sigma(m\e_0+ b_s)^2 \, \tc^2 }_0.
\end{equation}
We estimate each term separately.
The first term in the righthand side of~\eqref{e:D2cE}
we estimate analogously to~\eqref{e:supinsecondterm} above: 
\begin{equation*}
 \Bigl| \sum_{k\in\Z^3_0} P(\Dk \kappa_s) \tc_k \tc_{-k}  \Bigr|
  \leq  \hr_2^2  \sup_{k\in\ZZ_0} |P(\Dk \kappa_s)| 
  									\, |\G.k| \, \omega_k^{-2}.
\end{equation*}
For $\knorm > K$ we bound $ \omega_k^{-2} \leq \nu^{-2\knorm}$
and use the estimate provided by Lemma~\ref{l:Eterm}.
Similarly to~(\eqref{e:phi5}) we obtain
\[
  \Bigl| \sum_{k\in\Z^3_0} P(\Dk \kappa_s) \tc_k \tc_{-k}  \Bigr|
  \leq \PHI{6} \bydef  \hr_2^2 
  \max\left\{ 
    \max_{k \in  \ZZ^K_0} | \PPP^{[0]}_k(\hr_1)| \,|\G.k|\, \omega_k^{-2} 
   \, , \,  \EE^{[0]} (K,\nu,\hr_1)
   \right\} .
\]
To estimate the second term in righthand side of~\eqref{e:D2cE} we write
\[
   \sigma(m\e_0+ b_s) = \sigma(m\e_0 + \bb) + \hr_2 c
\] 
for some $c \in \Xsym$ with $\|c\|_\nu \leq 1$.
We amply the mean value theorem to obtain
\[
   3\conv{ \sigma(m\e_0+ b_s)^2 \, \tc^2 }_0  
   = 3\conv{\sigma(m\e_0 + \bb)^2 \tc^2 }_0 + 6
   \hr_2 \conv{(\sigma(m\e_0 + \bb) + r c) \tc^2 }_0 .
\]
for some $r \in (0,\hr_2)$.
Using the Banach algebra property we estimate (for $r \in (0,\hr_2)$)
\begin{alignat*}{1}
  3 \bigl|\conv{ \sigma(m\e_0+ \bb)^2 \, \tc^2 }_0 \bigr|
  & \leq \PHI{7} \bydef 3 \hr_2^2  \| \conv{\sigma(m\e_0 + \bb)^2 } \|_{\nu} ,\\
  3 \bigl|\conv{ (\sigma(m\e_0 + \bb) + r c) \tc^2  }_0 \bigr|
  & \leq \PHI{8} \bydef 6 \hr_2^3 \left( \| \sigma(m\e_0 + \bb)  \|_{\nu} +  \hr_2 \right).
\end{alignat*}

To conclude, we collect all the terms and set
\[
  \E^{[2]} = \roundup \sum_{n=1}^{8} \PHI{n}.
\]
so that we can explicitly enclose the energy $E(\hkappa,\sigma(m\e_0+\hb))$ of the critical point via~\eqref{e:enclosedenergy}.

%% file: stability.tex
%!TEX root = ms.tex

The zero $\hx=(\hkappa,\hb)$ of $\FF$ found in Theorem~\ref{thm:contraction} corresponds to a critical point of $\EEE(u)$, both with respect to variations in the profile $u$ and with respect to variations in the domain size. To be precise, recalling  that mass is fixed, $(\hkappa,m \e_0 + \sigma(\hb))$ is a critical point of $E(\kappa,c)$ as defined in~\eqref{e:defE}, with respect to variations in $\C^J \times   \{ c \in \C^{\Z^3} :  c_0=m \}$.

For stability questions,  in order to restrict attention to
real-valued functions $u(\x)$ and physical domains, we should consider variations in $\R^J_+ \times
\{ c \in \C^{\Z^3} : c_0=m , c_{-k}=c_k^* \}$ only.
In this section we explore how to determine the Morse index of the critical point when restricting to the class of \emph{symmetric} (i.e.~$\G$-invariant) perturbations in $\C^J \times  \{ c \in \Xsym :  c_0=m \}$, or rather the conjugate symmetric subset thereof. The stability analysis thus includes variations in the domain size, while we restrict attention to $\G$-symmetric perturbations
Indeed, we study the Morse index of 
$\hx=(\hkappa,\hb) \in \Ssym_+$ as a critical point of $\tE(\kappa,b) = E(\kappa,m \e_0 + \sigma(b))$ for variations in 
\[
  \Ssym_0 \bydef \{ x \in \C^J \times \X_0 : \conj x = x \}.
\]  
Clearly $\tE$ is real-valued on $\Ssym_0$.
We recall that $\conj x = \overline{\J x}$, where $\J$ is defined in~\eqref{e:defJ}
and satisfies $\J^T=\J$  and $\overline{\J}=\J^{-1}$, see~\eqref{e:usefulJ}.
Throughout this section we will use superscript $T$ to denote the transpose and overline to denote elementwise complex conjugation.

Although we have been working over the field $\C$ in most of our arguments, $\Ssym_0$ is a real vector space. The Morse index of $\hx$ as a critical point of $\tE$ on $\Ssym_0$ is defined by 
\[
  \morseindex(\hx) \bydef \max  \bigl\{ \dim_{\R}(V) : V \subset_\R \Ssym_0 \text{ such that } v^T D^2 \tE (\hx) v < 0    \text{ for all }  v \in V \bigr\},
\]
where by $V \subset_\R \Ssym_0$ we denote an $\R$-linear subspace.
We have used matrix notation for the second derivative, which we will do frequently in this section.
If $\morseindex(\hx)>0$ then $\hx$ is a saddle point 
under $\G$-invariant perturbations, hence not a local minimizer. 
For a periodic domain of \emph{fixed} size, such a stationary profile may still be stable. An analogous analysis, but now without varying $\kappa$, may be performed to establish the latter type of (in)stability.

In the case $\morseindex(\hx)=0$ we
need slightly more information to conclude about local minimization, namely consider 
\[
  \morseindex_0(\hx) = \max  \{ \dim_{\R}(V) : V \subset_\R \Ssym_0 \text{ such that } v^T D^2 \tE (\hx) v \leq 0    \text{ for all } 0\neq v \in V \}.
\]
If $\morseindex(\hx) = \morseindex_0(\hx)$ then there are no neutral directions,
and in that case $\morseindex(\hx)=0$ implies that $\hx$ is a nondegenerate local minimum of $\tE$. As we will see below, for critical points found through Theorem~\ref{thm:contraction} it always holds that $\morseindex(\hx) = \morseindex_0(\hx)$.
 
The second derivative of $\tE(x)$ in a direction $v\in \Ssym_0$ can be expressed as
\[
  D^2 \tE (x) [v,v]= D^2 \tE (x) [v, \conj v]= \overline{v}^T \;\! \overline{\J} \;\! D^2 \tE(x) \;\! v ,\qquad \text{ for any }v \in \Ssym_0.
\]
Furthermore, by differentiating the identity $\tE(\conj x) = \overline{\tE(x)}$ twice we infer that
\begin{equation}\label{e:JD2EJ}
  \overline{\J} \;\! D^2 \tE (x) \overline{\J} =  \overline{D^2 \tE (x)}
  \qquad \text{ for any } x  \in \Ssym_0.
\end{equation}
To simplify notation we now introduce
\[
  \THETA \bydef \overline{\J}\;\! D^2 \tE(\hx).
\]

It follows from~\eqref{e:JD2EJ}, the properties of $\J$ in~\eqref{e:usefulJ}, and $(D^2 \tE(\hx))^T = D^2 \tE(\hx)$, that  
\[
  \overline{\THETA^T} = \THETA.
\]
In particular, all eigenvalues of $\THETA$ are real.
We introduce 
\[
  \Ssym_0^c \bydef \{ x \in \C^J \times \X_0 : \conj x = -x \},
\]
so that $\C^J \times \X_0  = \Ssym_0 \oplus \Ssym_0^c$ as vector spaces over $\R$. Multiplication by $i$ provides a natural isomorphism between $\Ssym_0$ and $\Ssym_0^c$.
Furthermore,~\eqref{e:JD2EJ} implies that  $\THETA \conj = \conj \THETA$, hence $\THETA$  leaves $\Ssym_0$ and $\Ssym_0^c$ invariant.
We may thus interpret $\THETA$ both as a $\C$-linear operator on $\X_0$
and as an $\R$-linear operator on $\Ssym_0$.
It is not difficult to infer that 
the number $\negnum_\C(\THETA)$ of negative eigenvalues of $\THETA$ as a $\C$-linear operator over $\X_0$ equals the number $\negnum_\R(\THETA)$ of negative eigenvalues of $\THETA$ as an $\R$-linear operator over $\Ssym_0$ (both counted with multiplicity):
\[ 
   \negnum_\C(\THETA) =\negnum_\R(\THETA). 
\] 
The former is easier to study, because we have been computing over $\C$  in $\C^J \times \X_0$,
while the latter is easily seen to equal the Morse index:
\[  
   \morseindex(\hx)= \negnum_\R(\THETA) . 
\] 
From now on we simply denote $\negnum=\negnum_\C=\negnum_\R$.

To get our hands on the eigenvalues of $\THETA$, we exploit that 
\[
  D \FF(\hx) = \overline{\J} \;\! D^2 \tE (\hx) = \THETA,
\]
which follows from the arguments in Remark~\ref{r:symmetricHessian}, see~\eqref{e:D2EisJDF}.
The case that $\THETA$ has eigenvalue $0$  corresponds to 
$\morseindex_0(\hx) > \morseindex(\hx)$.
However, as argued in Remark~\ref{r:IADF}, the operator $D \FF(\hx)$ is injective, hence we conclude that
\[ 
  \morseindex(\hx) = \morseindex_0(\hx).
\]

To determine $\negnum(\THETA) = \negnum(D \FF(\hx))$ we study 
the eigenvalues of the inverse $D \FF(\hx)^{-1}$, or rather
the eigenvalue of the approximate inverse $\AA$
(where $\overline{\AA}^T=\AA$ if we choose the matrix $A$ such that
$\overline{A}^T=A$, see Remark~\ref{r:tweakA}, since for the tail $\Lambda^{-1}$ this holds by construction).
Indeed we have both $\negnum(D \FF(\hx)) = \negnum(D \FF(\hx)^{-1})$ and
$\negnum(D \FF(\hx)^{-1}) = \negnum(\AA)$.
The latter follows by considering the homotopy 
\[ 
   \AA_s \bydef (1-s) \AA + s D \FF(\hx)^{-1} \qquad\text{for }0 \leq s \leq 1. \]
 It follows from the proof of Theorem~\ref{thm:contraction}, see Remark~\ref{r:IADF}, that $\|I - \AA D \FF(\hx)\|_{B(X_\rho,X_\rho)}<1$,
hence by the triangle inequality $\|I - \AA_s D \FF(\hx)\|_{B(X_\rho,X_\rho)}<1$ for all $s \in [0,1]$. We infer that no eigenvalues can cross $0$ along (the spectral flow of) the homotopy,
so that 
\[
  \negnum(\AA) =  \negnum (D \FF(\hx)^{-1}) = \negnum (D \FF(\hx)).
\]
Since $\AA$ has positive diagonal tail (namely $(P(\Dkk)|\Gacts|)^{-1} >0$ for $\knorm > K$),
it suffices to study the eigenvalues of the (finite dimensional) matrix $A$:
\[
   \negnum(\AA) = \negnum (A).
\] 
The number of negative eigenvalues of the self-adjoint matrix $A$ is determined computationally.
By Sylvester's law of inertia we know that $\negnum(\overline{V}^T \!\! A V)=\negnum(A)$
for any nonsingular $V$. We choose $V$ to consist of numerical approximations of eigenvectors of $A$, so that $\overline{V}^T \!\! A V$ is approximately diagonal.
Indeed, we check rigorously that $\overline{V}^T \!\! A V$ is diagonally dominant and then 
use Gershgorin's theorem to rigorously compute
\[
  \negnum (\overline{V}^T \!\!  A V)
  =\negnum (A)
  =\negnum(\AA)
  =\negnum (D \FF(\hx)^{-1})
  =\negnum (D \FF(\hx)) 
  =\negnum(\THETA)
  =\morseindex(\hx).
\]

%% file: numerics.tex
%!TEX root = ms.tex

In this Section we explain how we constructed we rigorously computed energy minimizers at various points $(m,\gamma)$. We also present some details about the families of solutions as $m$ is varied with fixed~$\gamma$ and show some specific solutions. Before describing our results we first explain the use of interval arithmetic and outline the computational strategy. Some aspects of the latter are summarized in Figure \ref{Fig:Cartoon}. \\

\begin{figure}
    \centering
        \includegraphics[width=\textwidth]{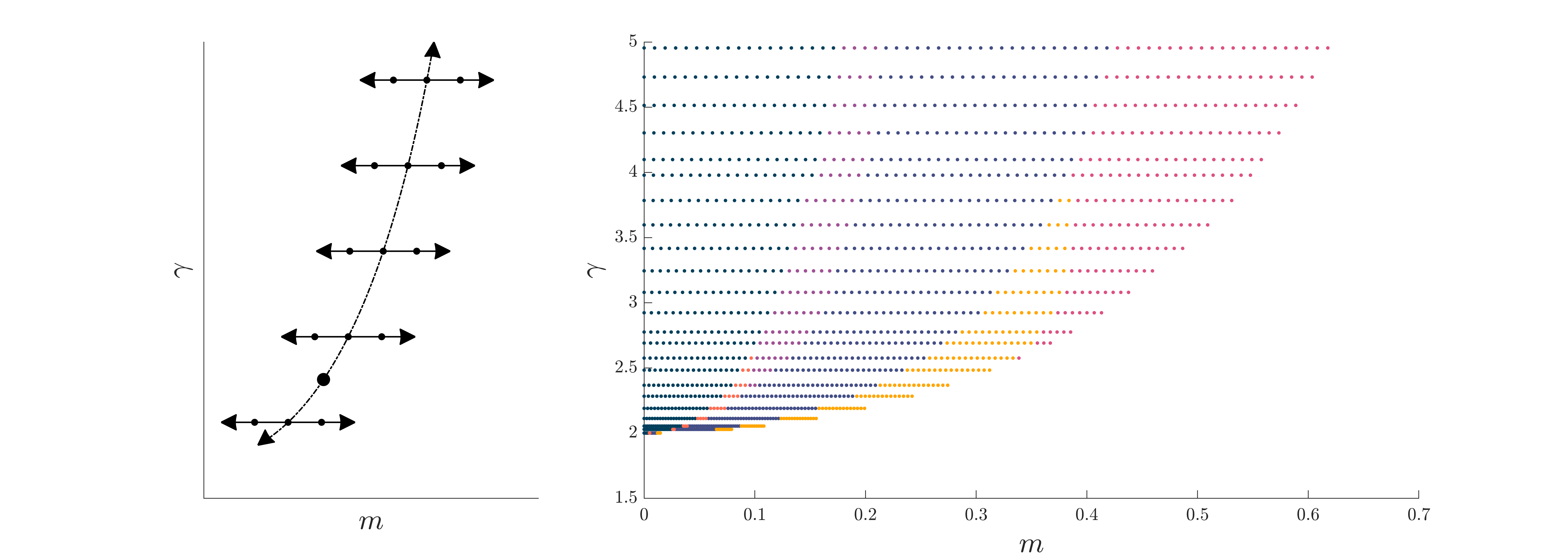}
    \caption{Left: Cartoon of the procedure used for generating candidate solutions. Right: Sample of the grid points onto which the energies were interpolated. The colors of the dots indicate the symmetry group of the lowest energy profile found at that point and match those in Figure \ref{fig:mainresult}.}
                           \label{Fig:Cartoon}
\end{figure}

\noindent{\bf Interval arithmetic}
We use the MatLab package Intlab~\cite{Intlab} whenever rigorously verifying any profiles. This means that all computational results are guaranteed to 
lie within a given interval and one needn't worry that the accumulation of floating point error has lead to inaccuracies. When we compute a quantity $Q$ with interval arithmetic we denote the lower bound on the exact value by $\rounddown Q$ and the upper bound by $\roundup Q$. We used interval arithmetic when verifying the energies to determine Figure \ref{Fig:RigEL}. In all cases potential errors due to using floating point arithmetic are much smaller than those due to finite truncation of the solution.\\

\noindent {\bf Computational strategy}
The mixed state ($u\equiv m$) changes linear stability along the curve
\[m^*(\gamma) = \sqrt{\frac{\gamma-2}{3\gamma}}.\]
This is the dashed curve in Figure \ref{fig:mainresult}. 
For fixed $\gamma -2 \ll 1$, a regular perturbation expansion at $(m, m^*(\gamma))$ allows the construction of branches of solutions valid for $|m-m^*| \ll 1$, see for instance \cite{CMW,vdBW3}. We 
 used direct energy minimization on the energy constrained to one dimensional structures (the lamellae) and each of space groups 17, 70, 194, 216, 224, 229 and 230 for $\gamma = 2.03$ and $m = .25m^*, .5 m^*$ and $.75m^*$. This produced approximately 15 candidate minimizers in each space group. The three lowest energy candidates from each group at each value of $m$ were then continued in $m$ with $\gamma = 2.03$ fixed to determine the local branch structure. In this region of the parameter space solutions can be very well approximated with a small number of modes. Comparing the results from a direct search method, a stochastic method and a gradient based approach for minimizing the energy we are confident that we found all the lowest energy solutions in each space group {\em at this value of} $\gamma$.

All branches but the ones from space group 70 (SG70) extended to $m=0$ and a maximum  $\le 1.5 m^*$ (cf.~Figure~\ref{Fig:EnergyMinimizingCurves}). The SG70 branches were observed to exist only for $0 < m_l(\gamma) \le m \le m_r(\gamma) < m^*$ for all values of $\gamma$ investigated.

At the initial $\gamma=2.05$ the energy curves within the computed space groups do not intersect. That is, if the double gyroid was the lowest energy profile for some $m$ amongst all solutions computed within SG230 then it was the profile with the lowest energy for all values of $m$ within SG230. For larger values of $\gamma$ we observed that this ordering was preserved except in some cases where no solution computed from that group was the global minimizer. For instance at $m=0$ spheres are not always the lowest energy solution seen from SG216.

We explored the parameter space by continuing a lamellar structure and the two or three solution profiles with the lowest energies from each space group as follows. The procedure is summarized in the left panel of Figure \ref{Fig:Cartoon}. We begin with the large spot, located at parameter values $\gamma=2.03$ and $m=0.5m^*(2.03)\approx 0.035$, where have found the a set of critical points as described above. These candidate minimizers were first continued in $\gamma$ along the curve $m = .5 m^*(\gamma)$. 
At many values $\gamma$ we then constructed solution curves using
pseudo-arclength continuation in $m$. This generated candidate profiles at different values of $m$ for each solution type. 
At each value of $\gamma$ the values of $m$ sampled were chosen adaptively by the continuation routine. 
The energies along levels of constant $\gamma$ were then interpolated onto a grid of $1000$ points from $m=0$ to the largest value of $m$ reached by any solution profile. The right panel presents every fifth value of $\gamma$ and every twentieth of $m$ (for illustration). This data 
provides starting points for refinements to obtain proofs. The spacing between the values of $\gamma$ used is smaller for smaller $\gamma$ and we also added some specific values to capture regions where three solutions have comparable energy.

To reduce the computational time we did not run the full proof at the more than 250000 computed profiles (although onerous it is not a difficult task). Instead, we identified values of $(m,\gamma)$ where the ordering of the lowest energy approximate solutions changed. There we refined in $m$, for fixed $\gamma$ finding values $m_a$ and $m_b$ at which we proved the existence of solutions such that the four energies satisfy
\[\roundup \E_a(m_a) + \E^{[2]}_a(m_a) < \rounddown \E_b(m_a) - \E^{[2]}_b(m_a) \quad\; \mbox{and} \quad \; \rounddown \E_a(m_b) - \E^{[2]}_a(m_b)  > \roundup \E_b(m_b) + \E^{[2]}_b(m_b)\]
for solutions from space groups $a$ and $b$.  Here $E^{[2]}$ is the energy bound derived in Section~\ref{s:energy}. In other words, the ordering of the minimal energies attained in space groups $a$ and $b$ changes somewhere in the interval $[m_a, m_b]$. We rigorously verified that the next closest in energy solutions truly have higher energy. Figure~\ref{Fig:RigEL} presents the center $m_c = (m_a+m_b)/2$ for all such points ($(m_b-m_a) < 10^{-3} m_c$). 
 We also computed the Morse index, as explained in Section~\ref{s:stability} to ensure that the energy minimizer amongst all candidates is definitely a local minimizer.
Figure~\ref{fig:mainresult} is derived directly from Figure~\ref{Fig:RigEL}
by coloring the different regions seperated by the (nearly) equal energy curves.

Points on the same boundary are connected only for visual assistance; we did not rigorously prove the existence of continuous curves (as opposed to the two dimensional case treated in~\cite{vdBW2}). The $\times$ symbols indicate points where we found approximate profiles and can compare their energies but
the proof was unsuccessful due to the memory requirements. 
As $\gamma$ increases, the interface between the positive and negative regions sharpens and more modes are required to construct an accurate approximation. 
All computations were performed on a consumer laptop. We have every confidence that the proof would be successful there if our algorithm were modified to be more memory efficient or we used a machine with more resources, but that is beyond the scope of the present work.

\begin{figure}
    \centering
        \includegraphics[width=.95\textwidth]{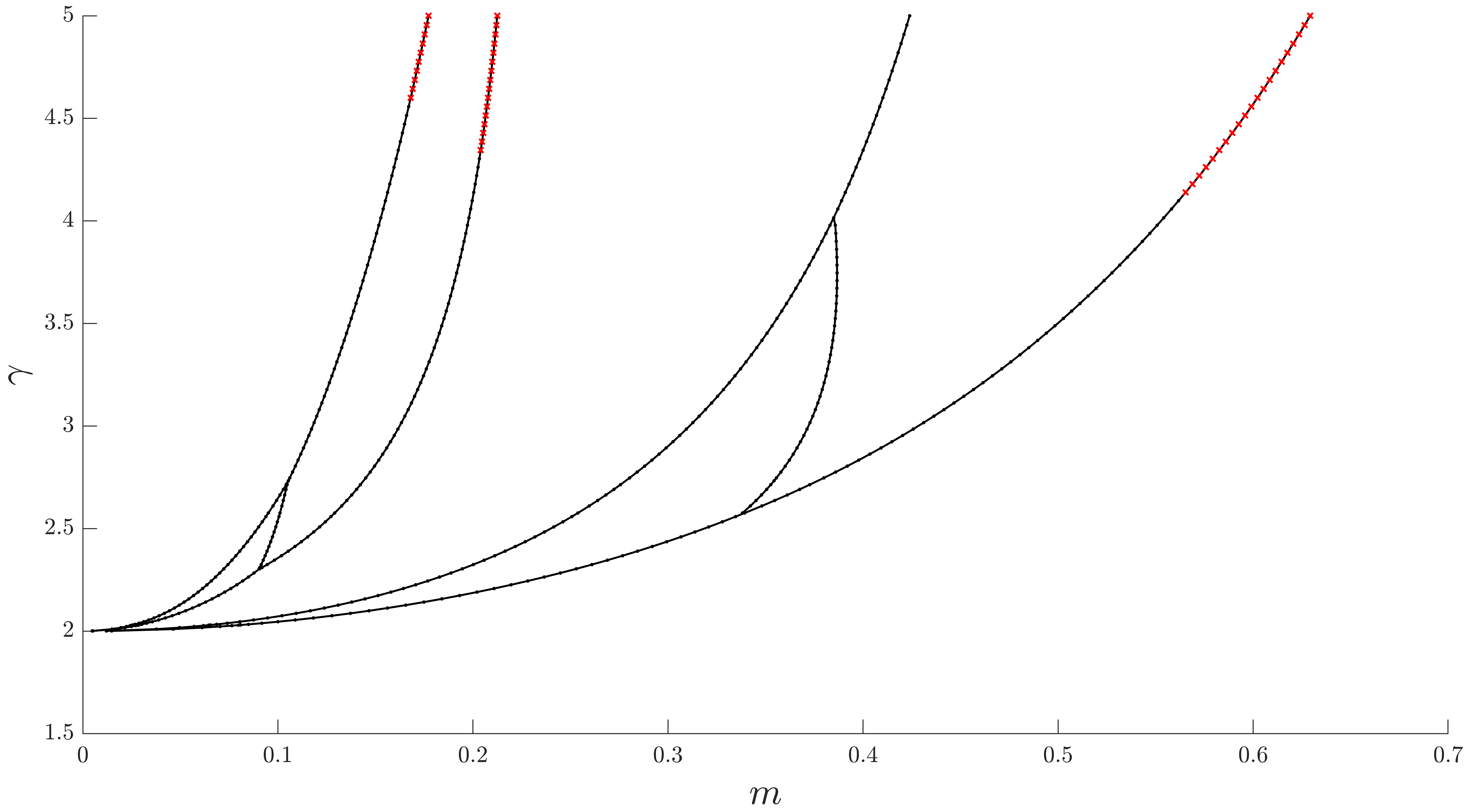}
    \caption{Boundaries in the energy landscape. Each dot represents a point $m_c = (m_a + m_b)/2$ such that, given solution profiles $a$ and $b$, $E_a(m_a) < E_b(m_a)$ and $E_a(m_b) > E_b(m_b)$ and $E_a$ and $E_b$ lower than the energies of all other computed solutions. The plotted points satisfy $m_c< 10^{-3} (m_a+m_b)$. The lines are visual aids and do not indicate that rigorous continuation was performed. Solution profiles and energy bounds rigorously established at all black dots. The proof was unsuccessful at the red $\times$'s.}
                           \label{Fig:RigEL}
\end{figure}

\subsection{Structure of solution branches for fixed $\gamma$.}
\label{subsec:branch}
Most solution branches bifurcate from the constant state at $m=m^*(\gamma)$.  Figure \ref{Fig:EnergyMinimizingCurves} shows solution branches for the energy minimizing phases at $\gamma = 2.5$. The left panel shows the energy and the right the length scales. Given how different the optimal length scales are we expect that the energy landscape could be quite different in confined or specified geometries.

\begin{figure}
    \centering
        \includegraphics[width=.95\textwidth]{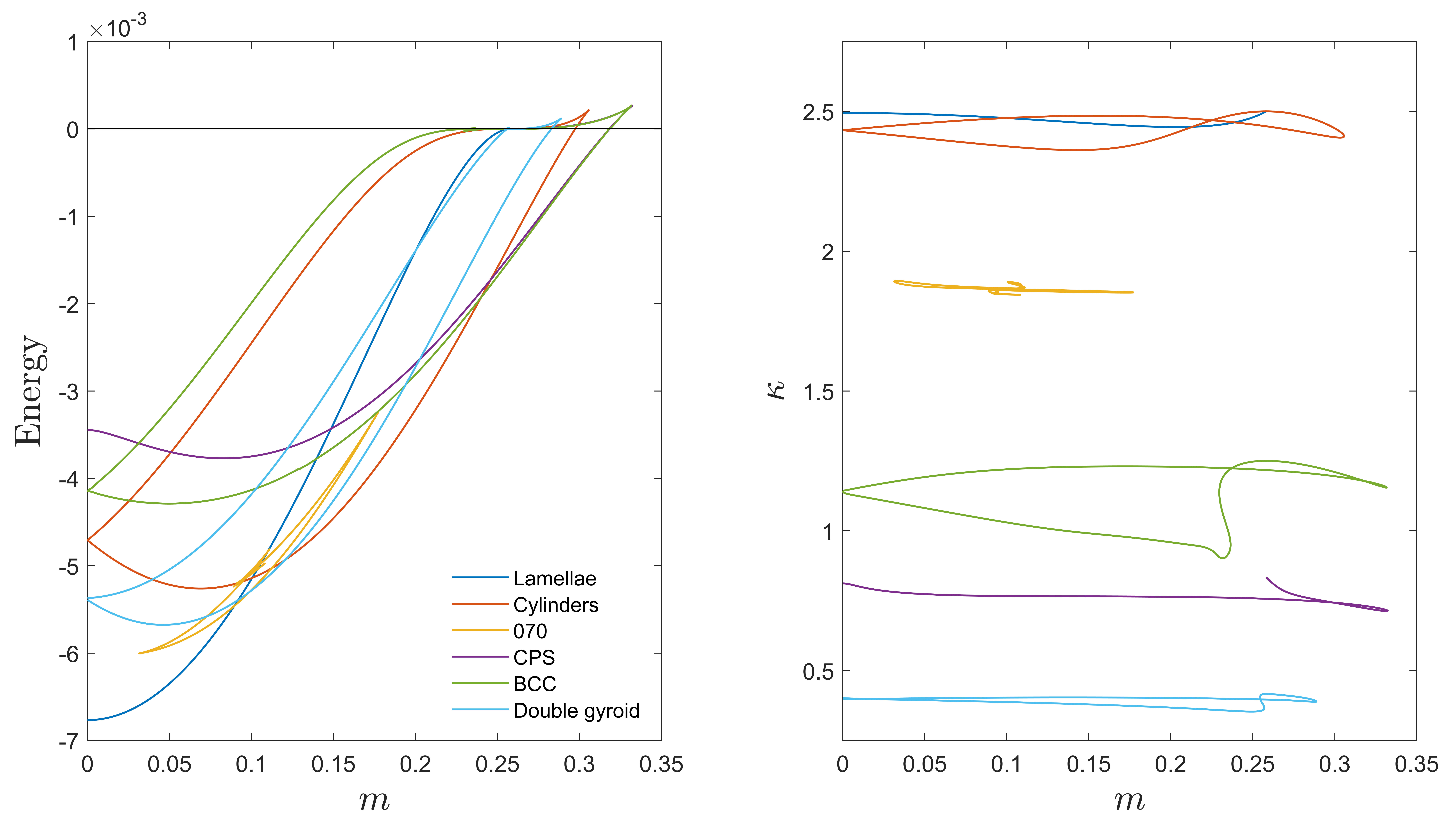}
    \caption{Branches of energy minimizing solutions with $\gamma = 2.5$. Left: energy. Right: smallest length scale.}
    \label{Fig:EnergyMinimizingCurves}
\end{figure}

Lastly, we compare the energies of the energy minimizing branches with other (non-minimizing) branches in Figure \ref{Fig:EnergyGroupComparisons}. This figure shows several features of the landscape: all groups can have branches that do not connect to the constant solution, the curves can be horrendously complicated and even non-minimizing branches can have energies very close to optimal. 
\begin{figure}
    \centering
        \includegraphics[width=.95\textwidth]{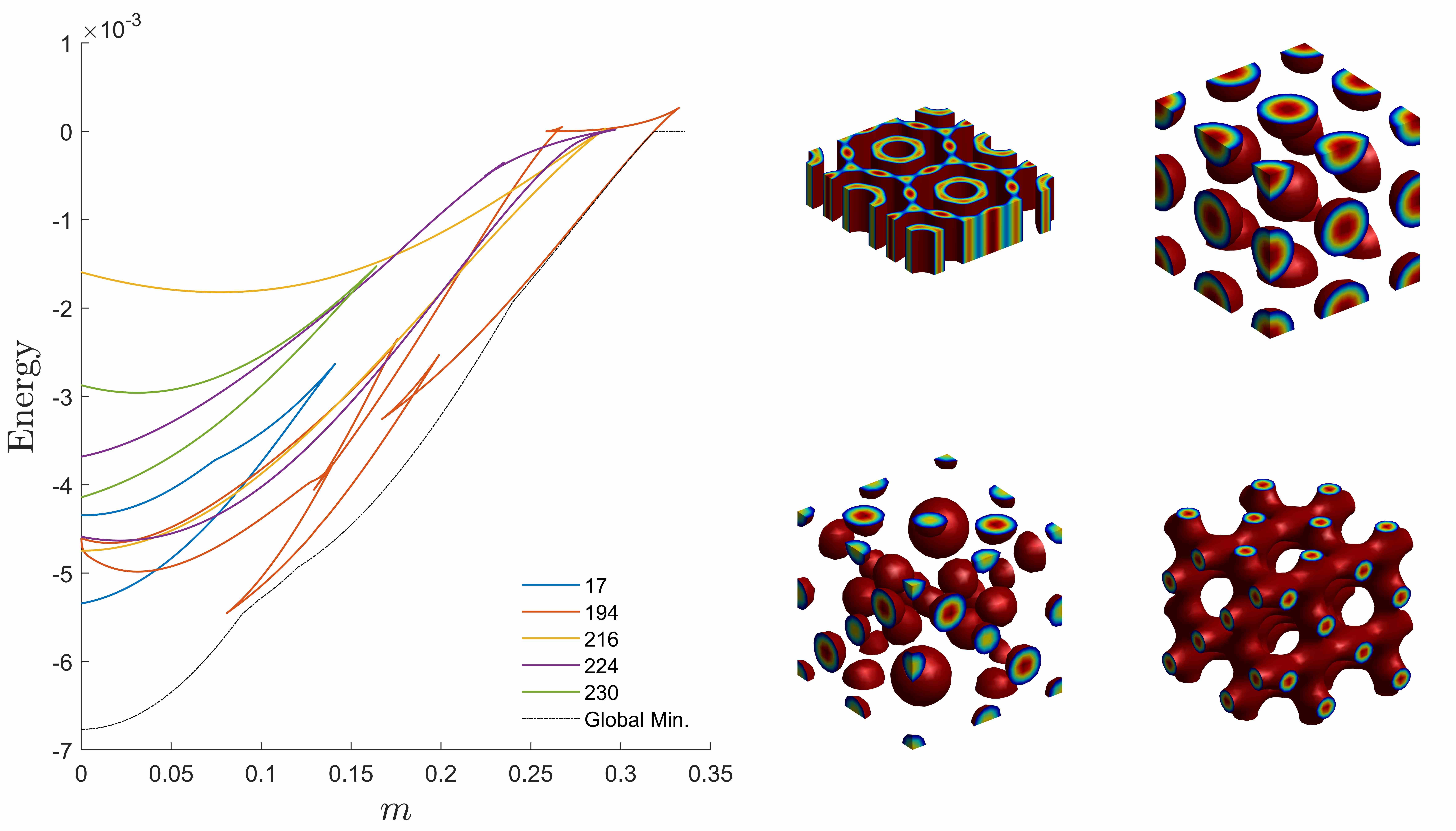}
    \caption{Non-minimizing branches. Left: the best SG194 and SG224 branches and the second best branches from SG17, SG216 and SG230. Right: sample profiles. Top left: SG17. Top right: 216. Bottom left: SG224. Bottom right: SG230.}
    \label{Fig:EnergyGroupComparisons}
\end{figure}

\subsection{The Morse index}
In addition to computing solutions and their energies, we also rigorously computed the Morse index as described in Section \ref{s:stability}. 
We computed the Morse index in both the space group and when imposing only minimal reflection symmetry (to factor out zero eigenvalues due to translational symmetry). Figure \ref{Fig:MorseDG} shows two branches of solutions in space group 230. The curves indicate which portion of the branches are stable both with and without full symmetry which is stable only when confined to the symmetry class and which is never locally stable. 

\begin{figure}
    \centering
        \includegraphics[width=.95\textwidth]{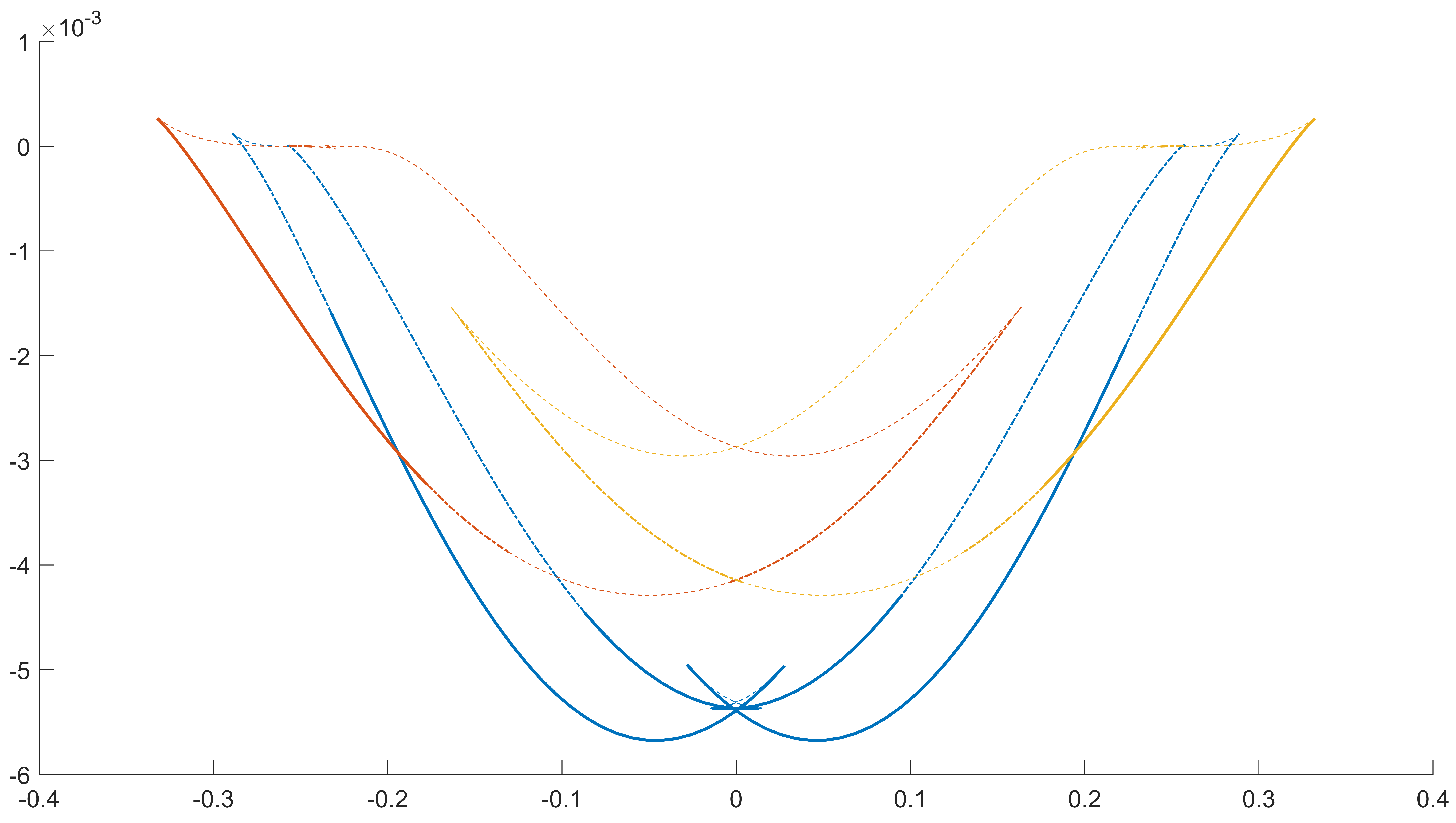}
    \caption{The stability on two branches of solutions from space group 230. Solid lines indicate stability even whilst imposing minimal symmetries, dashed-dotted portions are stable only within the symmetry class and dashed are not stable even within the symmetry class. The red and yellow curves are mirrors of each other with $m\to -m$. Within the symmetry class the Morse index is 0, 1, 2 or 3 whereas with minimal imposed symmetry there can be more than 100 unstable eigenvalues.}
    \label{Fig:MorseDG}
\end{figure}

\subsection{The perforated lamella}
\label{subsec:PL}
The perforated lamellae are a transient structure seen in physical experiments that is a local but not global minimizer. We computed two branches of solutions, one which connects the constant solution to lamellae and the other connects the constant solution to cylinders. Figures~\ref{Fig:EnergyPL} and~\ref{Fig:ExamplesPL} present comparisons of the energies and length scales of solutions from  SG194 with the essentially one dimensional lamellae and essentially two dimensional hexagonally packed cylinders. We can clearly see that the branches from SG194 extend from secondary bifurcations of lower dimensional structures to the constant solution. Rigorous analysis of such symmetry breaking bifurcations is subject of future work.
Notice that the branch of solutions connecting to the hexagonally packed cylinders have vertically aligned perforations whereas they are offset in those bifurcating from the lamellae. 
We computed the Morse index of perforated lamellae with both aligned and offset holes. Only those with offset perforations have Morse index 0 with minimal symmetry constraints, matching the fact that only offset perforations have been seen experimentally.
\begin{figure}
    \centering
        \includegraphics[width=.95\textwidth]{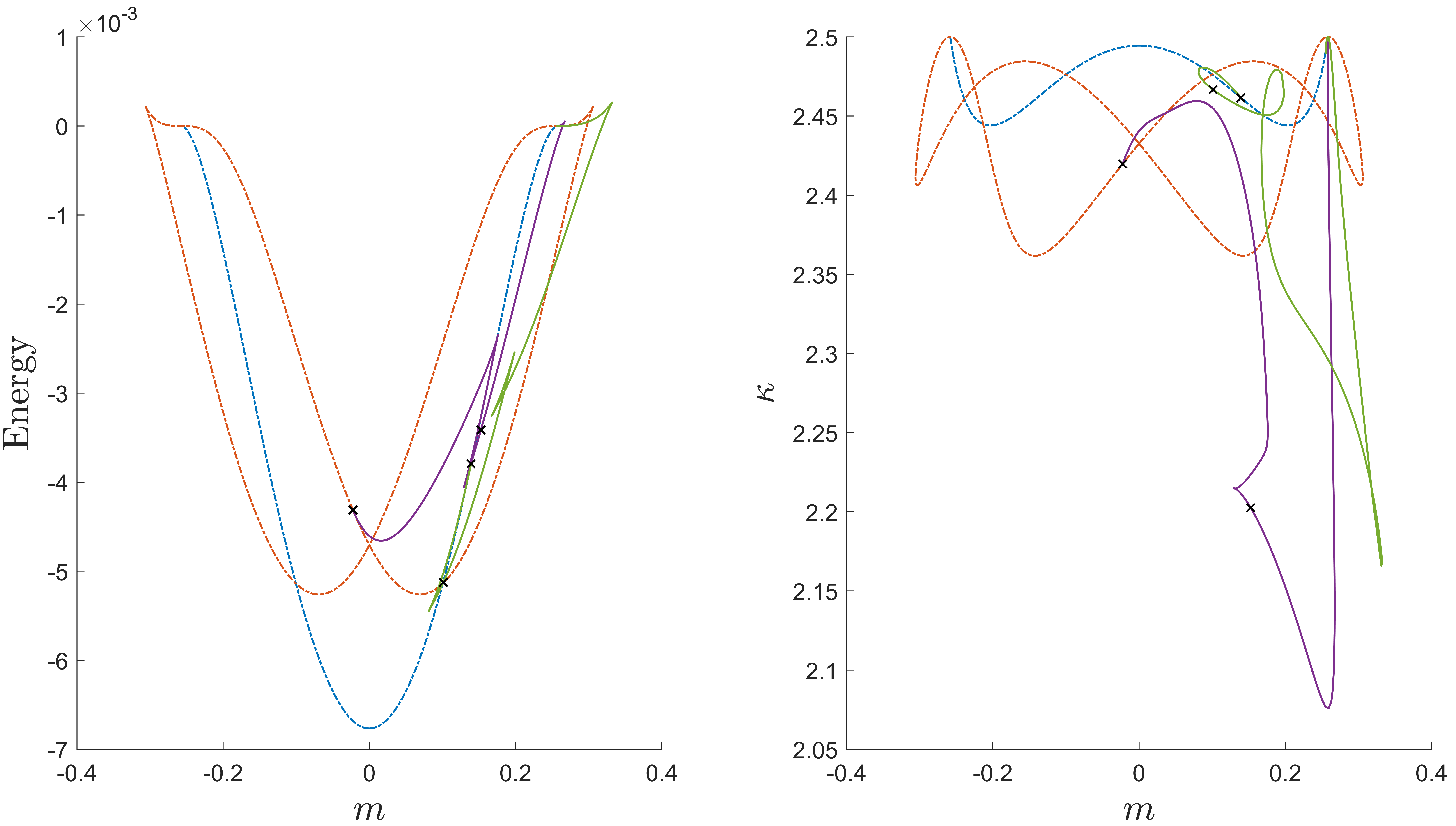}
    \caption{Solution branches in space group 194. Left: energy. The green branch connects the constant solution to lamellae (blue dashed curve) and the purple connects the constant solution to cylinders (red dashed curve). Right: length scales. Because the SG194 branch connects to a twice repeated lamella we have scaled $\kappa$ by 4. The $\times$'s correspond to the example profiles in Figure~\ref{Fig:ExamplesPL}.}
    \label{Fig:EnergyPL}
\end{figure}
\begin{figure}
    \centering
        \includegraphics[width=.95\textwidth]{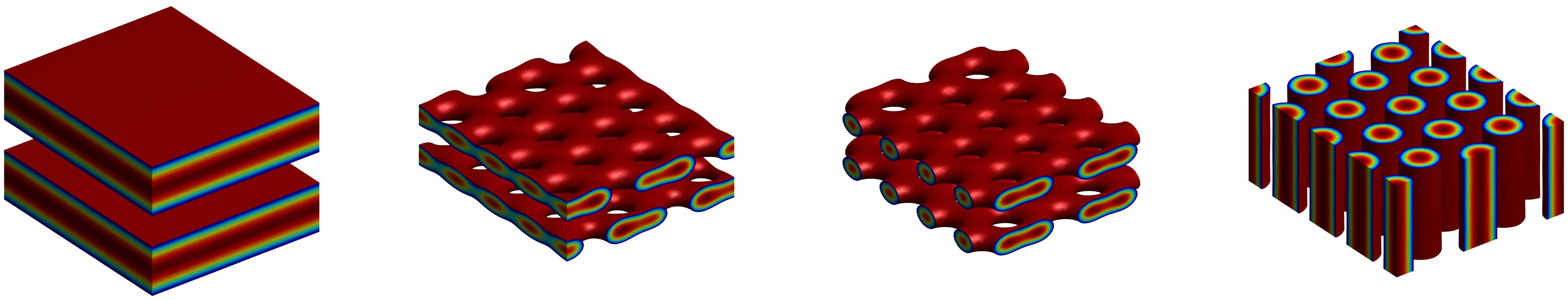}
    \caption{Solutions in space group 194. From the left: lamellae, perforated lamellae with offset perforations, a perforated lamellae with vertically aligned holes, cylinders.}
    \label{Fig:ExamplesPL}
\end{figure}

\subsection{Other exotic solutions}

\begin{figure}
    \centering
        \includegraphics[width=.95\textwidth]{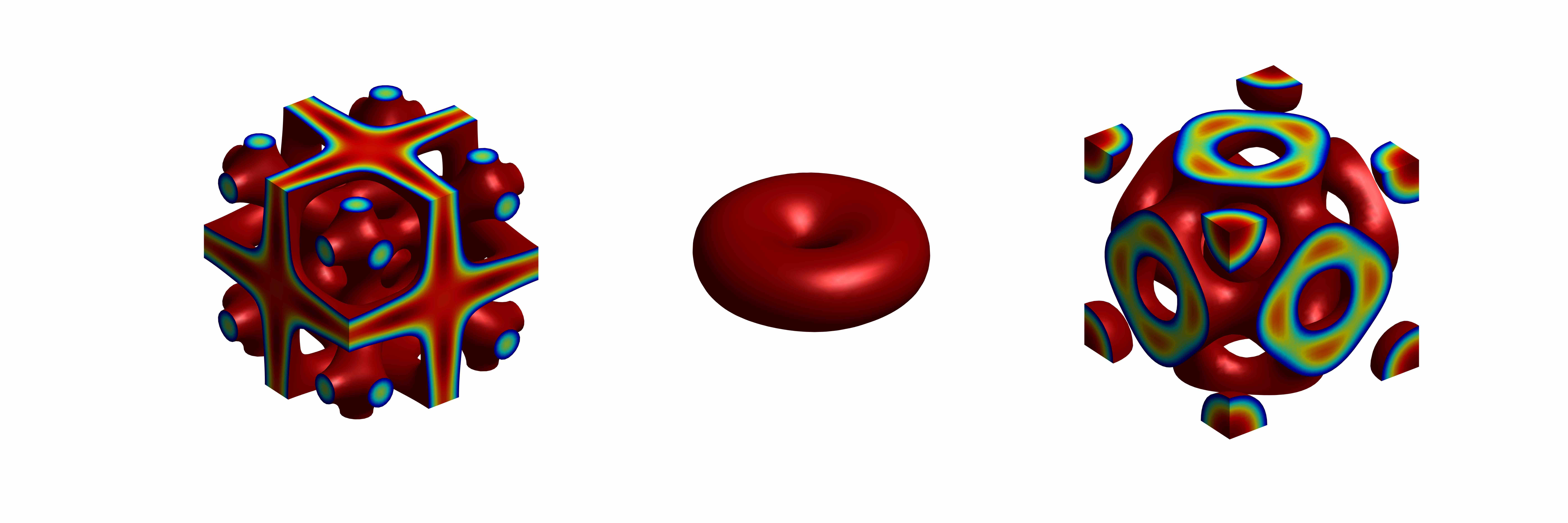}
    \caption{Exotic solution profiles. Left: Plumber's nightmare. $(m,\gamma) = (0.1114,2.5)$, SG 229. Center: Torus. $(m,\gamma) = (0.4394, 4)$, SG 194. Right: A continuous surface separating spheres. $(m,\gamma) = (0.1291, 2.5)$, SG 224.}
    \label{Fig:Exotic}
\end{figure}

By continuing a large number of solutions from many different symmetry classes we are able to identify structures not routinely seen experimentally. Figure \ref{Fig:Exotic} presents some of these. The left is the so-called Plumber's Nightmare with tubes inside tubes, the middle is tori stacked in hexagonally packed cylinders (we depicted just one torus) and the last is spheres distributed on a BCC lattice within a continuous frame. The Plumber's Nightmare is known to exist in other nano-materials but has never been seen in diblock copolymer experiments.

\subsection{Conclusions}

As already discussed extensively in the introduction, we are not claiming to have proven the global minimizers. Instead we have developed a methodology for comparing energies of stationary periodic profiles with any three dimensional space group symmetry, which are proven to be local minimizers with respect to variations in both the profile and the length scales. This is based on a carefully selected  functional analytic setup which respects both symmetry operations and variational properties (allowing a validated Morse index computation). It results in excellent rigorous bounds on all truncation errors, in particular concerning the energy of critical points, so that we can decide on the energy ordering of local minimizers. 

Much remains to be investigated. One issue already mentioned are symmetry breaking bifurcations. Another is to investigate other models in material science with similar mathematical structure where there are also questions regarding energy minimizing phases such as Phase Field Crystals \cite{PFCReview2012,Graphene}, the Functionalized Cahn-Hilliard energy \cite{PromislowFCH} and ternary block copolymers \cite{OhtaNkazawa,REN2003103}. In certain cases this may also motivate a more detailed study of profiles in space groups that allow for continuous symmetries. Asymptotics in the limit $\gamma \to 2$ are a final topic left for future work.

%% file: boundP.tex
%!TEX root = ms.tex

\subsection{Auxiliary lemmas}
\label{s:boundP}

We first list the estimates needed to control several tail terms. These are dominated (for large $k$) by the inverse of the linear part of the equation, which is characterized by
\[
  P(y) = \frac{1}{\gamma^2} \,\! y  - 1 + \frac{1}{y} .
\]
We define, for $K >y_P^{1/2}$ and $r_1 \geq 0$,
\begin{alignat}{1}
  \CC^{[0]}(K) & \bydef \frac{1}{P(K^2)}.
  \label{e:CC0} 
   \\
  \CC^{[1]}(K, r_1) & \bydef 
  \begin{cases}
	  \frac{\gamma^{-2} K^2}{P(K^2)} 
	  & \quad\text{for } \gamma \leq 2,\\[1mm]
	      \frac{K^2P'(K^2(1+r_1)) }{P(K^2)} 	    
		    & \quad\text{for } \gamma > 2.
  \end{cases}
   \label{e:CC1} 
\end{alignat}
\begin{lemma}\label{l:Pestimate}
Let $\bkappa \in \R^J_+$ and  $K^2 > \yP$. 
Then 
\begin{equation}\label{e:C0ineq}
   \frac{1}{|P(\Dk \bkappa)|} 
    \leq \CC^{[0]}(K)
	\qquad \text{for all } \knorm > K.
\end{equation}
Let $\bkappa \in \R^J_+$ and  $K^2 > \max\{ \yP ,2 \}$. 
Assume that $r_1 \in [0,1)$ is such that $K^2(1-r_1) \geq \gamma$. 
Then
\begin{equation}\label{e:C1ineq}
  \sup_{|\kappa-\bkappa|_{\RJ} \leq r_1} \frac{|P'(\Delta_k \kappa )|\, \Dk \bkappa}{|P(\Delta_k \bkappa)|}
    \leq \CC^{[1]}(K, r_1)
	\qquad \text{for all } \knorm > K  .
\end{equation}
\end{lemma}
\begin{proof}
Since $\knorm= \sqrt{\Dk \bkappa}$ the estimate~\eqref{e:C0ineq} follows directly from~\eqref{e:PyP1} and~\eqref{e:PyP2}.

The proof of the second estimate is more involved.
In view of \eqref{e:knorm} we write 
\[ 
  y=\Dk \bkappa=\knorm^2
\]
It follows from~\eqref{e:normFJ} that for any $\kappa$ satisfying $|\kappa-\bkappa|_{\RJ} \leq r_1 < 1$
we have
\[
  \kappa_j = \bkappa_j \left(1+ \frac{\kappa_j-\bkappa_j}{\bkappa_j} \right)
  \in  [\bkappa_j (1-r_1) , \bkappa_j (1+r_1)] \qquad\text{for }  1\leq j \leq J.
\]
Hence $|\kappa-\bkappa|_{\RJ} \leq r_1 < 1$ implies that
\begin{equation}\label{e:squeeze}
  (1-r_1)y= (1-r_1) \Delta_k \bkappa \leq  \Delta_k \kappa \leq (1+r_1) \Delta_k  \bkappa =  (1+r_1)y .
\end{equation}
We observe that $P'(y)\geq 0$ for $y \geq K^2(1-r_1) \geq \gamma$.
Since $y>y_P$ we use~\eqref{e:PyP1} and~\eqref{e:squeeze}
to reformulate~\eqref{e:C1ineq} as a uniform estimate on
\[
  \frac{P'( y (1+s)  ) \,y }{P(y)}  \qquad \text{ for } y > K^2 \text{ and } 
  s \in [-r_1, r_1].
\]
We note that $P''(y) = 2y^{-2} >  0$ for all $y>0$, so that it suffices to estimate
\[
 \frac{P'( y (1+r_1)  ) \,y }{P(y)}  \qquad \text{ for } y > K^2.
\]
We split the analysis into two cases: $\gamma>2$ and $\gamma \leq 2$.
For the latter case we simply observe that $P'(y) \leq \gamma^{-2}$ for any $y>0$ and estimate
\[
   \frac{P'( y (1+r_1)  ) \,y }{P(y)} \leq \frac{\gamma^{-2} K^2}{P(K^2)} ,
\]
where we have used that $K^2 \geq 2$ to conclude that $P(y)/y$ is increasing for $y \geq K^2$.

For the case $\gamma>2$ we observe that $y_P>\gamma>2$.
We write $w=y^{-1} < K^{-2} \leq \frac{1}{2}$ and $\epsilon = (1+r_1)^{-2} \in (\frac{1}{4},1]$,
so that 
\begin{equation}\label{e:reformulatePPy}
  \sup_{y>K^2} \frac{P'( y (1+r_1)  ) \,y }{P(y)} =  \sup_{0<w<K^{-2}} \frac{\gamma^{-2}- \epsilon w^2}{\gamma^{-2}-w+w^2} .
\end{equation}
To determine the supremum
in the right hand side we determine the sign of the derivative
\begin{alignat*}{1}
  \frac{d}{d w} \frac{\gamma^{-2}- \epsilon w^2}{\gamma^{-2}-w+w^2}
  & = \frac{-2\epsilon w (\gamma^{-2}-w+w^2) -(-1+2w)(\gamma^{-2}- \epsilon w^2)}{(\gamma^{-2}-w+w^2)^2} \\
  & =  \frac{ (1-\epsilon)\gamma^{-2} (1-2w) + 
    + \epsilon[ (w-2\gamma^{-2})^2+ \gamma^{-4}(\gamma^2-4)]   }{(\gamma^{-2}-w+w^2)^2} \\
  & \geq 0.
\end{alignat*}
Here we have grouped terms in the numerator conveniently to conclude that it is positive for any $w \in [0,\frac{1}{2}]$ and $\epsilon \in [0,1]$.
Since the derivative of the argument in the right hand side of~\eqref{e:reformulatePPy} is positive, the supremum is attained at $w=K^{-2}$ or, equivalently, at $y=K^2$. This concludes the proof of the bound~\eqref{e:C1ineq}.
\end{proof}

We define, for $\nu>1$, $K>0$ and $r_1 \in [0,1)$,
\begin{alignat}{1}
  \EE^{[0]} (K,\nu,r_1) &\bydef  (1+r_1) \EE^{[1]},\\
  \EE^{[1]} (K,\nu) &\bydef
  \begin{cases}
	\gamma^{-2} (e \log \nu)^{-2}  &\quad\text{for } K \leq (\log \nu)^{-1} \\
    \gamma^{-2} K^2 \nu^{-2K}   &\quad\text{for } K > (\log \nu)^{-1}	  ,
  \end{cases}
\label{e:defE1j} \\
  \EE^{[2]} (K,\nu,r_1) &\bydef 
	 2 (1-r_1)^{-3} K^{-2} \nu^{-2K}.	  
\label{e:defE2j}	
\end{alignat}	
\begin{lemma}\label{l:Eterm}
Let $\bkappa \in \R^J_+$ and  $K > 0$. 
Assume that $r_1 \in [0,1)$ is such that $K^2(1-r_1) \geq \max\{1,\gamma\}$. 
Then
\begin{alignat}{2}
\sup_{|\kappa-\bkappa|_{\RJ} \leq r_1} 
  | P(\Delta_k \kappa)|  \,  \nu^{-2\knorm}
 &\leq \EE^{[0]} (K,\nu,r_1)
 &&\qquad\text{for all } \knorm > K, \label{e:P0E0} \\
\sup_{|\kappa-\bkappa|_{\RJ} \leq r_1} 
  | P'(\Delta_k \kappa)|  \, (\Dk \bkappa) \, \nu^{-2\knorm}
 &\leq \EE^{[1]} (K,\nu)
 &&\qquad\text{for all } \knorm > K, \label{e:P1E1} \\
\sup_{|\kappa-\bkappa|_{\RJ} \leq r_1}  | P''(\Delta_k \kappa) | \, (\Dk \bkappa)^2  \, \nu^{-2\knorm}
 &\leq \EE^{[2]} (K,\nu,r_1)
 &&\qquad\text{for all } \knorm > K. \label{e:P2E2}
\end{alignat}
\end{lemma}
\begin{proof}
Let $\knorm>K$.	We start with the second bound. 
As in the proof of Lemma~\ref{l:Pestimate}, see~\eqref{e:squeeze}, we infer that 
$\Delta_k \kappa \geq (1-r_1) \Delta_k \bkappa$ 
whenever $|\kappa-\bkappa|_{\RJ} \leq r_1 <1$. 
By assumption $K^2(1-r_1) \geq \gamma$,
hence for the second estimate we now simply bound
$ 0 \leq P'(\Delta_k \kappa) < \gamma^{-2} $
for all $\kappa$ such that $|\kappa-\bkappa|_{\RJ} \leq r_1 <1$,
and all $k$ such that $\Delta_k \bkappa = \knorm^2 >K^2$. 
It then suffices to estimate $z^2 \nu^{-2z}$ uniformly for $z>K$. Such a bound is provided by $\gamma^2 \mathcal{E}^{[1]} (K,\nu) $. This concludes the proof of~\eqref{e:P1E1}.
 
For the first estimate we observe that 
\[
  \Delta_k \kappa   \geq   (1-r_1) \Delta_k \bkappa   =  
  (1-r_1) \knorm^2 > (1-r_1) K^2  \geq 1 ,
\]
whenever $|\kappa-\bkappa|_{\RJ} \leq r_1 <1$.
Hence $P(\Dk \kappa) \leq \gamma^{-2}\Dk \kappa \leq (1+r_1)  \gamma^{-2} \Dk \bkappa$. The remainder of the proof of~\eqref{e:P0E0} then follows directly from the arguments used in the previous estimate.
 
 For the third bound we estimate 
 \[
   0 \leq P''(\Delta_k \kappa) =  2 (\Delta_k \kappa)^{-3} \leq 2 (1-r_1)^{-3}
   (\Delta_k \bkappa)^{-3} .
 \]
 Since $\Delta_k \bkappa =  \knorm^2$
 it suffices to estimate $z^{-2} \nu^{-2z}$ uniformly for $z>K$.  
 It is now straightforward to derive 
 the estimate~\eqref{e:P2E2}. 
\end{proof}

%% file: boundY.tex
%!TEX root = ms.tex

\subsection{The bounds $\YY{i}$}
\label{s:boundY}

To obtain the $\YY{i}$-bounds satisfying~\eqref{e:boundY1}--\eqref{e:boundY2} we argue as follows.
Since $\bb \in \X_0^K$ 
we can compute all non-vanishing components of 
$\FF(\bb,\bkappa)$ with interval arithmetic.
In particular
\[
F_k(\bb,\bkappa) =0 \qquad\text{for all } \kk > 3K,
\]
since the nonlinearity is cubic.
Then the computation of the norms $| \pi_\kappa \AA \FF(\bx) |_{\FJ}$ 
and $ \| \pi_b \AA \FF(\bx) \|_{\X_0}$ 
takes a finite number of operations.
Taking the upper boundaries of the obtained intervals, denoted by $\roundup$, gives $\YY{1}$ and $\YY{2}$, i.e.,
\[
  \YY{1} = \roundup
  \left| A_{11} H(\bb,\bkappa) + A_{12} \pi_K F(\bb,\bkappa) \right|_{\FJ}  ,
\]
and
\[
  \YY{1} = \roundup \left[  
  \sum_{k \in \ZZ_0^K} \left|\left[A_{21} H(\bb,\bkappa) + A_{22} \pi_K F(\bb,\bkappa) \right]_k\right|  \omega_k  +
  \sum_{\substack{K < \kk \leq 3K\\k \in \ZZ_0}} \left| \lambda^{-1}_k F_k(\bb,\bkappa)]\right|  \omega_k 
  \right] .
\]

%% file: derivatives.tex
%!TEX root = ms.tex

\subsection{Expressions for the derivatives}

In order to derive the bounds~\eqref{e:boundZ11}--\eqref{e:boundW22}
we first derive expressions for the derivatives.
Recalling the definition of $\Lambda$ in~\eqref{e:defLambda} and $\Phi$ in~\eqref{e:defPhi} and $\orbitcount$ in~\eqref{e:orbitcount}, we conclude from~\eqref{e:Fcondensed} that
\begin{equation}\label{e:FLPhi}
  D_b F(\kappa,b)  = \Lambda + \orbitcount D_b \Phi(b),
\end{equation}
which, based on~\eqref{e:defsigma}, we express component-wise as
\begin{equation}\label{e:DbF}
  D_{b_{k'}} F_k (\kappa,b)=  
  \PP(\Dkk) |\Gacts| \delta_{k, k'}
  + 3 \, |\Gacts| \sum_{k'' \in \Gacts[k']}  \talpha(k',k'') \conv{\sigma(m\e_0+b)^2}_{k-k''}.
\end{equation}
The other partial derivative of $F$ is given by
\begin{equation}\label{e:DkappaF}
  D_{\kappa_j} F_k (\kappa,b)= \PP'(\Dkk)  \Dk^j \,|\Gacts|\, b_k .
\end{equation}

To compute the derivatives of $H_j$ with respect to $b_k$
we observe that $\PP(\Dkk) = \gamma^{-2} \Dkk  - 1 + (\Dkk) ^{-1} $
is invariant under the group action.
Hence, it follows from Lemma~\ref{l:symdiff} (using the property $\Dkj[-k]{j}=\Dkj[k]{j}$ and~\eqref{e:Hwrite}) that
\begin{equation}\label{e:DbH}
  D_{b_k}  H_j (\kappa,b)=  \PP'(\Dkk)    \Dk^j |\Gacts|  \sigma(b)_{-k} .
\end{equation}
Furthermore, it follows from~\eqref{e:Hrewrite} that
\begin{equation}\label{e:DkappaH}
  D_{\kappa_{j'}}  H_j (\kappa,b)=  \frac{1}{2} \sum_{k \in \ZZ_0} \PP''(\Dkk)    \Dk^j \Dk^{j'} |\Gacts|  b_k \sigma(b)_{-k} .
\end{equation}

\begin{remark}\label{r:symmetricHessian}
Let $\tE (\kappa, b) \bydef E(\kappa,\sigma(m\e_0+ b))$. 
It is readily seen from Lemmas~\ref{l:symdiff} and~\ref{l:symdiffpower} that
\begin{equation}\label{e:tEtf}
  D_{b_k}\tE (\kappa, b)  = \tf_{-k}(\kappa,\sigma(m\e_0 + b)),
\end{equation}
due to the decision to premultiply by $|\Gacts|$ in the definition~\eqref{e:tf} of $\tf$.
By Remark~\ref{r:explicitGamma} and the definition of $F$ in~\eqref{e:defFF}, Equation~\eqref{e:tEtf} can be rewritten as
\[
  D_{b_k}\tE (\kappa, b) = \phi_k F_{\tau(k)}(\kappa,b),
\]
so that (interpreting the derivative as a gradient)
\[
   D_b \tE (\kappa, b) = \J \;\! F(\kappa, b),
\]
and analogously, in view of~\eqref{e:dEdkappa}, for the full derivative we have
\[
  D \tE (\kappa, b) = \J \;\! \FF(\kappa, b). 
\]  
For the second derivative we obtain
\begin{equation}\label{e:D2EisJDF}
  D^2 \tE (\kappa, b) = \J \;\! D \FF(\kappa, b). 
\end{equation}
The similarity between~\eqref{e:DkappaF} and~\eqref{e:DbH} is thus explained by the symmetry of the second derivative of the energy: $D_{\kappa_{j}} D_{b_k} \tE = D_{b_k} D_{\kappa_{j}} \tE$.
\end{remark}

%% file: boundZ.tex
%!TEX root = ms.tex

\subsection{The bounds $\ZY{i}$}
\label{s:boundZ}

We are looking for estimates of the form
\begin{alignat*}{1}
  | \pi_\kappa [I  - \AA[D\FF(\bx) ]v |_{\FJ} 
  \hsp & \leq  
   \ZY{1}_1 | \pi_\kappa v|_{\FJ} + \ZY{1}_2 \| \pi_b v\|_{\X_0} 
   , \\
  \| \pi_b [I  - \AA[D\FF(\bx) ]v \|_{\X_0}  & \leq  
   \ZY{2}_1 | \pi_\kappa v|_{\FJ} + \ZY{2}_2 \| \pi_b v\|_{\X_0}  ,
\end{alignat*}
uniformly for
$v\in\BB_{(1,1)}(0)$.

We first split
\begin{equation}\label{e:IADF1}
I - \AA D\FF(\bx)
= [I (\pi_\kappa+\pi_K) - \AA (\pi_\kappa+\pi_K) D\FF(\bx)] 
 + [I \pi_\infty - \AA \pi_\infty D\FF(\bx)] .
\end{equation}
Since $\pi_\infty \bx$ vanishes, when we evaluate the derivatives~\eqref{e:DbH} and~\eqref{e:DkappaF} at the numerical approximation, the following tail terms vanish:
\begin{alignat}{1}
	D_b H(\bx) \pi_\infty &=0,  \label{e:DbHbx0}\\
	\pi_\infty D_\kappa F(\bx) &=0 . \label{e:DkappaFbx0}
\end{alignat}
In view of~\eqref{e:DkappaFbx0} the tail term in~\eqref{e:IADF1} thus reduces to 
\begin{equation}\label{e:IADF2}
 I \pi_\infty - \AA  \pi_\infty D \FF(\bx) = I \pi_\infty - \Lambda^{-1} \pi_ \infty [\Lambda + D_b \Psi(\bx) ] \pi_b= - \Lambda^{-1} \orbitcount \pi_ \infty D_b \Phi(\bb) \pi_b,
\end{equation}
where we have used the decomposition~\eqref{e:FLPhi}.

By combining~\eqref{e:IADF1} and~\eqref{e:IADF2} we infer that the restriction of $I - \AA D\FF(\bx)$ to the subspace $\CJ$ may be expressed as
($I_\kappa$ being the identity on $\CJ$) 
\begin{alignat*}{1}
  \pi_\kappa [I - \AA D\FF(\bx)] \Bigl|_{\CJ} & = 
  I_\kappa - A_{11} D_\kappa H(\bx) - A_{12} D_\kappa \pi_K F(\bx) , \\
  \pi_b [I - \AA D\FF(\bx)] \Bigl|_{\CJ} & = 
  - A_{21} D_\kappa H(\bx) - A_{22} D_\kappa \pi_K F(\bx) ,
\end{alignat*}
which are  $J\times J$ and  $N \times J$ matrices, respectively.
Both of these can be computed using interval arithmetic.
We then set
\begin{alignat*}{1}
	\ZY{1}_{1} &= \roundup \| I_\kappa - A_{11} D_\kappa H(\bx) - A_{12} D_\kappa \pi_K F(\bx) \|_{B(\FJ,\FJ)},\\
 	\ZY{2}_{1} &= \roundup \|  - A_{21} D_\kappa H(\bx) - A_{22} D_\kappa \pi_K F(\bx) \|_{B(\FJ,\X_0)}.
\end{alignat*}
These operator (matrix) norms can be calculated using the expressions~\eqref{e:norm11} and~\eqref{e:norm21}.

By combining~\eqref{e:IADF1} and~\eqref{e:IADF2} with~\eqref{e:DbHbx0} we
infer that
\begin{alignat}{2}
  \pi_\kappa [I - \AA D\FF(\bx)] \Bigl|_{\X_0} & =
  - A_{11} D_b H(\bx) \pi_K - A_{12} [\pi_K \Lambda + \orbitcount  D_b \pi_K \Phi(\bb)] ,
  \label{e:IADF3}\\
  \pi_b [I - \AA D\FF(\bx)] \Bigl|_{\X_0} & =
  - A_{21} D_b H(\bx) \pi_K - A_{22} [\pi_K \Lambda + \orbitcount D_b \pi_K \Phi(\bb)] -
  \Lambda^{-1} \orbitcount \pi_\infty \Phi(\bb). \label{e:IADF4}
\end{alignat}
The operator norm of~\eqref{e:IADF3} can be expressed (see~\eqref{e:l1sup}) as
\begin{equation}\label{e:supcol}
  \left\| \pi_\kappa [I - \AA D\FF(\bx)] \Bigl|_{\X_0} \right\|_{B(\X_0,\FJ)}
  = \sup_{k\in \ZZ_0} \omega_k^{-1} |  \pi_\kappa [I - \AA D\FF(\bx)] \e_k |_{\FJ}.
\end{equation}
We observe that
\begin{equation}\label{e:vanish1}
  \bigl( - A_{11} D_b H(\bx) \pi_K - A_{12} \pi_K \Lambda \bigr) \e_k =0
  \qquad \text{for } \knorm > K
\end{equation} 
and
\begin{equation}\label{e:vanish2}
  D_b \pi_K \Phi(\bb) \e_k = 0 \qquad \text{for } \knorm > 3K ,
\end{equation}
since $\Phi$ is cubic.
Writing $I_b=\pi_K+(\pi_{3K}-\pi_K)+(I_b-\pi_{3K})$ with
 $I_b$ the identity on $\X_0$, we apply~\eqref{e:splitl1norm} twice 
 and compute, in view of~\eqref{e:IADF3}, \eqref{e:supcol}, \eqref{e:vanish1} and~\eqref{e:vanish2},
\begin{alignat*}{1}
	\UUU_{1} &\bydef \max_{k\in\ZZ_0^K} \omega_k^{-1} | A_{11} D_b H(\bx) \e_k + A_{12} [\pi_K \Lambda + \orbitcount  \pi_K D_b \Phi(\bb)] \e_k |_{\FJ} , \\
	\UUU_{2} &\bydef \max_{\substack{K < \knorm \leq 3K\\ k\in\ZZ_0}}  \omega_k^{-1} | A_{12} \orbitcount  \pi_K  D_b \Phi(\bb) \e_k |_{\FJ} .
\end{alignat*}
For the remaining term we have
\[
  \sup_{\substack{\knorm > 3K\\ k\in\ZZ_0}}  \omega_k^{-1} |  \pi_\kappa [I - \AA D\FF(\bx)]  \e_k |_{\FJ} =0,
\]
due to~\eqref{e:vanish1} and~\eqref{e:vanish2}.
Hence we set
\[
  \ZY{1}_2 = \roundup \max \bigl\{ \UUU_1, \UUU_2 \bigr\} .
\]

Similarly, 
the operator norm of~\eqref{e:IADF4} can be expressed as
\[
  \left\| \pi_b [I - \AA D\FF(\bx)] \Bigl|_{\X_0}  \right\|_{B(\X_0,\X_0)}
  = \sup_{k\in \ZZ_0} \omega_k^{-1} \|  \pi_\kappa [I - \AA D\FF(\bx)]  \e_k \|_{\X_0}.
\]
Using an analogous splitting, we define 
\begin{alignat*}{1}
	\VVV_{1} &\bydef \max_{k \in \ZZ_0^K} \omega_k^{-1} \| \e_k - A_{21} D_b H(\bx) \e_k - A_{22} [ \Lambda - \orbitcount  \pi_K  D_b \Phi(\bb)] \e_k + \Lambda^{-1} \orbitcount (\pi_{3K}-\pi_K) D_b \Phi(\bb) \e_k \|_{\X_0} , \\
	\VVV_{2} &\bydef \max_{\substack{K < \knorm \leq 3K\\k \in \ZZ_0}} \omega_k^{-1} \| A_{22} D_b \orbitcount \pi_K \Phi(\bb) \e_k +\Lambda^{-1} \orbitcount  (\pi_{5K} - \pi_K) D_b \Phi(\bb) \e_k \|_{\X_0} ,
\end{alignat*}
where we have used that~\eqref{e:DbF} implies that
\begin{equation}\label{e:bandwidth2K}
   (D_b \Phi(\bb) \e_k)_{k'} = 0 \qquad \text{for } \knorm[k'-k] > 2K.
\end{equation} 
Where this estimate differs most from the previous one is in that 
we now need to estimate the tail term
\begin{equation}\label{e:supk3K}
 \sup_{\substack{\knorm > 3K\\k \in \ZZ_0}} \omega_k^{-1}   \| \Lambda^{-1} \orbitcount  D\Phi(\bb) \e_k \|_{\X_0} .
\end{equation}
For any $\knorm > 3K$ it follows from~\eqref{e:bandwidth2K} that
\begin{alignat}{1}
  \omega_k^{-1} \| \Lambda^{-1} \orbitcount D\Phi(\bb) \e_k \|_{\X_0} &= 
  \sum_{\substack{\knorm[k']>K \\k' \in \ZZ_0}} \frac{1}{|P(\Dk[k']\bkappa)|} |D\Phi_{k'}(\bb) \e_k | \frac{\omega_{k'}}{ \omega_k } \nonumber \\
  & \leq \left[ \sup_{\substack{\knorm[k']>K \\k' \in \ZZ_0}} \frac{1}{|P(\Dk[k']\bkappa)|}    \right]   \sum_{k' \in \ZZ_0} |D_{b_k}\Phi_{k'}(\bb)| \frac{\omega_{k'}}{ \omega_k } .
  \label{e:tailZ1term}
\end{alignat}
We estimate the two factors in the righthand side separately, starting with the latter. We use the definition~\eqref{e:defsigma} of $\sigma$
and the fact that
$|\talpha(k',k'')|=1$, to infer that
\begin{alignat}{1}
   \sum_{k' \in \ZZ_0} \frac{\omega_{k'}}{\omega_{k}}  
  \left| D_{b_k}\Phi_{k'}(\bb) \right|
  &=   3 \sum_{k' \in \ZZ_0}   \frac{\omega_{k'}}{\omega_{k}}  \left| \sum_{k'' \in \Gacts}  \talpha(k, k'')  \conv{(\sigma(m\e_0+\bb))^2}_{k'-k''} \right|  \nonumber \\
  & \leq  3
  \sum_{k \in \ZZ_0}   \frac{\omega_{k'}}{\omega_{k}}   \sum_{k'' \in \Gacts}  \left|  \conv{(\sigma(m\e_0+\bb))^2}_{k'-k''} \right| \nonumber \\ 
  & \leq 3 \, \| \conv{(\sigma(m\e_0+\bb))^2}\|_\nu, \label{e:Z1tailfactor2}
\end{alignat}
where the final inequality follows from Lemma~\ref{l:shiftestimate}.
To estimate the first factor in the righthand side of~\eqref{e:tailZ1term} we use Lemma~\ref{l:Pestimate}, which provides a uniform bound
\begin{equation}\label{e:Z1tailfactor1}
  \sup_{\substack{\knorm[k']>K\\k' \in \ZZ_0}} \frac{1}{|P(\Dk[k']\bkappa)|}  \leq \CC^{[0]}(K).
\end{equation}
Combining~\eqref{e:Z1tailfactor2} and~\eqref{e:Z1tailfactor1} we estimate~\eqref{e:supk3K} by
\[
   \VVV_3 \bydef 3 \,\CC^{[0]}(K)  \, \| \conv{(\sigma(m\e_0+\bb))^2}\|_\nu .
\]
Finally, collecting all terms, we set
\[
  \ZY{2}_2 = \roundup \max \bigl\{ \VVV_1, \VVV_2, \VVV_3 \bigr\} .
\]

%% file: boundW.tex
%!TEX root = ms.tex

\subsection{The bounds $\WY{i}_{i'i''}$}
\label{s:boundW}

To estimate $\AA [D\FF(\bx+w)- DF(\bx)]v$ for $w\in \BBsym_r(0)$ and $v\in \BB_{(1,1)}(0)$, we write $\bx=(\bkappa,\bb)$, $w=(r_1\mu,r_2 a)$ and $v=(\mu',a')$ with 
$|\mu|_{\RJ},|\mu'|_{\FJ} \leq1$ and $\|a\|_{\X_0},\|a'\|_{\X_0} \leq 1$.
We then split
\begin{alignat*}{1}
  \AA [D\FF(\bx+w)- DF(\bx)]v
  &= \AA [D\FF(\bkappa+r_1\mu,\bb+r_2 a) - D\FF(\bkappa+r_1\mu,\bb)] v\\
  &\qquad+ \AA [D\FF(\bkappa+r_1\mu,\bb) - D\FF(\bkappa,\bb)] v,
\end{alignat*}
and estimate both terms separately. 
In these estimates we will assume a priori bounds
\begin{equation}\label{e:checkrstar}
  r_1 \leq \rstar_1  
  \qquad\text{and}\qquad
  r_2 \leq \rstar_2 .
\end{equation}
In particular, under the assumptions~\eqref{e:checkrstar} we obtain estimates, uniform in $|\mu|_{\FJ},|\mu'|_{\FJ} \leq1$ and $\|a\|_{\X_0},\|a'\|_{\X_0}\leq 1$ of the form
\begin{alignat*}{1}
  \| \pi_\kappa \AA [D_\kappa \FF(\bkappa+r_1\mu,\bb) 
  						- D_\kappa\FF (\bkappa,\bb)] \mu' \|_{\FJ} 
  \hsp & \leq \WY{1}_{11} r_1 , \\
  \| \pi_b \AA [D_\kappa \FF(\bkappa+r_1\mu,\bb) 
  						- D_\kappa\FF (\bkappa,\bb)] \mu' \|_{\X_0}
  & \leq \WY{2}_{11} r_1 , \\
  \| \pi_\kappa \AA [D_b \FF(\bkappa+r_1\mu,\bb) 
  						- D_b\FF (\bkappa,\bb)] a' \|_{\FJ} 
  \hsp & \leq \WY{1}_{21} r_1 , \\
  \| \pi_b \AA [D_b \FF(\bkappa+r_1\mu,\bb) 
  						- D_b\FF (\bkappa,\bb)] a' \|_{\X_0}
  & \leq \WY{2}_{21} r_1 , \\
  \| \pi_\kappa \AA [D_\kappa \FF(\bkappa+r_1\mu,\bb+r_2 a) 
  						- D_\kappa\FF (\bkappa+r_1\mu,\bb)] \mu' \|_{\FJ} 
  \hsp & \leq \WY{1}_{12} r_2 , \\
  \| \pi_b \AA [D_\kappa \FF(\bkappa+r_1\mu,\bb+r_2 a) 
  						- D_\kappa\FF (\bkappa+r_1\mu,\bb)] \mu' \|_{\X_0}
  & \leq \WY{2}_{12} r_2 , \\
  \| \pi_\kappa \AA [D_b \FF(\bkappa+r_1\mu,\bb+r_2 a) 
  						- D_b\FF (\bkappa+r_1\mu,\bb)] a' \|_{\FJ} 
  \hsp & \leq \WY{1}_{22} r_2 , \\
  \| \pi_b \AA [D_b \FF(\bkappa+r_1\mu,\bb+r_2 a) 
  						- D_b\FF (\bkappa+r_1\mu,\bb)] a' \|_{\X_0} 
  & \leq \WY{2}_{22} r_2 .
\end{alignat*}

Before proceeding to the eight estimates, we define, for any $k \in \ZZ_0$,
\begin{alignat}{1}
 \PPP^{[1]}_k (\rstar_1) &\bydef  \max_{|\kappa-\bkappa|_{\RJ}\leq \rstar_1} |P'(\Delta_k \kappa)|,  \label{e:PPP1}\\
 \PPP^{[2]}_k (\rstar_1) &\bydef  \max_{|\kappa-\bkappa|_{\RJ}\leq \rstar_1} |P''(\Delta_k \kappa)|, \label{e:PPP2}\\
 \PPP^{[3]}_k (\rstar_1) &\bydef  \max_{|\kappa-\bkappa|_{\RJ}\leq \rstar_1} |P'''(\Delta_k \kappa)|. \label{e:PPP3}
\end{alignat}
In view of~\eqref{e:normFJ} the set $|\kappa-\bkappa|_{\RJ}\leq \rstar_1$ is described by the product of intervals 
\begin{equation}\label{e:productinterval}
  \II_{\bkappa}(\rstar_1) \bydef \prod_{j=1}^J	[\bkappa_j(1 - \rstar_1), \bkappa_j(1+ \rstar_1)].
\end{equation}
Hence, for any $k \in \ZZ^K_0$ the value of $\PPP^{[i]}_k(\rstar_1)$, $i=1,2,3$ can be enclosed explicitly via interval arithmetic.

%%%%%%%%%%%%%%%%% Q1 %%%%%%%%%%%%%%%%%

\subsubsection{The expressions for $\WY{i}_{11}$}
\label{s:Q1}

Using~\eqref{e:PPP3},
for any $r_1 \leq \rstar_1$ and any $\mu \in \RJ$ with $|\mu|_{\RJ}\leq 1$, the mean value theorem provides the estimate
\[
  |P''(\Delta_k (\bkappa+r_1\mu))-  P''(\Delta_k \bkappa)|
  \leq \PPP^{[3]}_k(\rstar_1) |\Delta_k|_{\FJ}^* r_1 ,
\]
with $\PPP^{[3]}_k$ given by~\eqref{e:PPP3}. 
Here we have used that, by definition of the dual norm,  and using that $|\mu|_{\FJ} \leq 1$, 
\[
  \biggl|\sum_{j'=1}^J \Delta^{j'}_k \mu_{j'} \biggr| \leq |\Delta_k|_{\RJ}^*
  \qquad \text{for } |\mu|_{\RJ} \leq 1.
\]
Based on~\eqref{e:DkappaH} we then estimate,
for any $| \mu' |_{\CJ} \leq 1$ and  $1 \leq j \leq J$,
\[
 r_1^{-1} \bigl|[D_\kappa H_j(\bkappa+r_1\mu,\bb) - D_\kappa H_j (\bkappa,\bb)] \mu' \bigr|
  \leq \RRR^{[1]}_j \bydef  \frac{1}{2} \sum_{k\in \ZZ^K_0} \PPP^{[3]}_k(\rstar_1) \, \Delta^j_k \, (\Dk \bkappa)^2 \, |\G.k| \,
  | \bb_k| \, |  \sigma(\bb)_{-k}|,
\]
where we have used that $|\Delta_k|_{\RJ}^*=|\Delta_k|_{\CJ}^* = \Dk \bkappa$ in view of~\eqref{e:choice}.

Similarly, we estimate, for $k \in \ZZ_0^K$, see~\eqref{e:DkappaF},
\[
  r_1^{-1} \bigl|[D_\kappa F_k(\bkappa+r_1\mu,\bb) - D_\kappa F_k(\bkappa,\bb)] \mu'\bigr| \leq \tRRR^{[1]}_k \bydef 
  \PPP^{[2]}_k(\rstar_1) \, (\Dk \bkappa)^2 \, |\Gacts| \, |\bb_k|.
\]
Finally, we take care of the premultiplication by $\AA$. We define 
$|A|$ by taking elementwise absolute value in the $(J+N)\times (J+N)$  matrix
and we write  $|A|_{ii'}$, $i,i'=1,2$ for the submatrices (Remark~\ref{r:matrixmatrix}).
We then set
\begin{alignat*}{1}
\WY{1}_{11} &= \roundup \bigl| |A|_{11} \RRR^{[1]} 
          + |A|_{12} \tRRR^{[1]} \bigr|_{\FJ}  ,\\
\WY{2}_{11} &= \roundup \bigl\| |A|_{21} \RRR^{[1]} 
          + |A|_{22} \tRRR^{[1]} \bigr\|_{\X_0}.
\end{alignat*}

%%%%%%%%%%%%%%%%% Q2 %%%%%%%%%%%%%%%%%

\subsubsection{The expressions for $\WY{i}_{21}$}
\label{s:Q2}

We set, for $1 \leq j \leq J$ and $k \in \ZZ_0^K$ 
\[
  \widehat{\RRR}^{[j]}_k \bydef  \PPP^{[2]}_k \, \Delta^j_k \, (\Dk \bkappa) \, |\sigma(\bb)_{-k}|.
\] 
Then, based on~\eqref{e:DbH} and again using the mean value theorem and~\eqref{e:choice}, we find
\[
  r_1^{-1} \bigl| [D_b H_j (\bkappa+r_1\mu,\bb) - D_b H_j (\bkappa,\bb)] a' \bigr|
\leq \RRR^{[2]}_j \bydef \|\widehat{\RRR}^{[j]}\|_{\X_0}^* ,
\]
where, see~\eqref{e:dualest},
\[
\|\widehat{\RRR}^{[j]}\|_{\X_0}^* = \max_{k \in \ZZ_0^K}
\omega_k^{-1} |\widehat{\RRR}^{[j]}_k  |  .
\]
By an analogous, but simpler, estimate we find from~\eqref{e:DbF} that
\begin{equation}\label{e:Q2Fk}
  \bigl| [D_{b_{k'}} F_k (\bkappa+r_1\mu,\bb) - D_{b_{k'}} F_k (\bkappa,\bb)] 
  \bigr| \leq \delta_{kk'} r_1  \PPP^{[1]}_k (\rstar_1) \, (\Dk \bkappa) \, |\Gacts| .
\end{equation}
Hence, for any $\|a'\|_{\X_0}$ and any $k \in \ZZ_0^K$,
\begin{equation*}
  r_1^{-1} \bigl| [D_b F_k (\bkappa+r_1\mu,\bb) - D_b F_k (\bkappa,\bb)] a' \bigr|
   \leq \tRRR^{[2]}_k 
   \bydef \PPP^{[1]}_k(\rstar_1) \, (\Dk \bkappa) \, 
   \nu^{-\knorm}.
   %\omega_k^{-1} .
\end{equation*}
The operator $\AA$ splits into $\AA = A \pi_K + \Lambda^{-1} \pi_\infty$.  
We estimate the finite part 
\[
 \pi_K \AA [D_b \FF(\bkappa+r_1\mu,\bb) -D_b\FF (\bkappa,\bb)] a'
 =
 A  \pi_K [D_b \FF(\bkappa+r_1\mu,\bb) -D_b\FF (\bkappa,\bb)] a'
\]
by using~$\tRRR^{[2]}$ and a  matrix multiplication, as in Section~\ref{s:Q1}.
In addition, we need to estimate the tail 
\[
 r_1^{-1} \pi_\infty \AA [D_b \FF(\bkappa+r_1\mu,\bb) -D_b\FF (\bkappa,\bb)] a'
 =
 r_1^{-1}  \Lambda^{-1} \pi_\infty [D_b \FF(\bkappa+r_1\mu,\bb) -D_b\FF (\bkappa,\bb)] a'
\]
in the $\X_0$-norm for any $\|a'\|_{\X_0}$.
In view of~\eqref{e:Q2Fk} we thus require a
uniform estimate on 
$ \frac{\PPP^{[1]}_k(\rstar_1) \, (\Dk \bkappa)}{P(\Delta_k \bkappa)}$
for all $\knorm > K $.
Such an estimate is provided by Lemma~\ref{l:Pestimate}
under the assumption that $K$ satisfies the restrictions~\eqref{e:Krestrictions}, which we will assume throughout the remainder of Section~\ref{s:boundW}.
Indeed, we then have
\[
   \frac{\PPP^{[1]}_k(\rstar_1) \, \Dk \bkappa}{P(\Delta_k \bkappa)}
    \leq \CC^{[1]}(K,r_1^*) ,
	\qquad\text{for all } \knorm > K ,
\]
with $\CC^{[1]}$ defined in~\eqref{e:CC1}.
We then set
\begin{alignat*}{1}
\WY{1}_{21} &= \roundup \bigl| |A|_{11} \RRR^{[2]} 
          +  |A|_{12}  \tRRR^{[2]} \bigr|_{\FJ}  ,\\
\WY{2}_{21} &= \roundup \left[\bigl\| |A|_{21} \RRR^{[2]}
          +  |A|_{22}  \tRRR^{[2]} \bigr\|_{\X_0}
		  + \CC^{[1]}(K,r_1^*) \right].
\end{alignat*}

\begin{remark}
It is relatively straightforward to obtain slightly sharper bounds by treating
the term $|A|_{22}  \tRRR^{[2]}$ and the tail term simultaneously through~\eqref{e:l1sup}, rather than estimating them separately. Other refinements or alternative approaches, for example treating the multiplication by $A$ with an operator norm estimate, are also possible. We have chosen the current bounds because they are fairly easy to write down and the resulting formulas in the Sections~\ref{s:Q1}--\ref{s:QQ2} are rather uniform
in appearance, 	hence relatively straightforward to parse and code, while still being reasonably sharp.
\end{remark}
	
%%%%%%%%%%%%%%%%% QQ1 %%%%%%%%%%%%%%%%%
						
\subsubsection{The expressions for $\WY{i}_{12}$}
\label{s:QQ1}

To abbreviate notation we write
$q^j_k = P''(\Delta_k (\bkappa+r_1\mu)) \, (\Delta_k \mu') \, \Delta^j_k$,
so that
\begin{alignat}{1}  
	r_2^{-1} [D_\kappa H_j(\bkappa+r_1\mu,\bb+r_2 a) 
    						- D_\kappa H_j (\bkappa+r_1\mu,\bb)] \mu'
	\nonumber \\
	 & \hspace{-5cm} =  \sum_{k\in \ZZ^K_0} q^j_k |\Gacts| \,\frac{1}{2} \bigl[
	  a_k \sigma(\bb)_{-k} + \bb_k \sigma(a)_{-k} \bigr] 
	 + \frac{1}{2} r_2 \sum_{k\in \ZZ_0} q^j_k |\Gacts| a_k \sigma(a)_{-k} .
	 \label{e:QQ2split}
\end{alignat}
By~\eqref{e:DkjCginvariance} it holds that
\[
  q^j_{\gacts} = q^j_k = q^j_{-k}.
\] 
By using these properties and Lemma~\ref{l:symquad} (twice) we rewrite
\begin{equation*}
  \sum_{k\in \ZZ^K_0} q^j_k |\Gacts| \bb_k \sigma(a)_{-k} 
	 = \sum_{k\in \Z^3_0} q^j_k  \sigma(\bb)_k \sigma(a)_{-k}  
	 = \sum_{k\in \Z^3_0} q^j_k  \sigma(\bb)_{-k} \sigma(a)_{k} 
	  =  \sum_{k\in \ZZ^K_0} q^j_k |\Gacts| a_k \sigma(\bb)_{-k}. 
\end{equation*}
Hence the first sum in~\eqref{e:QQ2split} reduces to
\[
  \sum_{k\in \ZZ^K_0} q^j_k |\Gacts|  a_k \sigma(\bb)_{-k} .
\]
Since $|a_k| \leq \omega_k^{-1} = |\G.k|^{-1} \nu^{-\knorm}$, this
is estimated by (again using $|\Dk|^*_{\CJ} = \Dk \bkappa$)
\[
  \RRR^{[3]}_j \bydef \max_{k \in  \ZZ^K_0} 
  \PPP^{[2]}_k(\rstar_1) \, (\Dk \bkappa) \, \Delta^j_k \,|\sigma(\bb)_{-k} | \,   \nu^{-\knorm}.
\]  
The second term in the righthand side of~\eqref{e:QQ2split} is estimated by (interpreting it as a linear operator in $a_k$ for fixed $\sigma(a)_{-k}$ and using~\eqref{e:dualest})
\begin{equation}\label{e:r2sup}
 \rstar_2 \max_{k \in  \ZZ_0} \frac{1}{2}  \PPP^{[2]}_k(\rstar_1) \, (\Dk \bkappa)  \, \Dkj{j} \, |\Gacts|\, \omega_k^{-2}.
\end{equation}
For $k\in \ZZ_0^K$ we can just evaluate the argument of in the supremum in~\eqref{e:r2sup}. For $\knorm >K$
we estimate $\Dkj{j} \leq \bkappa_j^{-1} \Dk \bkappa$ 
and $\omega_k =|\Gacts| \nu^{\knorm} \geq \nu^{\knorm}$.
The expression~\eqref{e:r2sup} is then bounded through Lemma~\ref{l:Eterm} by $\rstar_2 \mathcal{U}^{[2]}_j$, where the latter factor is given by the explicitly computable expression
\[
  \mathcal{U}^{[2]}_j \bydef \frac{1}{2}  \max\left\{ 
   \max_{k \in  \ZZ^K_0}  \PPP^{[2]}_k(\rstar_1) \, (\Dk \bkappa)\, \Delta^j_k \, |\Gacts| \, \omega_k^{-2} 
  \, , \,  \bkappa_j^{-1} \EE^{[2]} (K,\nu,r_1^*)
  \right\} ,
\]
where $\EE^{[2]}$ is defined in~\eqref{e:defE2j}.

Next we estimate
\[
 r_2^{-1} [D_\kappa F_k(\bkappa+r_1\mu,\bb+r_2 a) - D_\kappa F_k(\bkappa+r_1\mu,\bb)]  \mu'
= P'(\Delta_k (\bkappa+r_1\mu)) \, \Delta_k \mu' \, |\Gacts| \, a_k
\]
for any $k \in \ZZ_0^K$ by
\[
\tRRR^{[3]}_k \bydef 
  \PPP^{[1]}_k (\rstar_1) \, (\Delta_k \bkappa) \, 
  \nu^{-\knorm} .
\]
Premultiplication by $\AA$ them leads to an estimate which is very similar to the one in Section~\ref{s:Q2}. The required tail bound is again provided by Lemma~\ref{l:Pestimate}.
We then set
\begin{alignat*}{1}
\WY{1}_{12} &= \roundup \bigl| |A|_{11} (\RRR^{[3]} + \rstar_2\mathcal{U}^{[2]}) 
          + |A|_{12}  \tRRR^{[3]} \bigr|_{\FJ}  ,\\
\WY{2}_{12} &= \roundup \left[ \bigl\| |A|_{21} (\RRR^{[3]} +\rstar_2\mathcal{U}^{[2]}) 
          + |A|_{22}  \tRRR^{[3]} \bigr\|_{\X_0}+ \CC^{[1]}(K,r_1^*)
		  \right].
\end{alignat*}

%%%%%%%%%%%%%%%%% QQ2 %%%%%%%%%%%%%%%%%

\subsubsection{The expressions for $\WY{i}_{22}$}
\label{s:QQ2}

Analogously to the second term in the righthand side of~\eqref{e:QQ2split},
we estimate
\[
	r_2^{-1} \bigl| [D_b H_j(\bkappa+r_1\mu,\bb+r_2 a) 
    						- D_b H_j (\bkappa+r_1\mu,\bb)] a' \bigr|
	\leq \max_{k \in  \ZZ_0}  \PPP^{[1]}_k (\rstar_1) \, \Delta^j_k \, |\Gacts| \, \omega_k^{-2},
\]
cf.~\eqref{e:r2sup}.
Analogously to Section~\ref{s:QQ1}, this is bounded by 
\[
  \mathcal{U}^{[1]}_j \bydef  \max\left\{ 
   \max_{k \in  \ZZ^K_0}  \PPP^{[1]}_k (\rstar_1) \, \Delta^j_k \, |\Gacts| \,  \omega_k^{-2} 
  \, , \,   \bkappa_j^{-1} \EE^{[1]}(K,\nu)
  \right\} ,
\]
where $\EE^{[1]}$ is defined in~\eqref{e:defE1j}.

Finally, 
we note that
\[
[D_b F(\bkappa+r_1\mu,\bb+r_2 a) 
    						- D_b F (\bkappa+r_1\mu,\bb)] a' 
= \orbitcount
[D_b \Phi(\bkappa+r_1\mu,\bb+r_2 a) 
    						- D_b \Phi (\bkappa+r_1\mu,\bb)] a' .
\]
We see from~\eqref{e:defsigma} that $\talpha(k',k'') a'_{k'} = \sigma(a')_{k''
}$ for all $k' \in \ZZ_0$ and $k'' \in \Gacts[k']$. 
Hence it follows that for all $k \in \ZZ_0$
\begin{alignat}{1}
  r_2^{-1} [D_b \Phi_k (\bkappa+r_1\mu,\bb+r_2 a) 
    						- D_b \Phi_k (\bkappa+r_1\mu,\bb)] a' 
							\nonumber \\
	&\hspace{-3cm}= 3 \conv{ [(m\e_0+ \sigma(\bb) +r_2 \sigma(a))^2 -(m\e_0+ \sigma(\bb))^2 ]\sigma(a') }_k  .
			\label{e:oneconvolutionterm}
\end{alignat}
To manipulate the righthand side in~\eqref{e:oneconvolutionterm} we set
\[
  \xi_k = 3 \conv{ [(m\e_0+ \sigma(\bb) +r_2 \sigma(a))^2 -(m\e_0+ \sigma(\bb))^2 ]\sigma(a') }_k   \qquad \text{for all } k \in \Z^3.
\]
Since $\xi \in \Xsym$, we can
apply~\eqref{e:compatiblenorms} and use the Banach algebra property~\eqref{e:BA} and the triangle inequality to estimate, for $r \in (0,r_2)$
\begin{alignat}{1}
  \| \xi \|_{\X_0}
   & \leq 
  \| \xi \|_{\X}
  \label{e:restrictedk} \\ & =
  \|  \xi \|_{\nu}
  \nonumber \\ & \leq
  6 \left(\|m\e_0+ \sigma(\bb)\|_{\nu} +r_2\| \sigma(a) \|_{\nu} \right) \| \sigma(a) \|_{\nu} \| \sigma(a') \|_{\nu},
  \nonumber
\end{alignat}
where we simply use the restriction to Fourier indices $k\in \ZZ_0$ and $k \in \ZZ$ to interpret the norms in~\eqref{e:restrictedk}.
Hence we obtain, uniformly for $\|\sigma(a)\|_\nu= \|a\|_{\X_0} \leq 1$ and  $\|\sigma(a')\|_\nu= \|a'\|_{\X_0} \leq 1$,
\begin{equation}\label{e:productestimate}
  \| r_2^{-1} [D_b \Phi(\bkappa+r_1\mu,\bb+r_2 a) 
    						- D_b \Phi (\bkappa+r_1\mu,\bb)] a' \|_{\X_0} 
   \leq  6 (m + \| \bb \|_{\X_0} + \rstar_2 ),
\end{equation}
where we have used that $\|m\e_0+ \sigma(\bb)\|_{\nu} = m+\|\sigma(\bb))\|_{\nu} = m+ \|\bb\|_{\X_0}$ by~\eqref{e:compatiblenorms}.
We then use~\eqref{e:splitl1norm} to write
\[
   \| (A_{22}  \pi_K +   \Lambda^{-1} \pi_\infty ) \orbitcount \|_{B(\X_0,\X_0)} =
   \max \bigl\{ \|A_{22} \orbitcount \|_{B(\X_0,\X_0)} , \|\Lambda^{-1} \orbitcount \pi_\infty \|_{B(\X_0,\X_0)} \bigr\}.
\]  
To estimate the tail 
we use Lemma~\ref{l:Pestimate} to bound the operator norm 
$\| \Lambda^{-1} \orbitcount \pi_\infty \|_{B(\X_0,\X_0)} \leq \CC^{[0]}(K) $.
It then follows from~\eqref{e:productestimate} and the definition of the operator norm that we can set
\begin{alignat*}{1}
\WY{1}_{22} &= \roundup \left[ \bigl| |A|_{11} \mathcal{U}^{[1]} \bigr|_{\FJ} 
          + 6 \bigl( m + \| \bb \|_{\X_0} +  \rstar_2 \bigr) 
		  \|A_{12} \orbitcount\|_{B(\X_0,\FJ)}  \right]  ,\\
\WY{2}_{22} &= \roundup \left[ 
		\bigl\| |A|_{21}  \mathcal{U}^{[1]} \bigr\|_{\X_0}  
		  +  6 \bigl( m + \| \bb \|_{\X_0} +  \rstar_2 \bigr)
		  \max \bigl\{ \|A_{22} \orbitcount\|_{B(\X_0,\X_0)} , \CC^{[0]}(K)  \bigr\}  \right] .
\end{alignat*}